\documentclass[11pt, a4paper,reqno]{amsart}

\usepackage[T1]{fontenc}
\usepackage[utf8]{inputenc}
\usepackage{lmodern}
\usepackage{amsmath}
\usepackage{amssymb}
\usepackage{amsthm}
\usepackage{mathtools}
\usepackage{graphicx}
\usepackage[shortlabels]{enumitem}
\usepackage[british]{babel}
\usepackage{booktabs,array}
\usepackage{microtype}
\usepackage{xcolor}
\usepackage{tikz}
\usetikzlibrary{arrows.meta,decorations.pathreplacing}
\usepackage{hyperref}
\usepackage{dsfont}

\calclayout
\hypersetup{
 colorlinks=true,
 linkcolor=blue,
 citecolor=blue,
 urlcolor=blue
}

\allowdisplaybreaks[2]
\numberwithin{equation}{section}

\definecolor{hsInk}{gray}{0.00}
\definecolor{hsBlue}{gray}{0.00}
\definecolor{hsOrange}{gray}{0.35}
\definecolor{hsMagenta}{gray}{0.48}
\definecolor{hsGreen}{gray}{0.18}
\definecolor{hsGrey}{gray}{0.48}

\tikzset{
  hs axis/.style={
    draw=hsInk,
    line width=.65pt,
    -{Latex[length=2mm,width=1.3mm]}
  },
  hs guide/.style={
    draw=hsGrey,
    densely dashed,
    line width=.6pt
  },
  hs anchor/.style={
    draw=hsOrange!70!black,
    line width=.75pt,
    fill=hsOrange!35
  },
  hs error/.style={
    draw=hsMagenta!75!black,
    densely dashed,
    line width=.7pt,
    fill=hsMagenta!25
  },
  hs vector/.style={
    draw=hsInk,
    line width=.8pt,
    -{Stealth[length=2.2mm,width=1.4mm]}
  },
  hs note/.style={
    font=\footnotesize,
    text=hsInk,
    align=center
  },
  hs title/.style={
    font=\bfseries\small,
    text=hsInk,
    align=center
  }
}

\newtheorem{theorem}{Theorem}[section]
\newtheorem{proposition}[theorem]{Proposition}
\newtheorem{lemma}[theorem]{Lemma}
\newtheorem{corollary}[theorem]{Corollary}
\theoremstyle{definition}
\newtheorem{remark}[theorem]{Remark}

\makeatletter
\renewcommand{\l@section}{%
  \@tocline{1}{0pt}{1pc}{1.65pc}{}%
}
\makeatother

\newcommand{\norm}[1]{\left\|#1\right\|}
\newcommand{\R}{\mathbb R}

\renewcommand{\epsilon}{\varepsilon}
\newcommand{\Sph}{\mathbb S}
\newcommand{\dd}{\, \mathrm d}
\newcommand{\one}{\mathds 1}
\DeclareMathOperator{\supp}{supp}
\DeclareMathOperator{\sgn}{sgn}
\renewcommand{\L}{\operatorname{L}}
\newcommand{\C}{\operatorname{C}}
\renewcommand{\H}{\operatorname{H}}
\newcommand{\W}{\operatorname{W}}

\newcommand{\Sp}{\mathrm S}

\newcommand{\pair}[2]{\left\langle #1,#2\right\rangle}

\renewcommand{\Subset}{\subset\!\subset}

\allowdisplaybreaks

\title[Sharp kinetic trace theory]
{Sharp kinetic trace theory}

\author{Lukas Niebel}
\address[Lukas Niebel]{Institut f\"ur Analysis und Numerik, Universit\"at M\"unster\\
Orl\'eans-Ring 10, 48149 M\"unster, Germany.}
\email{lukas.niebel@uni-muenster.de}

\author{Lisa Valentini}
\address[Lisa Valentini]{
Department of Pure Mathematics and Mathematical Statistics
\\ University of Cambridge \\
Wilberforce Road, Cambridge CB3 0WA, United Kingdom}
\email{lv390@cam.ac.uk}

\date{27th July 2026}


\begin{document}

\begin{abstract}
  We establish sharp kinetic trace estimates and counterexamples across several velocity models. For half-space position domains, without any
  common bound on velocity support, we prove the natural trace estimate for both Lebesgue and standard Gaussian velocity measures. Density yields
  natural trace operators and
  Green's formula on the corresponding kinetic energy spaces. For bounded spatial domains in \(d\ge2\), in the bounded-support Euclidean velocity model
  and in the spherical velocity model, we identify the sharp boundary regularity threshold for the trace weights
  \(\min\{|v \cdot n|,|v \cdot n|^p\}\),
  \(1\le p<\infty\). Writing \(\alpha_p=1/(p+1)\), the estimate holds on every bounded \(\mathrm{C}^{1,\alpha}\) domain with
  \(\alpha\ge\alpha_p\), and
  it fails for every \(0<\alpha<\alpha_p\) on some strictly convex bounded domain of exact regularity \(\mathrm{C}^{1,\alpha}\). In particular,
  the
  natural trace (\(p=1\)) has the regularity threshold \(\mathrm{C}^{1,1/2}\). On bounded \(\mathrm{C}^{1,1/2}\) domains in \(d\ge2\), density
  yields natural
  trace operators and Green's formula in the bounded-support Euclidean velocity model with either Lebesgue or standard Gaussian measure, and in
  the
  spherical velocity model. Norm-preserving velocity translation rules out the unrestricted Lebesgue trace estimate on every bounded
  \(\mathrm{C}^1\)
  domain. In the unrestricted Gaussian model, for each \(1\le p<2\), we construct counterexamples on every bounded \(\mathrm{C}^{1,1}\) domain in
  dimension \(d\ge2\), answering Question~1.8 of Albritton, Armstrong, Mourrat, and Novack (2024) negatively. For \(2\le p<\infty\), the Gaussian
  \(\omega_2\) estimate and density instead yield \(\omega_p\)-trace operators.
\end{abstract}

\maketitle

\tableofcontents

\section{Introduction}

\subsection{The kinetic trace problem}

Let $\Omega\subseteq \R^d$ be the position space, with Lebesgue measure. Throughout, unless a stronger regularity hypothesis is stated, every
physical spatial domain with boundary is assumed to have at least a \(\C^1\) boundary. Let \(V \subseteq \R^d\) denote the velocity space and
\(\mu\) its measure. Boundary conditions---prescribed inflow, absorption, and diffuse or specular reflection---are imposed on the phase boundary
\[
  \Gamma=\partial\Omega\times V,
  \qquad
  \Gamma_\pm=\{(x,v)\in\Gamma:\ \pm v\cdot n_\Omega(x)>0\}
  ,
\]
where $n_\Omega$ is the outward unit normal vector field. For smooth functions with the necessary compact support, integration by parts for the
transport field gives
\begin{equation}\label{eq:intro-green}
  \int_{\Omega\times V}
  \bigl((v\cdot\nabla_xf)g+f(v\cdot\nabla_xg)\bigr)
  \dd\mu(v)\dd x
  =
  \int_{\Gamma}fg\,v\cdot n_\Omega\dd\mu(v)\dd S(x).
\end{equation}
Thus the natural boundary measure is \(|v\cdot n_\Omega|\dd S(x)\dd\mu(v)\). Its density vanishes on the grazing set
\[
  \Gamma_0=\{(x,v)\in\Gamma:\ v\cdot n_\Omega(x)=0\},
\]
where the characteristics are tangent to the boundary.

For kinetic Fokker--Planck and Kolmogorov equations, the natural kinetic Sobolev space at the weak/energy level does not control a full
derivative
in \(x\). Related
scale-invariant kinetic spaces and weak-solution theory are developed in \cite{AuscherImbertNiebel2024}, while homogeneous kinetic Sobolev spaces
adapted to the Kolmogorov equation are studied in \cite{AuscherNiebel2026}. In the stationary Gaussian setting of \cite{AAMN2024}, where $V=\R^d$
and $\mu= \gamma_d$ for
\[\dd\gamma_d(v)=(2\pi)^{-d/2}e^{-|v|^2/2}\dd v,\]
the relevant kinetic Sobolev space is
\begin{equation}\label{eq:intro-hkin}
  \H^1_{\mathrm{hyp}}(\Omega)
  :=
  \left\{
  f\in \L^2(\Omega;\H^1_\gamma(\R^d)):
  v\cdot\nabla_xf\in \L^2(\Omega;\H^{-1}_\gamma(\R^d))
  \right\},
\end{equation}
where
\[
  \norm{g}_{\H^1_\gamma}^2
  =\int_{\R^d}(|g|^2+|\nabla_vg|^2)\dd\gamma_d, \qquad \H^{-1}_\gamma=(\H^1_\gamma)'.
\]
In the notation introduced below, \(\H^1_{\mathrm{hyp}}(\Omega)=\mathcal X_\gamma(\Omega)\) with equal graph norms. For a bounded \(\C^1\) domain
$\Omega$, the trace question is whether
\begin{equation}\label{eq:intro-full-trace}
  \int_{\partial\Omega}\int_{\R^d}
  |f(x,v)|^2|v\cdot n_\Omega(x)|
  \dd\gamma_d(v)\dd S(x)
  \le C_\Omega\norm{f}_{\H^1_{\mathrm{hyp}}(\Omega)}^2
\end{equation}
holds for smooth \(f\). This is not covered by standard Sobolev trace theory. The formal multiplier \(\sgn(v\cdot n_\Omega)\), which would turn
the signed flux in \eqref{eq:intro-green} into its absolute value, is discontinuous in \(v\) at \(\Gamma_0\).
Multiplication by this sign does not preserve \(\H^1_\gamma\), so the
resulting product is not an admissible test function in Green's identity
\eqref{eq:intro-green} when the transport derivative is only
\(\H^{-1}_\gamma\)-valued.
When the normal
varies with \(x\), differentiating a regularised multiplier also couples the tangential velocity to the boundary geometry.

A bounded trace operator would extend Green's formula continuously to the kinetic energy space, provide well-defined inflow and outflow traces,
and justify the energy estimates used for uniqueness and stability.

\subsection{Classification}
\label{sec:classification}

We use the following terminology for the velocity models.
\begin{itemize}[leftmargin=*]
  \item In the \emph{unrestricted Lebesgue model}, \(V=\R^d\), \(\dd\mu=\dd v\), and no common support bound in the velocity variable is imposed.
  \item In the \emph{unrestricted Gaussian model}, \(V=\R^d\), \(\dd\mu=\dd\gamma_d\), and no common support bound in the velocity variable is
        imposed.
  \item In the \emph{bounded-support Euclidean model}, \(V=\R^d\) with either Lebesgue or Gaussian measure, but we restrict each estimate to
        functions satisfying \(\supp_v f\subset\overline B_L\) for a prescribed \(L<\infty\). Thus the velocity space remains \(\R^d\), and the
        constant in the trace estimate may depend on \(L\).
  \item In the \emph{spherical velocity model}, \(V=\Sph^{d-1}\) with surface measure $\dd\mu=\dd\sigma$, and $\nabla_v$ denotes the spherical
        gradient $\nabla_{\Sph^{d-1}}$.
\end{itemize}

We call the first two cases the \emph{unrestricted models} and the last two cases the \emph{bounded-velocity models}.

\medskip

All spaces and functions are real. For \((V,\mu)\) belonging to either unrestricted model or to the spherical velocity model, we write
\(\H^1_\mu=\H^1(V,\mu)\) for the inhomogeneous Sobolev space with norm
\[
  \norm{g}_{\H^1_\mu}^2
  :=\int_{V}(|g|^2+|\nabla_vg|^2)\dd\mu,
\]
and define \(\H^{-1}_\mu=(\H^1_\mu)'\) using the \(\L^2(\mu)\) pivot. Thus duality is bilinear and an \(\L^2(\mu)\) function \(q\) acts by
\(\pair{q}{\phi}=\int q\phi\dd\mu\). Set
\[
  \mathcal X_\mu(\Omega)
  :=\{f\in \L^2(\Omega;\H^1_\mu):
  v\cdot\nabla_xf\in \L^2(\Omega;\H^{-1}_\mu)\},
\]
where the transport derivative is distributional, and use the graph norm
\begin{equation}\label{eq:graph-norm}
  \norm{f}_{\mathcal X_\mu(\Omega)}^2
  :=\norm{f}_{\L^2(\Omega;\H^1_\mu)}^2
  +\norm{v\cdot\nabla_xf}_{\L^2(\Omega;\H^{-1}_\mu)}^2.
\end{equation}
The distributional transport operator is closed from \(\L^2_x\H^1_\mu\) to \(\L^2_x\H^{-1}_\mu\), so these graph spaces are complete. We call
\(\mathcal X_\mu(\Omega)\) the kinetic energy space; it is the natural graph space for weak solutions whose transport derivative is
\(\H^{-1}_\mu\)-valued.

For every domain and velocity model considered below, \(\mathcal X_\mu(\Omega)\) has a dense subspace of functions that are smooth up to the
physical boundary. Until the precise smooth classes are specified in Section~\ref{sec:energy}, we denote the relevant one by \(\mathcal D\).
Density is proved for the half-space in Proposition~\ref{prop:flat-intrinsic-density}; the bounded-domain statements are
Propositions~\ref{prop:bounded-gaussian-density}, \ref{prop:bounded-support-density}, \ref{prop:bounded-lebesgue-density},
and~\ref{prop:spherical-density}.

For \(1\le p<\infty\), define $\alpha_p:=\frac{1}{p+1}$ and
\[
  \omega_p \colon [0,\infty) \to [0,\infty), \qquad \omega_p(s):=\min\{s,s^p\}, \quad s\ge0.
\]
We also write
\[
  \mathcal E^\mu_\Omega(f):=\norm{f}_{\mathcal X_\mu(\Omega)}^2,
\]
and
\begin{equation}\label{eq:trace-functional}
  \mathcal T^\mu_{p,\Omega}(f)
  :=\int_{\partial\Omega}\int_V
  |f(x,v)|^2\omega_p(|v\cdot n_\Omega(x)|)\dd\mu(v)\dd S(x)
\end{equation}
for \(f\in\mathcal D\). We call \(\mathcal T^\mu_{p,\Omega}\) the weighted boundary functional.

We say that for some $1 \le p < \infty$ the \emph{\(\omega_p\)-trace estimate} holds if there exists $C_\Omega>0$ such that
\begin{equation}
  \label{eq:trace_ineq}
  \mathcal T^\mu_{p,\Omega}(f) \le C_\Omega \, \mathcal E^\mu_\Omega(f)
\end{equation}
for every \(f\in\mathcal D\). By the density statements mentioned above and Proposition~\ref{prop:trace-completion}, whenever the
\(\omega_p\)-trace estimate holds, classical boundary restriction on \(\mathcal D\) extends uniquely to a bounded linear operator
\[
  \operatorname{Tr}^{(p)}_\mu\colon
  \mathcal X_\mu(\Omega)\longrightarrow
  \L^2\!\left(\partial\Omega\times V,
  \omega_p(|v\cdot n_\Omega(x)|)\dd S(x)\dd\mu(v)\right).
\]
We call this the unique continuous \(\omega_p\)-trace operator associated with the stated dense core and weighted boundary space. Conversely,
failure of \eqref{eq:trace_ineq} on \(\mathcal D\) rules out a bounded operator between these spaces that agrees with classical boundary
restriction on \(\mathcal D\).

``Unrestricted velocities'' means that no common velocity-support bound is imposed. In the bounded-support Euclidean model, for either Lebesgue
or Gaussian measure and a prescribed velocity-support radius,
\[
  B_L:=\{v\in\R^d:|v|<L\},
\]
define
\begin{equation*}
  \mathcal X_{\mu,L}(\Omega)
  :=\left\{f\in\mathcal X_\mu(\Omega):
  \operatorname*{ess\,supp}_v f\subset\overline B_L\right\},
\end{equation*}
endowed with the graph norm \(\norm{\cdot}_{\mathcal X_\mu(\Omega)}\). Essential velocity support is preserved under \(\L^2\)-convergence, so
\(\mathcal X_{\mu,L}(\Omega)\) is closed in \(\mathcal X_\mu(\Omega)\) and is therefore complete.

For every fixed \(L>0\), \(\mathcal X_{\mu,L}(\Omega)\) has a dense subspace of smooth functions whose velocity support is compactly contained in
\(B_L\). When working in \(\mathcal X_{\mu,L}(\Omega)\), we again denote this subspace by \(\mathcal D\) until its precise definition in
Section~\ref{sec:energy}. Its density is proved in Proposition~\ref{prop:bounded-support-density}.

In this context, we say that for some \(1\le p<\infty\) the \emph{\(\omega_p\)-trace estimate} holds if there exists \(C_{\Omega,L}>0\) such that
\[
  \mathcal T^\mu_{p,\Omega}(f)
  \le C_{\Omega,L}\,\mathcal E^\mu_\Omega(f)
\]
for every \(f\in\mathcal D\). The implicit constant may depend on \(L\). By density and Proposition~\ref{prop:trace-completion}, whenever this
estimate holds, classical boundary restriction on \(\mathcal D\) extends uniquely to a bounded linear operator
\[
  \operatorname{Tr}^{(p)}_{\mu,L}\colon
  \mathcal X_{\mu,L}(\Omega)\longrightarrow
  \L^2\!\left(\partial\Omega\times\R^d,
  \omega_p(|v\cdot n_\Omega(x)|)\dd S(x)\dd\mu(v)\right).
\]
This unique continuous extension is the corresponding \(\omega_p\)-trace operator on \(\mathcal X_{\mu,L}(\Omega)\). Conversely, failure of this
estimate on \(\mathcal D\) rules out a bounded operator between these spaces that agrees with classical boundary restriction on \(\mathcal D\).

The \emph{natural trace} corresponds to \(p=1\). For \(p>1\), the weight is weaker near grazing velocities, since \(\omega_p(s)=s^p<s\) for
\(0<s<1\). Thus the corresponding estimates need not imply a natural trace or Green's formula. All weighted boundary functionals below are taken
over the full phase boundary \(\Gamma\).

Smooth functions on \(\overline\Omega\) are restrictions of ambient smooth functions. Unless indicated otherwise, implicit constants may depend
on fixed profiles and cut-offs.

\begin{theorem}
  \label{thm:intro-classification}
  Let \(1\le p<\infty\).
  \begin{enumerate}[label=\textup{(\roman*)},leftmargin=*]
    \item On a half-space, the \(\omega_p\)-trace estimate holds for both the unrestricted Lebesgue and unrestricted Gaussian models in every
          dimension \(d\ge1\).
    \item On every bounded \(\C^1\) domain, the unrestricted Lebesgue \(\omega_p\)-trace estimate fails in every
          dimension \(d\ge1\).
    \item If \(d\ge2\), then on every bounded \(\C^{1,1}\) domain the unrestricted Gaussian \(\omega_p\)-trace estimate fails for \(1\le p<2\)
          and holds for \(2\le p<\infty\).
    \item If \(d\ge2\), then in the bounded-support Euclidean model with prescribed ball \(\overline B_L\), the \(\omega_p\)-trace estimate holds
          on every bounded \(\C^{1,\alpha}\) domain with \(\alpha\ge\alpha_p\); the constant may depend on \(L\). For each \(0<\alpha<\alpha_p\),
          it fails on a strictly convex bounded domain of exact regularity \(\C^{1,\alpha}\).
    \item If \(d\ge2\), then in the spherical velocity model the \(\omega_p\)-trace estimate holds on every bounded \(\C^{1,\alpha}\) domain with
          \(\alpha\ge\alpha_p\). For each \(0<\alpha<\alpha_p\), it fails on a strictly convex bounded domain of exact regularity
          \(\C^{1,\alpha}\).
    \item For $d=1$, the unrestricted Gaussian \(\omega_p\)-trace estimate and the Lebesgue \(\omega_p\)-trace estimate in the bounded-support
          Euclidean model hold on every bounded interval. The unrestricted Lebesgue estimate fails. For \(\Sph^0=\{-1,+1\}\), the result reduces
          to the ordinary Sobolev trace theorem.
  \end{enumerate}
  For \(d\ge2\), \(\alpha_p\) is therefore the sharp boundary regularity threshold for the class of bounded \(\C^{1,\alpha}\) domains in the
  bounded-velocity models. Failure below the threshold is existential, not universal---see Table~\ref{tab:main-results}.
\end{theorem}

\begin{table}[htbp]
  \centering
  \small
  \caption{Summary of the kinetic trace estimates.
    \textsc{Universal} means that
    the estimate
    fails on every domain in the stated class;
    \textsc{existential} means that, for each indicated \(\alpha\), at least
    one counterexample domain exists.}
  \label{tab:main-results}
  \begin{tabular}{@{}
    >{\raggedright\arraybackslash}p{0.16\textwidth}
    >{\raggedright\arraybackslash}p{0.12\textwidth}
    >{\raggedright\arraybackslash}p{0.27\textwidth}
    >{\raggedright\arraybackslash}p{0.22\textwidth}
    >{\raggedright\arraybackslash}p{0.13\textwidth}@{}}
    \toprule
    Velocity model
     & Weight
     & Spatial assumption
     & Conclusion
     & Reference                                                                                       \\
    \midrule
    Unrestricted Lebesgue or Gaussian
     & all \(p\)
     & \(d\ge1\); half-space
     & Holds; natural for \(p=1\)
     & Thm.~\ref{thm:flat}, Prop.~\ref{cor:flat-p}                                                     \\
    \midrule
    Unrestricted Lebesgue
     & all \(p\)
     & \(d\ge1\); bounded \(\C^1\) domains
     & Fails (\textsc{Universal})
     & Prop.~\ref{prop:interval-lebesgue-failure}                                                      \\
    \midrule
    Unrestricted Gaussian
     & all \(p\)
     & \(d=1\); bounded intervals
     & Holds; natural for \(p=1\)
     & Cor.~\ref{cor:interval-intrinsic}                                                               \\
    \midrule
    Unrestricted Gaussian
     & \(1\le p<2\)
     & \(d\ge2\); bounded \(\C^{1,1}\) domains
     & Fails (\textsc{Universal})
     & Thm.~\ref{thm:unbounded}                                                                        \\
    \midrule
    Unrestricted Gaussian
     & \(2\le p<\infty\)
     & \(d\ge2\); bounded \(\C^{1,1}\) domains
     & Holds
     & Prop.~\ref{prop:quadratic-gaussian},
    Prop.~\ref{cor:unrestricted-gaussian-p-trace}                                                      \\
    \midrule
    Bounded-support Euclidean,
    \(\mu=\mathcal L^d\) or \(\gamma_d\)
     & all \(p\)
     & \(d\ge2\); \(\alpha\ge\alpha_p\), bounded \(\C^{1,\alpha}\) domains
     & Holds; natural for \(p=1\), $\alpha \ge 1/2$
     & Thm.~\ref{thm:bounded-positive},
    Prop.~\ref{cor:intrinsic-p-trace}                                                                  \\
    \midrule
    Bounded-support Euclidean,
    \(\mu=\mathcal L^d\) or \(\gamma_d\)
     & all \(p\)
     & \(d\ge2\); \(0<\alpha<\alpha_p\), strictly convex domains of exact regularity \(\C^{1,\alpha}\)
     & Fails (\textsc{Existential})
     & Thm.~\ref{thm:bounded-euclidean}                                                                \\
    \midrule
    Spherical
     & all \(p\)
     & \(d\ge2\); \(\alpha\ge\alpha_p\), bounded \(\C^{1,\alpha}\) domains
     & Holds; natural for \(p=1\), $\alpha \ge 1/2$
     & Cor.~\ref{cor:spherical-positive},
    Prop.~\ref{cor:intrinsic-p-trace}                                                                  \\
    \midrule
    Spherical
     & all \(p\)
     & \(d\ge2\); \(0<\alpha<\alpha_p\), strictly convex domains of exact regularity \(\C^{1,\alpha}\)
     & Fails (\textsc{Existential})
     & Thm.~\ref{thm:spherical}                                                                        \\
    \midrule
    Bounded-support Euclidean,
    \(\mu=\mathcal L^1\) or \(\gamma_1\);
    spherical, \(\Sph^0=\{-1,+1\}\)
     & all \(p\)
     & \(d=1\); bounded intervals
     & Holds
     & Cor.~\ref{cor:interval-intrinsic},
    Sec.~\ref{sec:dimension-one}                                                                       \\
    \bottomrule
  \end{tabular}
\end{table}

\begin{remark}[Relation to Question~1.8]
  The natural kinetic trace problem was explicitly posed as open by Armstrong and Mourrat \cite[Question~4.2]{ArmstrongMourrat2019} and was
  retained, in the formulation relevant here, as \cite[Question~1.8]{AAMN2024}. The problem, or closely related local and model-specific
  formulations of it, was also recorded as open by Garain and Nystr\"om \cite[Section~1.6 and Problem~4]{GarainNystrom2023}, Silvestre
  \cite[Remark~4.4]{Silvestre2022}, Avelin and Hou \cite[Section~1.1]{AvelinHou2026}, Hou \cite[Introduction, p.~2]{Hou2026}, the second author
  \cite[Section~2.3.1 and equation~(2.14)]{Valentini2026}, and Kim and Weidner \cite[Section~1.3.1, p.~6]{KimWeidnerMaxwell2026}.

  In the unrestricted Gaussian model, Theorem~\ref{thm:flat} gives a positive answer to the half-space analogue of \cite[Question~1.8]{AAMN2024}.
  In contrast, the answer to the general bounded-domain question is negative in dimensions \(d\ge2\): Theorem~\ref{thm:unbounded} shows that the
  natural trace estimate fails on every bounded \(\C^{1,1}\) domain.
\end{remark}

\begin{remark}[Optimality of Silvestre's grazing exponent]
  \label{rem:silvestre-sharp}
  On \(\C^{1,1}\) domains, Silvestre's local, time-dependent Lebesgue-velocity trace estimate uses the quadratic grazing weight
  \(\omega_2(s)=\min\{s,s^2\}\), and the optimality of this weight was left open \cite[Proposition~4.3 and Remark~4.4]{Silvestre2022}. In the
  stationary unrestricted Gaussian model considered here, Proposition~\ref{prop:quadratic-gaussian} proves the corresponding \(\omega_2\)-trace
  estimate, whereas Theorem~\ref{thm:unbounded} rules out every \(\omega_p\)-trace estimate with \(1\le p<2\) on every bounded \(\C^{1,1}\)
  domain in dimensions \(d\ge2\). Since \(\omega_p\le\omega_2\) for \(p\ge2\), the estimate also holds for every finite \(p\ge2\). Thus \(p=2\)
  is the smallest admissible exponent in the scale \(\omega_p\). In this sense, the quadratic grazing exponent is optimal for the stationary
  unrestricted Gaussian analogue of Silvestre's estimate.
\end{remark}

\begin{remark}[Time-dependent analogue]
  \label{rem:time-dependent}
  At the level of smooth trace estimates, the preceding classification has a direct time-dependent analogue on every fixed cylinder
  \(I\times\Omega\), with \(v\cdot\nabla_x\) replaced by
  \[
    \partial_t+v\cdot\nabla_x.
  \]
  In every case above where the estimate holds, functions in the relevant smooth class that are compactly supported in time satisfy
  \[
    \int_I \mathcal T^\mu_{p,\Omega}(f(t))\dd t
    \le C\left(
    \norm{f}_{\L^2(I\times\Omega;\H^1_\mu)}^2
    +\norm{(\partial_t+v\cdot\nabla_x)f}
    _{\L^2(I\times\Omega;\H^{-1}_\mu)}^2
    \right),
  \]
  with the same spatial assumptions, grazing weights, and velocity-support conditions. Indeed, the multiplier proofs may be rerun using
  space--time Green identities. All multipliers are independent of \(t\), so that
  \[
    (\partial_t+v\cdot\nabla_x)m=v\cdot\nabla_xm.
  \]
  In the spherical velocity model, the bounded-support Euclidean multiplier argument is applied to the lifted operator
  \(r\partial_t+v\cdot\nabla_x\), where \(v=r\theta\).

  For an arbitrary function in the relevant smooth class, smooth up to the endpoints of a finite interval \(I=(t_0,t_1)\), the corresponding
  global estimate is
  \begin{align*}
     & \int_I \mathcal T^\mu_{p,\Omega}(f(t))\dd t \\
     & \le C\left(
    \norm{f}_{\L^2(I\times\Omega;\H^1_\mu)}^2
    +\norm{(\partial_t+v\cdot\nabla_x)f}
    _{\L^2(I\times\Omega;\H^{-1}_\mu)}^2  \right.  \\
     & \quad+ \left. \norm{f(t_0)}
      _{\L^2(\Omega\times V;\dd x\dd\mu(v))}^2
    +\norm{f(t_1)}
      _{\L^2(\Omega\times V;\dd x\dd\mu(v))}^2
    \right).
  \end{align*}
  The signed temporal boundary contributions vanish for functions compactly supported in time and cancel for time-periodic functions. All
  counterexamples transfer by taking \(F(t,x,v)=\eta(t)f(x,v)\) with a fixed nonzero \(\eta\in\C_c^\infty(I)\). Consequently, the sharp exponents
  and boundary regularity thresholds in Theorem~\ref{thm:intro-classification} are unchanged in the time-dependent setting.

  In the regimes where the natural trace exists, the smooth space--time Green's formula contains, on the same side as the spatial boundary term,
  the additional temporal boundary term
  \[
    (f(t_1),g(t_1))_{\L^2(\Omega\times V;\,\dd x\dd\mu(v))}
    -(f(t_0),g(t_0))_{\L^2(\Omega\times V;\,\dd x\dd\mu(v))}.
  \]
\end{remark}

\bigskip
Let us now explain the ideas of the proofs.

\subsection{The half-space proof}
On a half-space the geometry reduces to the normal variables. Write \(x=(z,y)\) and \(v=(u,w)\), where \(y>0\) is the distance to the boundary
and \(w\) is the fixed signed normal-velocity component. The starting point is Green's identity on the truncated half-space \(\mathbb
H^d_\delta=\{y>\delta\}\). For a multiplier \(m=m(y,w)\),
\begin{align*}
  -\int_{\R^{d-1}}\int_{\R^d}
  w\,m(\delta,w)|f(z,\delta,v)|^2\dd\mu(v)\dd z
   & =
  \int_{\mathbb H^d_\delta}\int_{\R^d}
  (v\cdot\nabla_xm)|f|^2\dd\mu(v)\dd x \\
   & \quad+
  2\int_{\mathbb H^d_\delta}
  \pair{v\cdot\nabla_xf}{mf}_{\H^{-1}_\mu,\H^1_\mu}\dd x.
\end{align*}
The boundary term suggests choosing \(m\) so that \(m(\delta,w)\to\sgn w\) as \(\delta\downarrow0\), since then \(w\,m(\delta,w)\to|w|\), the
natural boundary weight.

The exact choice \(m=\sgn w\) cannot be used. Its jump at the grazing velocity \(w=0\) means that \((\sgn w)f\) need not belong to
\(\H^1_\mu\), so the last term in Green's identity is not defined for an \(\H^{-1}_\mu\)-valued transport derivative. We therefore seek a
smooth multiplier which agrees with the sign away from \(w=0\), while its transition region collapses as the boundary is approached.

Let \(a(y)\) denote the width of this transition. A very narrow transition is expensive when the multiplier is differentiated in velocity,
whereas a wide transition is expensive when it is differentiated along transport. Balancing these effects leads to the cubic grazing scale
\begin{equation}\label{eq:cubic_scale}
  a(y)=y^{1/3}.
\end{equation}
Accordingly, we introduce \(r=wy^{-1/3}\) and take
\[
  m(y,w)=F_\varepsilon(r),
\]
where \(F_\varepsilon\) is a smooth odd profile equal to \(\sgn r\) outside a bounded interval. This cubic relation also appears in kinetic
boundary analysis \cite{HwangJangVelazquez2014,HendersonLucertiniWang2024,RosOtonWeidner2025,KimWeidner2026}.

The profile is allowed to overshoot \(\pm1\) before returning to them. This
return produces a favourable transport contribution slightly away from
\(w=0\). The contribution near \(w=0\) has the opposite sign, but it is
weakened by the small normal-velocity component. A one-dimensional
Poincar\'e estimate then transfers the control from the favourable region
across the whole transition window.
Combined with Green's identity, this gives the key kinetic grazing-cone
Hardy inequality, which controls the \(y^{-2/3}\)-weighted integral of
\(|f|^2\) over the cubic grazing cone \(\{|w|<2y^{1/3}\}\) by the kinetic
energy norm.

With this Hardy inequality in hand, we return to the boundary term in
Green's identity and let \(\delta\downarrow0\). The multiplier converges to
\(\sgn w\), and the boundary factor converges to \(|w|\), giving the natural
trace estimate.
Since the construction depends only on \((y,w)\), no restriction on the tangential velocity is needed, and the argument
applies to both the unrestricted Lebesgue and unrestricted Gaussian models.

\subsection{Three sources of failure}
\label{sec:counterexamples}

To disprove the trace estimate \eqref{eq:trace_ineq}, it is enough to find smooth functions whose bulk energies stay bounded while their
boundary functionals tend to infinity. All counterexamples below follow this plan. We call a smooth test function concentrated near chosen
positions and velocities a \emph{packet}.

\begin{enumerate}[label=\textup{\Alph*.},leftmargin=*]
  \item \emph{Unrestricted Lebesgue model: norm-preserving velocity translations.} Lebesgue measure is unchanged by translations in \(v\).
        Moving a velocity packet from the origin to a large velocity \(Re\) therefore leaves its velocity \(\H^1\)-norm unchanged. On a bounded
        domain, we may also take the packet independent of \(x\), so its transport derivative vanishes and its bulk energy is independent of
        \(R\). On a boundary patch where \(e\cdot n_\Omega\) is bounded below, \(|v\cdot n_\Omega|\) is of order \(R\). Since \(\omega_p(s)=s\)
        for \(s\ge1\), the weighted boundary functional grows linearly in \(R\) for every \(p\ge1\).

  \item \emph{Unrestricted Gaussian model: curvature-induced normal-flux amplification.} Gaussian measure makes a shift to large velocity
        expensive, so the preceding construction no longer works. Instead, we centre a packet at a large velocity that is tangent to the
        boundary at one point. The boundary normal changes across a nearby patch, so the same velocity acquires a normal component there. The
        packet widths and spatial support can be balanced so that the bulk energy stays bounded while
        \[
          \mathcal T^\gamma_{p,\Omega}(f_h)
          \asymp h^{(p-2)/6},
        \]
        where \(h\downarrow0\) is the thickness of the spatial layer. The weighted boundary functional therefore diverges for \(p<2\). This
        mechanism requires \(d\ge2\), and the needed supporting patch exists on every bounded \(\C^{1,1}\) domain; no positive lower bound on
        curvature is assumed.

  \item \emph{Bounded-velocity models: rapid normal variation.} Here the packet cannot be sent to arbitrarily large velocity. We instead choose
        a domain of exact regularity \(\C^{1,\alpha}\) whose normal changes by an amount of order \(\ell^\alpha\) across a boundary patch of size
        \(\ell\).
        A velocity that is nearly tangent at one point then has a normal component of order \(\ell^\alpha\) at a nearby point, without any
        large-velocity effect. The packet can be arranged to follow transport characteristics and to have bounded bulk energy, while
        \[
          \mathcal T^\mu_{p,\Omega}(f_\ell)
          \asymp \ell^{\alpha(p+1)-1}.
        \]
        This diverges when \(\alpha<1/(p+1)=\alpha_p\). The construction applies to both the bounded-support Euclidean and spherical velocity
        models. It gives a counterexample domain for each \(\alpha<\alpha_p\), not failure on every \(\C^{1,\alpha}\) domain.
\end{enumerate}

The first two mechanisms require velocities whose size tends to infinity and are therefore blocked in the bounded-velocity models. The last two
require a changing boundary normal and disappear on a half-space. The first mechanism also uses the boundedness of \(\Omega\): an
\(x\)-independent packet has infinite bulk energy on a half-space, while localising it in \(x\) creates a transport cost that grows with \(R\).

At the critical exponent \(\alpha=\alpha_p\), the third family has a bounded, non-vanishing weighted boundary functional rather than a
divergent one, so it is no longer a counterexample.

\subsection{A scaling heuristic for the boundary threshold}

Assume \(d\ge2\). We briefly explain the origin of the exponent \(\alpha_p=1/(p+1)\) and the geometry behind the bounded-velocity
counterexamples. The discussion is only heuristic.

The characteristics of the free transport operator \(v\cdot\nabla_x\) are the straight lines \(x(s)=x_0+sv\), along which \(v\) is constant. Such
a characteristic is grazing at \(x_0\in\partial\Omega\) when \(v\cdot n_\Omega(x_0)=0\). This is the limiting picture. At a finite tangential
scale, the counterexample packet is centred on a nearby, nearly tangent secant characteristic. At the reference point \(x_0\), its normal
velocities satisfy
\[
  |v\cdot n_\Omega(x_0)|\sim a\ll1,
\]
and their scalar normal components occupy a range of width comparable to \(a\). Thus \(a\) measures both the distance of the packet from the
grazing set and its thickness in the normal-velocity direction.

Let \(h\) be the thickness of the spatial boundary layer and let \(\ell\) be the tangential size of the packet. We take the tangential part of
the central velocity to have fixed nonzero size. A characteristic in the packet then takes time of order \(h/a\) to cross the layer and travels a
comparable tangential distance. Hence
\begin{equation}\label{eq:intro-characteristic-scaling}
  \ell\sim\frac{h}{a},
  \qquad\text{or equivalently}\qquad
  h\sim a\ell.
\end{equation}

Suppose that the packet has amplitude \(A\), while the remaining velocity directions are kept at fixed scale. We arrange it to be approximately
constant along transport characteristics, so that the transport term does not determine the scaling. In the later construction, a cut-off in a
phase that is exactly constant along transport, together with the convex boundary, also confines the tangential support without creating an
additional transport cost. Cutting the packet off across the normal-velocity interval costs a velocity derivative of size \(A/a\). Its dominant
bulk-energy contribution therefore has scale
\[
  \mathcal E^\mu_\Omega(f)
  \sim
  A^2\ell^{d-1}\frac{h}{a}.
\]
Using \(h\sim a\ell\), this becomes
\[
  \mathcal E^\mu_\Omega(f)\sim A^2\ell^d.
\]
We may therefore choose \(A^2\sim\ell^{-d}\) so that the bulk energy remains of order one.

The boundary geometry determines the remaining scale \(a\). If \(\partial\Omega\) is \(\C^{1,\alpha}\), then its unit normal is
\(\alpha\)-H\"older continuous, and across a boundary patch of tangential size \(\ell\),
\begin{equation}\label{eq:heuristic_ub}
  |n_\Omega(x)-n_\Omega(x_0)|
  \lesssim \ell^\alpha.
\end{equation}
This is only an upper bound on a general domain. The counterexamples are constructed on strictly convex domains of exact regularity
\(\C^{1,\alpha}\), for which the H\"older estimate \eqref{eq:heuristic_ub} is attained along suitable boundary patches. A suitable velocity
tangent to the boundary at one point therefore acquires, relative to the changing boundary normal, a normal component of order \(\ell^\alpha\) at
nearby points, providing the geometric mechanism underlying the sharpness result.

To turn this tangent picture into a packet of characteristics that crosses the thin boundary layer, we use nearby, nearly tangent secants. On the
selected patches, the gap between the secant direction and the local tangent has the same order \(\ell^\alpha\). We therefore choose the packet's
normal-velocity scale to match this geometric scale:
\[
  a\sim\ell^\alpha,
\]
hence $h\sim a\ell\sim\ell^{1+\alpha}$.

On the boundary, the packet occupies an area of size \(\ell^{d-1}\) and a normal-velocity interval of width \(a\). Throughout its boundary
support, \(|v\cdot n_\Omega|\lesssim a<1\), and it is comparable to \(a\) on a fixed positive fraction of the rescaled packet. Thus the upper and
lower bounds have the same scale: the boundary weight contributes \(a^p\), and the energy normalisation above gives
\[
  \mathcal T^\mu_{p,\Omega}(f)
  \sim
  A^2\ell^{d-1}a^{p+1}
  \sim
  \frac{a^{p+1}}{\ell}
  \sim
  \ell^{\alpha(p+1)-1}.
\]
For this packet family, the boundary functional therefore diverges when \(\alpha<1/(p+1)\), remains of order one when \(\alpha=1/(p+1)\), and
decays when \(\alpha>1/(p+1)\). This predicts the critical boundary exponent
\[
  \alpha_p=\frac1{p+1}.
\]
The failure below the threshold is existential: it occurs on the model domains just described, not necessarily on every \(\C^{1,\alpha}\) domain.
These same scales also guide the positive argument, but curvature makes their analytic use more delicate.

\subsection{How the scaling enters the curved-boundary proof}

On a half-space the normal direction is fixed. The signed normal velocity is simply \(w\), the grazing set is always \(w=0\), and the regularised
multiplier used in the proof depends only on the height \(y\) and on \(w\). On a curved boundary, by contrast, the normal changes from point to
point, so the set of grazing velocities moves along the boundary.

To see the resulting difficulty, rotate coordinates so that locally \(\Omega=\{x_d>\Phi(z)\}\), set \(y=x_d-\Phi(z)>0\) and \(b=\nabla\Phi\), and
split the velocity as \(v=(u,w)\). After flattening the graph, the coefficient of \(\partial_y\) and the transport field become
\[
  c(z,v)=w-u\cdot b(z),
  \qquad
  X_b=u\cdot\nabla_z+c(z,v)\partial_y.
\]
On this boundary graph, \(v\cdot n_\Omega\dd S=-c\dd z\), so \(c\) is the unnormalised flattened normal velocity entering the physical boundary
weight. It is tempting to insert \(c\) directly into the flat multiplier. However, differentiating such a multiplier along \(X_b\) would also
differentiate \(b\). At the critical regularity, \(b\) is only H\"older continuous, so these derivatives are not available. This is the main new
obstruction in the curved case.

The first remedy is to smooth the boundary slope in the tangential variables. At height \(y\), let \(b_y\) be a tangential mollification of
\(b\), and introduce the approximate normal velocity
\[
  \bar c=w-u\cdot b_y(z).
\]
The smoothing scale must shrink as the boundary is approached. Too little smoothing makes derivatives of \(b_y\) too large, whereas too much
smoothing makes \(\bar c\) a poor approximation of the true normal velocity \(c\).

The scaling from the preceding subsection prescribes the balance. At \(\alpha=\alpha_p\), the relation \(a\sim\ell^{\alpha_p}\) gives \(\ell\sim
a^{p+1}\), while \(h\sim a\ell\) gives \(h\sim a^{p+2}\). Identifying the physical height \(h\) with \(y\) and the grazing width \(a\) with
\(\rho(y)\), and denoting the tangential smoothing scale by \(\ell(y)\), gives
\[
  \rho(y)\sim y^{1/(p+2)},
  \qquad
  \ell(y)\sim y^{(p+1)/(p+2)}.
\]
Thus
\[
  \rho(y)\ell(y)\sim y,
  \qquad
  \ell(y)^{\alpha_p}\sim\rho(y),
  \qquad
  \alpha_p=\frac1{p+1}.
\]
The first relation is the characteristic balance \(h\sim a\ell\), while the second says that, at the critical regularity, the mollification error
and the grazing width have the same power of \(y\). This comparability alone does not provide smallness. The constants in the two scale choices
are selected so that the approximation error and the terms produced when \(b_y\) is differentiated can be absorbed; here the boundedness of the
velocities is essential.

A first multiplier is then built from the normalised variable \(\bar c/\rho(y)\). It is designed to give bulk control near the boundary rather
than to produce the boundary trace directly. The favourable transport contribution has the critical scale
\[
  \frac{\rho(y)^p}{y}
  \sim y^{-2/(p+2)}.
\]
The squared velocity-gradient cost is bounded by the same Hardy weight. A suitable one-dimensional multiplier profile turns this balance into a
Hardy estimate controlling the bulk mass with weight \(y^{-2/(p+2)}\). The boundary contribution of this first multiplier vanishes, so a separate
argument is required to recover the desired boundary weight.

For that purpose, the proof uses a second, dyadic multiplier. Formally, one would like to use
\[
  -\operatorname{sgn}(c)|c|^{p-1},
\]
whose boundary flux is \(|c|^p\). Depending on \(p\), this expression may fail to be smooth at \(c=0\), and in every case it still contains the
rough slope \(b\). The proof therefore decomposes the normal-velocity region into dyadic bands. On each band, it mollifies the slope at a
band-dependent tangential scale and restricts the multiplier to a matching height layer. These choices turn the formal expression into an
admissible multiplier whose interior costs are controlled by the Hardy estimate. Uniform finite overlap keeps the multiplier bounded, including
when \(p=1\).

Green's identity then yields the local boundary estimate. Rotations, localisation, and a finite collection of boundary charts globalise it on the
domain. This proves the bounded-support Euclidean estimate; the spherical estimate is subsequently obtained by radially thickening a function on
the sphere and applying the Euclidean result.

\subsection{Previous literature}

\subsubsection*{Stronger transport graph norms}

The classical transport trace theorems use stronger graph norms than the kinetic energy space considered here. In Bardos's \(\L^2\) setting, the
first-order transport expression itself belongs to \(\L^2\) \cite{Bardos1970}. Cessenat's neutron-transport spaces analogously require both
\(f\in\L^p\) and \(v\cdot\nabla_xf\in\L^p\) \cite{Cessenat1984,Cessenat1985}. Their approximation and weighted trace results therefore do not
cover an \(\H^{-1}_v\)-valued transport derivative.

Manteuffel, Ressel, and Starke likewise worked with the strong \(\L^2\) graph norm, which controls both \(f\) and \(v\cdot\nabla_xf\) in
\(\L^2\). They proved that the incoming trace is a bounded surjection onto an exit-time-weighted boundary space. Their example also shows that a
natural-flux \(\L^2\) trace need not exist on the entire graph space \cite[Theorem~2.2 and pp.~562--563]{ManteuffelResselStarke2000}.

\subsubsection*{Kinetic energy spaces}

By contrast, for graph spaces with an \(\H^{-1}_v\)-valued transport derivative, Baouendi and Grisvard proved the one-dimensional natural trace
estimate and the Green identity \cite[Theorem~1 and Corollary~1]{BaouendiGrisvard1968}. Degond and Mas--Gallic stated the same one-dimensional
estimate, referring to Baouendi and Grisvard \cite[Lemma~II.1]{DegondMasGallic1987}. Armstrong and Mourrat described the independent argument in
that paper as incomplete \cite[Appendix~A]{ArmstrongMourrat2019}. Brunken and Smetana later observed that its grazing cut-off step does not
establish convergence, or even a uniform bound, for the \(\H^{-1}_v\)-valued transport derivative \cite[Supplement,
  Section~SM2]{BrunkenSmetana2022}.

Carrillo asserted the corresponding higher-dimensional natural trace theorem, invoking the ideas of Degond and Mas--Gallic
\cite[Lemma~2.3]{Carrillo1998}. The proof relies on an approximation argument with suitable smooth functions and decomposes each approximation
into incoming and outgoing pieces, but does not establish a uniform bound for their graph norms in terms of the graph norm of the original
approximation \cite[Appendix~A]{ArmstrongMourrat2019}. Armstrong and Mourrat left the global trace estimate, including grazing velocities, open
and proved traces only locally away from the grazing set \cite[Question~4.2 and Lemma~4.3]{ArmstrongMourrat2019}. Their work nevertheless
contained an
independent error in the bounded-domain Dirichlet argument: the cut-off estimate in the two displays following (4.20) is not uniform in the test
function \cite[Lemma~4.5, Step~2]{ArmstrongMourrat2019}. The revised work of Albritton, Armstrong, Mourrat, and Novack states that this argument
could not be repaired, removes the bounded-domain Dirichlet and Cauchy--Dirichlet conclusions, and formulates the natural Gaussian trace estimate
as Question~1.8 \cite[Section~1C, footnote~6, and Remark~4.3]{AAMN2024}.

Bal and Palacios asserted global incoming and outgoing traces for a spherical-velocity graph space in the proof of \cite[Theorem~2.3,
  Step~2]{BalPalacios2020}. The decomposition argument there yields estimates in terms of the graph norms of the one-sided pieces, but does not
establish that these norms are controlled by the graph norm of their sum. Consequently, the argument does not prove the asserted global trace
statement \cite[Supplement, Section~SM2]{BrunkenSmetana2022}. For a space--time graph space with spherical velocities, Brunken and Smetana proved
smooth density and local traces on compact non-characteristic boundary pieces \cite[Propositions~3.1 and~3.2]{BrunkenSmetana2022}. They also
proved a global one-sided trace on graph-norm closures of smooth functions vanishing on the opposite boundary part \cite[Proposition~3.4 and
  Remark~3.5]{BrunkenSmetana2022}. The additional global trace and integration-by-parts statement needed in their uniqueness argument, for
graph-space functions with zero inflow trace that satisfy the homogeneous variational equation, is Assumption~4.4 and is left open there.

\subsubsection*{Renormalised and solution-dependent traces}

Other results concern solutions of transport equations rather than bounded trace operators on an ambient graph space. Cannone and Cercignani
treated the trace property used in the Boltzmann initial--boundary theory on Lyapunov surfaces \cite{CannoneCercignani1991}. Under his hypotheses
on the Sobolev vector field \(b\), Boyer considered distributional solutions \(\rho\in\L^\infty(0,T;\L^p(\Omega))\), \(1<p\le\infty\), of
\[
  \partial_t\rho+b\cdot\nabla_x\rho=g,
  \qquad
  g\in\L^1(0,T;\L^p(\Omega)).
\]
He proved existence and uniqueness of a renormalised trace. For finite \(p\), the trace estimate uses the exit-time-weighted \(\L^p\) boundary
space \cite[Theorems~3.1 and~5.1]{Boyer2005}.

In an operator-theoretic setting, Nier constructed maximal-accretive realisations of geometric Kramers--Fokker--Planck operators with a class of
boundary conditions, and proved integration-by-parts identities and global subelliptic estimates \cite[Theorems~1.1 and~1.2]{Nier2018}. Mischler
constructed renormalised boundary traces for weak Vlasov solutions and for kinetic equations with Maxwell boundary conditions
\cite{Mischler2000,Mischler2010}. Bernard recorded the gap in Carrillo's Lemma~2.3 and constructed weak \(\L^1\) solutions with an absorbing
boundary condition by duality from probabilistic solution theory. At positive times, these solutions have natural-flux \(\L^1\) traces and
satisfy Green's formula \cite[Theorems~2 and~4 and Lemmas~17--18]{Bernard2024}.

Carrapatoso and Mischler used renormalised boundary traces in their Maxwell-reflection semigroup theory \cite{CarrapatosoMischler2024}. The
companion work of Carrapatoso, Gabriel, Medina, and Mischler records a renormalised trace theorem for kinetic Fokker--Planck solutions. It gives
a local \(\L^2\) trace for the measure \((n_\Omega\cdot\widehat v)^2\,\dd v\,\dd S\), where \(\widehat v=v/ \sqrt{1+|v|^2}\), and a local
natural-flux \(\L^2\) trace when the initial value and incoming trace have the corresponding local \(\L^2\) control
\cite[Theorem~3.1]{CarrapatosoGabrielMedinaMischler2026}.

Zhu formulated time-dependent weak solutions as solution--trace pairs and established renormalised boundary identities \cite{Zhu2024}. Choi and
Song constructed global weak solutions to nonlinear, possibly degenerate kinetic Fokker--Planck equations with inflow or partial
absorption--reflection boundary conditions as solution-trace pairs. Their traces satisfy natural-flux \(\L^1\), kinetic-energy, and entropy
bounds; boundedness of a trace operator on the underlying function spaces is left open \cite[Theorems~1.1 and~1.3 and
  Remark~1.2(4)]{ChoiSong2025}.

Variational formulations provide another way to avoid requiring a suitable
trace operator. Litsg{\aa}rd and Nystr\"om obtained bounded-domain
solvability without such an operator on the entire kinetic energy space
\cite{LitsgardNystrom2021}.

\subsubsection*{Quadratic grazing estimates}

On \(\C^{1,1}\) domains, Silvestre proved a local trace with weight \(\min\{|v\cdot n_\Omega|,|v\cdot n_\Omega|^2\}\); the natural weight was
left open \cite[Proposition~4.3 and Remark~4.4]{Silvestre2022}. Ouyang and Silvestre defined a restriction operator from their weighted
fractional kinetic graph space to local boundary \(\L^2\), with the same quadratic grazing weight, in their work on the non-cutoff Boltzmann
equation \cite[Proposition~3.9]{OuyangSilvestre2024}. Avelin and Hou used a weak local trace estimate with quadratic grazing weight and a Green
identity in their stationary theory on bounded position and velocity domains \cite[Lemma~3.1]{AvelinHou2026}. For spherical velocities, Bosboom,
Egger, and Schlottbom stated a trace estimate with weight \(\lvert v\cdot n_\Omega\rvert\tau^2\) on bounded convex domains and spherical
velocities, where \(\tau=\tau(x,v)\) is the exit time. Their proof uses the assertion that \(\nabla_s\tau_+\) is bounded whenever the boundary is
Lipschitz. This assertion fails already for a cube, so that argument does not establish the estimate under the stated assumptions
\cite[Lemma~3]{BosboomEggerSchlottbom2025}. The second author defined a compatible-trace subspace and proved reciprocal one-sided flux estimates,
each controlling one boundary trace by the graph norm and the opposite trace \cite[Definition~2.6 and Theorem~2.7]{Valentini2026}.

\subsection{Proof architecture and outline}

Section~\ref{sec:flat} proves the natural half-space trace estimate. Section~\ref{sec:lebesgue-counterexample} gives the unrestricted
Lebesgue counterexample by norm-preserving velocity translation. Sections~\ref{sec:unbounded} and~\ref{sec:bounded} develop the unrestricted
Gaussian and bounded-velocity counterexamples, respectively. Section~\ref{sec:bounded-positive} proves the matching bounded-velocity trace
estimate.
Section~\ref{sec:energy} derives the trace operators, including the unrestricted Lebesgue and Gaussian natural traces on a half-space, and proves
Green's formula for these unrestricted half-space natural traces and for the bounded-velocity kinetic energy spaces.
Section~\ref{sec:dimension-one} gives the one-dimensional classification.
The appendices collect proofs of density results.

\subsection*{Acknowledgements}

This work grew out of discussions initiated at the workshop ``Regularity Theory for Evolution Equations'' at the SwissMAP Research Station, Les
Diablerets. The authors thank the organisers and the SwissMAP Research Station for their hospitality. We thank Cl\'ement Mouhot and Luis
Silvestre for stimulating discussions during the genesis of this work.

Lukas Niebel is supported by the Deutsche Forschungsgemeinschaft (DFG, German Research Foundation) under Germany's Excellence Strategy EXC
2044/2--390685587, Mathematics M\"unster: Dynamics--Geometry--Structure. Lisa Valentini is supported by a Gates Cambridge Scholarship issued by
the Gates Cambridge Trust.

\subsection*{Declaration of AI Use}
OpenAI's GPT-5.5 Pro provided preliminary counterexamples and OpenAI's GPT-5.6 Sol provided refinements of those counterexamples and an initial
proof of the natural half-space trace estimate. Generative AI was also used in the subsequent development of some of these arguments and in
drafting portions
of the exposition. The authors verified the mathematical content and are responsible for the final manuscript.

\section{\texorpdfstring{The natural trace estimate on the half-space}
  {The natural trace estimate on the half-space}}
\label{sec:flat}

Let \(d\ge1\). Write \(x=(z,y)\in\R^{d-1}\times(0,\infty)\), and \(v=(u,w)\in\R^{d-1}\times\R\). Set
\[
  \mathbb H^d:=\R^{d-1}\times(0,\infty).
\]
Here \(\mu\) is either Lebesgue measure on \(\R^d\) or the standard Gaussian measure \(\gamma_d\).

\begin{theorem}[Natural trace estimate on a half-space]\label{thm:flat}
  There is a constant \(C_d>0\) such that every \(f\in \C_c^\infty(\overline{\mathbb H^d}\times\R^d)\) satisfies the $\omega_1$-trace estimate
  \eqref{eq:trace_ineq}, that is,
  \begin{equation}\label{eq:flat-main}
    \int_{\R^{d-1}}\int_{\R^d}
    |f(z,0,v)|^2|w|\dd\mu(v)\dd z
    \le C_d\left(
    \norm{f}_{\L^2(\mathbb H^d;\H^1_\mu)}^2+
    \norm{v\cdot\nabla_xf}_{\L^2(\mathbb H^d;\H^{-1}_\mu)}^2\right),
  \end{equation}
  and consequently the $\omega_p$-trace estimate for every $1\le p < \infty$.
\end{theorem}

We first identify the cubic scale and then use a single self-similar multiplier which supplies both grazing coercivity and the limiting boundary
sign.

\subsection{The cubic grazing scale}

We first analyse the behaviour of a smoothed sign multiplier near the boundary. The cubic scale \eqref{eq:cubic_scale} is visible before the
multiplier is chosen. If \(s(y,w)=S(w/a(y))\) for some smoothed sign $S \in \C^1(\R)$ and a scaling
$a \in \C^1((0,\infty);(0,\infty))$, then, in the transition region where \(|w|\lesssim a(y)\), one has
\begin{equation}
  \label{eq:exponents-balance}
  |\partial_ws|^2\lesssim a(y)^{-2},
  \qquad |v\cdot\nabla_xs|=|w\partial_ys|\lesssim |a'(y)|.
\end{equation}
For \(a(y)=y^\theta\), $\theta>0$, these singular densities have orders \(y^{-2\theta}\) and \(y^{\theta-1}\). Balancing them gives
\(-2\theta=\theta-1\), hence \(\theta=1/3\). Thus the grazing window where boundary concentration may occur is \(|w|\lesssim y^{1/3}\). We call
$\{(y,w):0<y<1,\ |w|<2y^{1/3}\}$ the grazing cone. The aperture condition \(|w|<2y^{1/3}\) is invariant under the kinetic scaling
\((y,w)\mapsto(\lambda^3y,\lambda w)\), $\lambda>0$.

\subsection{\texorpdfstring{The grazing functional and multiplier tools}
  {The grazing functional and multiplier tools}}

For the remainder of this section, fix a function \(f\in \C_c^\infty(\overline{\mathbb H^d}\times\R^d)\).

We quantify concentration near grazing velocities by the grazing-cone mass
\begin{equation}\label{eq:flat-P}
  \mathcal P(f):=
  \int_0^1\int_{\R^{d-1}}\int_{\R^d}
  y^{-2/3}\one_{\{|w|<2y^{1/3}\}}|f(z,y,v)|^2
  \dd\mu(v)\dd z\dd y.
\end{equation}
We will control \(\mathcal P(f)\) by the kinetic energy norm appearing on the right-hand side of \eqref{eq:flat-main}. The resulting estimate is
a
kinetic grazing-cone analogue of Hardy's inequality.

The weight is dictated by the cubic grazing scaling, which gives $y^{-2/3}$ in \eqref{eq:exponents-balance}. The corresponding weighted integral
is finite for the test class: the \(w\)-section has length \(O(y^{1/3})\), leaving the integrable singularity \(y^{-1/3}\).

The multiplier used below will control \(\mathcal P(f)\) and recover the boundary flux in the same Green identity.

The argument uses the following Green identity on the truncated half-space \(\mathbb H^d_\delta=\{y>\delta\}\), for \(\delta>0\).

\begin{lemma}[Truncated Green identity]\label{lem:flat-green}
  Let \(f\) be as above and let \(m=m(y,w)\) be \(\C^1\) for \(y>0\). Assume that \(mf\in\L^2(\mathbb H^d_\delta;\H^1_\mu)\), that
  \[
    \int_{\mathbb H^d_\delta}\int_{\R^d}
    |v\cdot\nabla_xm|\,|f|^2\dd\mu\dd x<\infty,
  \]
  and that the boundary term below is finite. Then
  \begin{align}
    -\int_{\R^{d-1}}\int_{\R^d}
    w\,m(\delta,w)|f(z,\delta,v)|^2\dd\mu(v)\dd z
     & =
    \int_{\mathbb H^d_\delta}\int_{\R^d}
    (v\cdot\nabla_xm)|f|^2\dd\mu(v)\dd x \notag \\
     & \quad+
    2\int_{\mathbb H^d_\delta}
    \pair{v\cdot\nabla_xf}{mf}_{\H^{-1}_\mu,\H^1_\mu}\dd x.
    \label{eq:flat-green}
  \end{align}
\end{lemma}

\begin{proof}
  Apply the divergence theorem in \(x\) to \(v\,m|f|^2\). On \(y=\delta\) the outward normal is \(-e_y\); compact support removes the remaining
  boundary terms. Expanding \(v\cdot\nabla_x(m|f|^2)\) gives \eqref{eq:flat-green}; the stated hypotheses also make the duality term integrable.
  Since no integration by parts in \(v\) is used, the identity is the same for both Lebesgue and Gaussian velocity measures.
\end{proof}

Consider the annulus
\[
  K:=\left\{r\in\R:\frac54<|r|<\frac32\right\}.
\]
The elementary anchored Poincaré estimate
\begin{equation}\label{eq:flat-anchor}
  \int_{-2}^{2}|G(r)|^2\dd r
  \le C\left(\int_K|G(r)|^2\dd r+
  \int_{-2}^{2}|G'(r)|^2\dd r\right)
\end{equation}
holds for every \(G\in \H^1(-2,2)\). To prove it, subtract the average over \(K\), control the oscillation by the fundamental theorem of
calculus, and apply Cauchy--Schwarz to the average.

The following profile will also be used in the proof of the $\omega_p$ estimates on curved domains with bounded velocities.
\begin{lemma}[Compact odd multiplier profile]
  \label{lem:compact-odd-hardy-profile}
  Let
  \[
    K=\left\{r\in\R:\frac54<|r|<\frac32\right\}.
  \]
  There are constants \(c_0,C_0>0\) such that, for every \(0<\varepsilon<1/8\), there is an odd \(Q_\varepsilon\in\C_c^\infty((-2,2))\)
  satisfying \(rQ_\varepsilon(r)\ge0\) and
  \begin{equation}\label{eq:compact-odd-hardy-profile}
    -r^2Q_\varepsilon'(r)
    \ge c_0\one_K(r)
    -C_0\varepsilon\one_{\{|r|<2\varepsilon\}}(r).
  \end{equation}
\end{lemma}

\begin{proof}
  Choose a smooth odd nondecreasing function \(\sigma\colon\R\to[-1,1]\) equal to \(\sgn r\) for \(|r|\ge2\). Choose an even
  \(\vartheta\in\C_c^\infty((-2,2))\), with \(0\le\vartheta\le1\), \(\vartheta=1\) on \([-1,1]\), nonincreasing for \(r>0\), and
  \(\vartheta'(r)\le-c_\vartheta<0\) on \(5/4<r<3/2\). Set
  \[
    Q_\varepsilon(r):=\sigma(r/\varepsilon)\vartheta(r).
  \]
  Then \(Q_\varepsilon\) is odd and \(rQ_\varepsilon(r)\ge0\). Moreover,
  \[
    -r^2Q_\varepsilon'(r)
    =-\varepsilon^{-1}r^2\sigma'(r/\varepsilon)\vartheta(r)
    -r^2\sigma(r/\varepsilon)\vartheta'(r).
  \]
  The first term is nonpositive, supported in \(|r|<2\varepsilon\), and bounded below there by \(-C_0\varepsilon\). The second term is
  nonnegative. On \(K\), it is bounded below by a positive constant: indeed, \(\varepsilon<1/8\) implies \(|r|/\varepsilon>2\), and
  \(|\vartheta'(r)|\ge c_\vartheta\). This proves \eqref{eq:compact-odd-hardy-profile}.
\end{proof}

\begin{lemma}[Coercive overshooting sign profile]
  \label{lem:flat-coercive-sign}
  There exist \(0<\varepsilon<1/8\), an odd \(F_\varepsilon\in\C^\infty(\R)\), and \(C>0\) such that
  \begin{equation}\label{eq:flat-coercive-sign-properties}
    rF_\varepsilon(r)\ge0,\qquad
    |F_\varepsilon(r)|\le2,\qquad
    F_\varepsilon(r)=\sgn r\quad\text{for }|r|\ge2,
    \qquad
    \supp F_\varepsilon'\subset[-2,2].
  \end{equation}
  If
  \[
    h_\varepsilon(r):=-r^2F_\varepsilon'(r),
  \]
  then, for every measurable \(\rho\colon(-2,2)\to\R\) satisfying \(e^{-2}\le\rho\le1\) and every \(G\in\H^1(-2,2)\),
  \begin{equation}\label{eq:flat-weighted-profile-coercivity}
    \int_{-2}^{2}|G|^2\rho\,\dd r
    \le C\left(
    \int_{-2}^{2}h_\varepsilon|G|^2\rho\,\dd r
    +
    \int_{-2}^{2}|G'|^2\rho\,\dd r
    \right).
  \end{equation}
\end{lemma}

\begin{proof}
  Choose \(\sigma\) and \(\vartheta\) as in the proof of Lemma~\ref{lem:compact-odd-hardy-profile} and set
  \[
    F_\varepsilon(r):=\sigma(r/\varepsilon)(1+\vartheta(r)).
  \]
  The properties in \eqref{eq:flat-coercive-sign-properties} follow from the choices of \(\sigma\) and \(\vartheta\). Moreover,
  \[
    h_\varepsilon(r)
    =-\varepsilon^{-1}r^2\sigma'(r/\varepsilon)(1+\vartheta(r))
    -r^2\sigma(r/\varepsilon)\vartheta'(r).
  \]
  The first term is supported in \(|r|<2\varepsilon\) and is bounded below there by \(-C_0\varepsilon\). The second term is nonnegative
  everywhere and bounded below by a constant \(c_0>0\) on \(K\). Hence
  \begin{equation}\label{eq:flat-coercive-sign-pointwise}
    h_\varepsilon
    \ge c_0\one_K-C_0\varepsilon\one_{\{|r|<2\varepsilon\}}.
  \end{equation}

  Fix a weight \(\rho\) as in the statement and abbreviate
  \[
    \mathsf L=\int_{-2}^{2}|G|^2\rho,\qquad
    \mathsf R=\int_K|G|^2\rho,\qquad
    \mathsf D=\int_{-2}^{2}|G'|^2\rho,\qquad
    \mathsf I=\int_{-2}^{2}h_\varepsilon|G|^2\rho.
  \]
  Since \(e^{-2}\le\rho\le1\), the anchored estimate \eqref{eq:flat-anchor} gives
  \[
    \mathsf L\le C_{\rm an}(\mathsf R+\mathsf D),
  \]
  with \(C_{\rm an}\) independent of \(\rho\). On the other hand, \eqref{eq:flat-coercive-sign-pointwise} gives
  \[
    \mathsf I\ge c_0\mathsf R-C_0\varepsilon\mathsf L.
  \]
  Set \(c_*:=\min\{c_0,1\}\). Since \(\mathsf D\ge0\), the last two inequalities imply
  \[
    \mathsf I+\mathsf D
    \ge c_*(\mathsf R+\mathsf D)-C_0\varepsilon\mathsf L
    \ge
    \left(\frac{c_*}{C_{\rm an}}-C_0\varepsilon\right)\mathsf L.
  \]
  Choose \(\varepsilon<1/8\) so small that \(C_0\varepsilon\le c_*/(2C_{\rm an})\). Then
  \[
    \mathsf L
    \le \frac{2C_{\rm an}}{c_*}
    \bigl(\mathsf I+\mathsf D\bigr),
  \]
  which is \eqref{eq:flat-weighted-profile-coercivity}.
\end{proof}

\begin{lemma}[Self-similar multiplier]\label{lem:flat-self-similar}
  Let \(F\in\C^1(\R)\), assume that \(F\) and \(F'\) are bounded, and that \(\supp F'\subset[-2,2]\). For
  \[
    r:=wy^{-1/3},\qquad m_F(y,w):=F(r),
  \]
  one has
  \begin{equation}\label{eq:flat-self-similar-derivatives}
    \partial_wm_F=y^{-1/3}F'(r),
    \qquad
    v\cdot\nabla_xm_F
    =-\frac13y^{-2/3}r^2F'(r).
  \end{equation}
  Moreover,
  \begin{align}
    \norm{m_Ff}_{\L^2_x\H^1_\mu}^2
    +\int_{\mathbb H^d}\int_{\R^d}
    |v\cdot\nabla_xm_F|\,|f|^2\dd\mu\dd x
    \le C_F\left(
    \norm{f}_{\L^2_x\H^1_\mu}^2+\mathcal P(f)\right).
    \label{eq:flat-self-similar-bound}
  \end{align}
\end{lemma}

\begin{proof}
  Since
  \[
    \partial_wr=y^{-1/3},
    \qquad
    w\partial_yr=-\frac13y^{-2/3}r^2,
  \]
  the identities in \eqref{eq:flat-self-similar-derivatives} follow from the chain rule. On \(0<y<1\), the support assumption on \(F'\) gives
  \[
    |\partial_wm_F|^2+|v\cdot\nabla_xm_F|
    \le C_Fy^{-2/3}\one_{\{|w|<2y^{1/3}\}}.
  \]
  For \(y\ge1\), both quantities are bounded by \(C_F\). The product rule
  \[
    \nabla_v(m_Ff)=m_F\nabla_vf+(\partial_wm_F)f\,e_w
  \]
  now proves \eqref{eq:flat-self-similar-bound}.
\end{proof}

\subsection{\texorpdfstring{The overshooting sign-multiplier argument}
  {The overshooting sign-multiplier argument}}

The profile and the corresponding self-similar regions are shown in Figure~\ref{fig:flat-hardy-regions}.

\begin{figure}[htbp]
  \centering
  \begin{minipage}[c]{.51\linewidth}
    \centering
    \scalebox{.72}{\begin{tikzpicture}[font=\small]
        \begin{scope}[x=2.0cm,y=1.55cm]
          \def\hseps{.10}
          \path[hs anchor] (-1.5,-1.18) rectangle (-1.25,1.18);
          \path[hs anchor] (1.25,-1.18) rectangle (1.5,1.18);
          \path[hs error] (-2*\hseps,-1.18) rectangle (2*\hseps,1.18);

          \draw[hs axis] (-2.38,0) -- (2.40,0)
          node[below right=-1pt,hs note] {\(r\)};
          \draw[hs axis] (0,-1.30) -- (0,1.32)
          node[above left=-1pt,hs note] {\(F_\varepsilon(r)\)};

          \draw[hs guide] (-2.28,-.5) -- (-2,-.5);
          \draw[hs guide] (2,.5) -- (2.28,.5);
          \draw[draw=hsInk,line width=1.2pt]
          (-2.28,-.5) -- (-2,-.5)
          .. controls (-1.78,-.52) and (-1.48,-.96) .. (-1,-1)
          -- ({-2*\hseps},-1)
          .. controls ({-.9*\hseps},-1) and ({.9*\hseps},1)
          .. ({2*\hseps},1)
          -- (1,1)
          .. controls (1.48,.96) and (1.78,.52) .. (2,.5)
          -- (2.28,.5);

          \node[hs note,anchor=east] at (-2.03,-.38) {\(-1\)};
          \node[hs note,anchor=west] at (2.03,.38) {\(+1\)};
          \node[hs note,anchor=south] at (-.72,-1.02) {\(-2\)};
          \node[hs note,anchor=north] at (.72,1.02) {\(+2\)};
          \node[hs note,text=hsOrange!65!black] at (-1.375,.7) {\(K\)};
          \node[hs note,text=hsOrange!65!black] at (1.375,.7) {\(K\)};
          \node[hs note,text=hsMagenta!75!black,anchor=west]
          at (.28,-.72) {\(\lvert r\rvert<2\varepsilon\)};

          \path[hs anchor]
          (-2.03,-1.58) rectangle (-1.81,-1.45);
          \node[hs note,text=hsOrange!65!black,anchor=west]
          at (-1.68,-1.515)
          {coercivity on \(K\):
            \(-r^2F_\varepsilon'(r)\ge c\)};
          \path[hs error]
          (-2.03,-1.88) rectangle (-1.81,-1.75);
          \node[hs note,text=hsMagenta!75!black,anchor=west]
          at (-1.68,-1.815)
          {central defect:
            \(-r^2F_\varepsilon'(r)\ge-C\varepsilon\)};
        \end{scope}
      \end{tikzpicture}}
  \end{minipage}\hfill
  \begin{minipage}[c]{.44\linewidth}
    \centering
    \scalebox{.72}{
      \begin{tikzpicture}[font=\small]
        \def\hseps{.10}
        \begin{scope}[x=7cm,y=1.2cm]
          \path[fill=hsBlue!12]
          plot[domain=0:1,variable=\t,samples=101]
          ({pow(\t,3)},{2*\t})
          -- plot[domain=1:0,variable=\t,samples=101]
          ({pow(\t,3)},{-2*\t})
          -- cycle;

          \path[hs anchor]
          plot[domain=0:1,variable=\t,samples=81]
          ({pow(\t,3)},{1.5*\t})
          -- plot[domain=1:0,variable=\t,samples=81]
          ({pow(\t,3)},{1.25*\t})
          -- cycle;
          \path[hs anchor]
          plot[domain=0:1,variable=\t,samples=81]
          ({pow(\t,3)},{-1.5*\t})
          -- plot[domain=1:0,variable=\t,samples=81]
          ({pow(\t,3)},{-1.25*\t})
          -- cycle;

          \path[hs error]
          plot[domain=0:1,variable=\t,samples=81]
          ({pow(\t,3)},{2*\hseps*\t})
          -- plot[domain=1:0,variable=\t,samples=81]
          ({pow(\t,3)},{-2*\hseps*\t})
          -- cycle;

          \draw[draw=hsBlue,line width=1pt]
          plot[domain=0:1,variable=\t,samples=101]
          ({pow(\t,3)},{2*\t});
          \draw[draw=hsBlue,line width=1pt]
          plot[domain=0:1,variable=\t,samples=101]
          ({pow(\t,3)},{-2*\t});

          \draw[hs axis] (-.025,0) -- (1.10,0)
          node[below right=-1pt,hs note] {\(y\)};
          \draw[hs axis] (0,-2.45) -- (0,2.48)
          node[above left=-1pt,hs note] {\(w\)};
        \end{scope}

        \path[draw=hsBlue,fill=hsBlue!12,line width=.8pt]
        (.15,-3.62) rectangle (.55,-3.36);
        \node[hs note,anchor=west] at (.75,-3.49)
        {grazing cone:
          \(\lvert w\rvert<2y^{1/3}\)};

        \path[hs anchor] (.15,-4.37) rectangle (.55,-4.11);
        \node[hs note,anchor=west] at (.75,-4.24)
        {coercive anchor bands:
          \(\frac54y^{1/3}<\lvert w\rvert<\frac32y^{1/3}\)};

        \path[hs error] (.15,-5.12) rectangle (.55,-4.86);
        \node[hs note,anchor=west] at (.75,-4.99)
        {central \(\varepsilon\)-defect:
          \(\lvert w\rvert<2\varepsilon y^{1/3}\)};
      \end{tikzpicture}}
  \end{minipage}
  \caption{The overshooting sign profile \(F_\varepsilon\) (left) and the
    grazing cone, coercive anchor bands, and central defect after the
    self-similar substitution \(r=wy^{-1/3}\) (right).}
  \label{fig:flat-hardy-regions}
\end{figure}
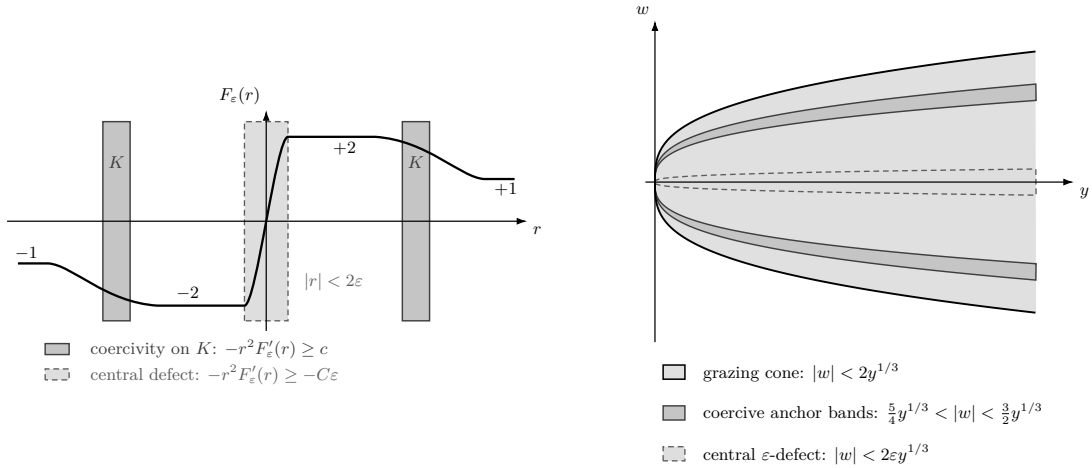

\begin{proof}[Proof of Theorem~\ref{thm:flat}]
  Write \(T:=v\cdot\nabla_x\) and set
  \[
    M:=\norm{f}_{\L^2_x\H^1_\mu},
    \qquad
    N:=\norm{Tf}_{\L^2_x\H^{-1}_\mu},
    \qquad
    P:=\mathcal P(f).
  \]
  Let \(\varepsilon\) and \(F_\varepsilon\) be supplied by Lemma~\ref{lem:flat-coercive-sign}, and define
  \[
    r:=wy^{-1/3},
    \qquad
    m(y,w):=F_\varepsilon(r),
    \qquad
    h_\varepsilon(r):=-r^2F_\varepsilon'(r).
  \]
  Lemma~\ref{lem:flat-self-similar} gives
  \begin{equation}\label{eq:flat-one-profile-product}
    \norm{mf}_{\L^2_x\H^1_\mu}^2
    +\int_{\mathbb H^d}\int_{\R^d}|Tm|\,|f|^2\dd\mu\dd x
    \le C(M^2+P),
  \end{equation}
  and
  \begin{equation}\label{eq:flat-one-profile-transport}
    Tm=\frac13y^{-2/3}h_\varepsilon(r).
  \end{equation}

  We first convert the one-dimensional coercivity \eqref{eq:flat-weighted-profile-coercivity} into a grazing estimate. Fix \((z,y,u)\) with
  \(0<y<1\) and put
  \[
    G(r):=f(z,y,u,y^{1/3}r).
  \]
  For Lebesgue measure, take \(\rho=1\). For Gaussian measure, the change of variables \(w=y^{1/3}r\) produces, apart from its constant
  normalising factor,
  \[
    \rho_y(r):=e^{-y^{2/3}r^2/2}.
  \]
  On \(|r|\le2\), one has \(e^{-2}\le\rho_y\le1\), while the tangential Gaussian factor is unchanged. Since
  \[
    G'(r)=y^{1/3}\partial_wf(z,y,u,y^{1/3}r),
    \qquad
    \dd w=y^{1/3}\dd r,
  \]
  applying \eqref{eq:flat-weighted-profile-coercivity}, multiplying by $y^{-1/3}$, changing back to $w$, and integrating in \((z,y,u)\) gives
  \begin{equation}\label{eq:flat-one-profile-slice}
    P
    \le C\int_0^1\int_{\R^{d-1}}\int_{\R^d}
    y^{-2/3}h_\varepsilon(r)|f|^2
    \dd\mu\dd z\dd y
    +CM^2.
  \end{equation}
  Define the quantity
  \[
    A:=\int_{\mathbb H^d}\int_{\R^d}(Tm)|f|^2\dd\mu\dd x.
  \]
  Also set
  \[
    I_{\ge1}:=
    \int_{\{y\ge1\}}\int_{\R^{d-1}}\int_{\R^d}
    y^{-2/3}h_\varepsilon(r)|f|^2\dd\mu\dd z\dd y.
  \]
  Since \(h_\varepsilon\) is bounded and \(y^{-2/3}\le1\) for \(y\ge1\),
  \[
    |I_{\ge1}|\le CM^2.
  \]
  Moreover, \eqref{eq:flat-one-profile-transport} gives the exact identity
  \[
    \int_0^1\int_{\R^{d-1}}\int_{\R^d}
    y^{-2/3}h_\varepsilon(r)|f|^2\dd\mu\dd z\dd y
    =3A-I_{\ge1}.
  \]
  Combining this bound with \eqref{eq:flat-one-profile-transport}--\eqref{eq:flat-one-profile-slice} yields the key inequality
  \begin{equation}\label{eq:flat-one-profile-grazing}
    P\le CA+CM^2.
  \end{equation}

  Recall that \(P<\infty\) for the smooth compactly supported test class, as observed after \eqref{eq:flat-P}. Thus
  \eqref{eq:flat-one-profile-product} justifies all the following bulk integrals before any estimate is closed. Apply
  Lemma~\ref{lem:flat-green} on \(\mathbb H^d_\delta\) with multiplier \(m\). By \eqref{eq:flat-coercive-sign-properties},
  \[
    w\,m(\delta,w)\ge0,\qquad
    |w\,m(\delta,w)|\le2|w|,
    \qquad
    w\,m(\delta,w)\longrightarrow |w|.
  \]
  Smoothness and compact support of \(f\) therefore give dominated convergence in the boundary integral. The first bulk term converges by the
  absolute integrability in \eqref{eq:flat-one-profile-product}; the duality term converges by Cauchy--Schwarz and absolute continuity. Hence
  \begin{align}
    -B & =A+2D, \label{eq:flat-one-profile-green} \\
    B  & :=\int_{\R^{d-1}}\int_{\R^d}
    |w|\,|f(z,0,v)|^2\dd\mu\dd z, \notag          \\
    D  & :=\int_{\mathbb H^d}
    \pair{Tf}{mf}_{\H^{-1}_\mu,\H^1_\mu}\dd x. \notag
  \end{align}

  Since \(B\ge0\), identity \eqref{eq:flat-one-profile-green} gives
  \[
    A=-B-2D\le2|D|.
  \]
  By duality and \eqref{eq:flat-one-profile-product},
  \[
    |D|\le N\norm{mf}_{\L^2_x\H^1_\mu}
    \le CN(M+P^{1/2}).
  \]
  Inserting these bounds into \eqref{eq:flat-one-profile-grazing} and using Young's inequality gives
  \[
    P\le CM^2+CN(M+P^{1/2})
    \le\frac12P+C(M^2+N^2).
  \]
  Consequently, we obtain the kinetic grazing-cone Hardy inequality
  \begin{equation}\label{eq:flat-one-profile-hardy}
    P\le C(M^2+N^2).
  \end{equation}

  Returning to the same Green identity, and now using \eqref{eq:flat-one-profile-product} and \eqref{eq:flat-one-profile-hardy}, we obtain
  \[
    B\le |A|+2|D|
    \le \int_{\mathbb H^d}\int_{\R^d}|Tm|\,|f|^2\dd\mu\dd x
    +2N\norm{mf}_{\L^2_x\H^1_\mu}
    \le C(M^2+N^2).
  \]
  This is \eqref{eq:flat-main}. Finally, \(\omega_p(t)\le t\) for every \(t\ge0\) and \(1\le p<\infty\), so the \(\omega_p\)-trace estimate
  follows as well.
\end{proof}

\begin{remark}[Why flatness matters]\label{rem:flatness}
  The construction depends on the fixed normal coordinate \(w\). Replacing it by \(v\cdot N(x)\), where \(N\) is a sufficiently regular interior
  extension of \(n_\Omega\), introduces \(v\cdot\nabla_x(v\cdot N(x))\) in the transport derivative. This term is quadratic in \(v\) and measures
  variation of the extended normal. The boundary-flattening change of variables in \cite[Section~3, p.~8]{Silvestre2022} produces the
  corresponding quadratic velocity term in the transformed drift.
\end{remark}

\section{\texorpdfstring{Unrestricted Lebesgue model:
    norm-preserving velocity translation} {Unrestricted Lebesgue model: norm-preserving velocity translation}}
\label{sec:lebesgue-counterexample}

The key mechanism for this counterexample is norm-preserving velocity translation.

\begin{proposition}[Failure in the unrestricted Lebesgue model]
  \label{prop:interval-lebesgue-failure}
  Let \(d\ge1\), \(p\ge1\), and $(V,\mu)=(\R^d,\mathcal L^d)$. Let \(\Omega\subset\R^d\) be a bounded \(\C^1\) domain. There is no constant
  \(C_\Omega>0\) such that the $\omega_p$-trace estimate \eqref{eq:trace_ineq} holds for every \(f\in\C_c^\infty(\overline \Omega\times\R^d)\).
\end{proposition}

\begin{proof}
  Choose a unit vector \(e\), a boundary patch \(\Sigma\subset\partial \Omega\) of positive surface measure, and \(c_0>0\) such that \(e\cdot
  n_\Omega\ge c_0\) on \(\Sigma\). In dimension one, a boundary point equipped with counting measure is such a patch. For a nonzero
  \(\psi\in\C_c^\infty(B_1)\) and \(R>2\), set
  \[
    f_R(x,v)=\psi(v-Re).
  \]

  This packet is independent of \(x\), so its transport derivative vanishes, and translation invariance of Lebesgue measure gives
  \[
    \mathcal E_\Omega^{\mathcal L^d}(f_R)
    =|\Omega|\norm{\psi}_{\H^1(\R^d)}^2.
  \]
  On \(\Sigma\times\supp_vf_R\), one has \(|v\cdot n_\Omega|\ge c_0R/2\ge1\) for all sufficiently large \(R\). Hence
  \(\omega_p(|v\cdot n_\Omega|)=|v\cdot
  n_\Omega|\) there and
  \[
    \mathcal T^{\mathcal L^d}_{p,\Omega}(f_R)
    \ge \frac{c_0}{2}|\Sigma|\norm\psi_{\L^2(\R^d)}^2R
    \longrightarrow\infty. \qedhere
  \]
\end{proof}

Boundedness of \(\Omega\) makes an \(x\)-independent packet an admissible finite-energy test function. Localising the same packet in an unbounded
half-space introduces a transport cost of order \(R\), so this construction does not contradict Theorem~\ref{thm:flat}.

\section{\texorpdfstring{Unrestricted Gaussian model: the
    \(\omega_2\)-trace estimate and curvature-induced normal-flux amplification} {Unrestricted Gaussian model: the omega-2-trace estimate and
    curvature-induced normal-flux amplification}}
\label{sec:unbounded}

In this section, consider the unrestricted Gaussian model \((V,\mu)=(\R^d,\gamma_d)\).

\subsection{The \(\omega_2\) benchmark}

We first record the \(\omega_2\) estimate for bounded \(\C^{1,1}\) domains, which uses the multiplier from Silvestre's local, time-dependent
Lebesgue-velocity argument \cite[Proposition~4.3]{Silvestre2022}.

\begin{proposition}[Gaussian \(\omega_2\) estimate
    (after Silvestre \cite{Silvestre2022})]
  \label{prop:quadratic-gaussian}
  Let $d \ge 1$, and let \(\Omega\subset\R^d\) be a bounded \(\C^{1,1}\) domain. Then there exists a constant \(C_\Omega>0\) such that every
  \(f\in \C_c^\infty(\overline \Omega\times\R^d)\) satisfies the $\omega_2$-trace estimate \eqref{eq:trace_ineq}, that is,
  \begin{align}
    \int_{\partial \Omega}\int_{\R^d}|f(x,v)|^2 &
    \omega_2(|v\cdot n_\Omega(x)|)\dd\gamma_d(v)\dd S(x) \notag      \\
                                                & \le C_\Omega\left(
    \norm{f}_{\L^2(\Omega;\H^1_\gamma)}^2+
    \norm{v\cdot\nabla_xf}_{\L^2(\Omega;\H^{-1}_\gamma)}^2\right).
    \label{eq:quadratic-gaussian}
  \end{align}
\end{proposition}

\begin{proof}
  The proof uses Green's formula with the bounded clipped multiplier that factorises the quadratic grazing weight. Let
  \[
    \kappa(t):=\max\{-1,\min\{t,1\}\}.
  \]
  Since \(\partial\Omega\) is \(\C^{1,1}\), its outward normal is Lipschitz and admits an extension \(N\in\W^{1,\infty}(\Omega;\R^d)\) with
  boundary trace \(n_\Omega\). Set \(a(x,v):=\kappa(v\cdot N(x))\). On \(\partial\Omega\), with \(s=v\cdot n_\Omega(x)\),
  \[
    a(x,v)(v\cdot n_\Omega(x))
    =\kappa(s)s
    =\min\{s^2,|s|\}
    =\omega_2(|s|).
  \]
  The weak chain and product rules give, almost everywhere,
  \[
    |a|\le1,\qquad
    |\nabla_va|\le \norm{N}_{\L^\infty(\Omega)},\qquad
    |v\cdot\nabla_xa|
    \le \norm{\nabla N}_{\L^\infty(\Omega)}|v|^2,
  \]
  and consequently
  \[
    \norm{af}_{\L^2(\Omega;\H^1_\gamma)}
    \le C_\Omega\norm{f}_{\L^2(\Omega;\H^1_\gamma)}.
  \]
  Thus, applying the weak Gauss--Green theorem first for each fixed \(v\) and then integrating over \(v\) yields
  \begin{align*}
    \mathcal T^\gamma_{2,\Omega}(f)
     & =
    \int_\Omega\int_{\R^d}|f|^2v\cdot\nabla_xa
    \dd\gamma_d\dd x \\
     & \quad+
    2\int_\Omega
    \pair{v\cdot\nabla_xf}{af}_{\H^{-1}_\gamma,\H^1_\gamma}\dd x.
  \end{align*}
  The clipping points cause no additional term because the weak chain rule applies directly to the Lipschitz function \(\kappa\).

  It remains only to control the quadratic velocity moment created by \(v\cdot\nabla_xa\). Gaussian integration by parts and Young's inequality
  give, for \(g\in\C_c^\infty(\R^d)\),
  \[
    \int_{\R^d}|v|^2|g|^2\dd\gamma_d
    \le C_d\norm{g}_{\H^1_\gamma(\R^d)}^2.
  \]
  Therefore Cauchy--Schwarz in \(x\), followed by Young's inequality, gives
  \begin{align*}
    \mathcal T^\gamma_{2,\Omega}(f)
     & \le
    C_\Omega\norm{f}_{\L^2(\Omega;\H^1_\gamma)}^2
    +2\norm{v\cdot\nabla_xf}_{\L^2(\Omega;\H^{-1}_\gamma)}
    \norm{af}_{\L^2(\Omega;\H^1_\gamma)} \\
     & \le C_\Omega\left(
    \norm{f}_{\L^2(\Omega;\H^1_\gamma)}^2+
    \norm{v\cdot\nabla_xf}_{\L^2(\Omega;\H^{-1}_\gamma)}^2
    \right),
  \end{align*}
  which is \eqref{eq:quadratic-gaussian}.
\end{proof}

\subsection{Curvature-induced normal-flux amplification and the Gaussian counterexample}

We now construct counterexamples for \(1\le p<2\). Unlike norm-preserving velocity translation for Lebesgue measure, \(\L^2(\gamma)\)-normalised
translation to a tangential speed \(R\) produces a squared derivative cost of order \(R^2\). Thus the construction combines a fixed-width
translated Gaussian packet with a supporting quadratic boundary cap.

Silvestre notes that, on a non-convex domain, the constants in local boundary estimates are expected to deteriorate when the tangential component
of the reference velocity grows in a direction of negative curvature \cite[p.~4]{Silvestre2022}. A related curvature--velocity coupling is
encoded by the kinetic distance \(\alpha(x,v)=|v\cdot\nabla\xi(x)|^2 -2\xi(x)\,v\cdot\nabla^2\xi(x)v\), used in \cite[Definition~2 and
  equation~(5)]{GuoKimTononTrescases2017}; its Hessian term combines the boundary geometry with the velocity. The reason the construction below
applies to every bounded \(\C^{1,1}\) domain is Lemma~\ref{lem:unb-cap}: it produces the required supporting quadratic cap without any lower
curvature assumption. This packet construction realises the coupling on the coherent part of a supporting quadratic cap: normal rotation on the
spatial scale \(\sqrt h\) converts the tangential packet speed \(R\) into a boundary-normal flux of order \(R\sqrt h\), which is the mechanism
driving the growth of the weighted boundary functional.

The proof has four parts. First, a Gaussian-weighted primitive and an \(\L^2(\gamma)\)-normalised translated packet give the velocity-space
estimates. Next, an enclosing ball centred at a fixed interior point and touching the domain at a farthest boundary point supplies a quadratic
cap, confines the entire thin layer to one graph chart, and yields a large ball on which the normal rotates coherently in one fixed tangential
direction. Third, the tangential, normal, and transport scales are balanced to keep the bulk energy bounded. Finally, the upper and lower
boundary estimates give the scale of the weighted boundary functional in \eqref{eq:unb-main-scale}.

\begin{theorem}[Gaussian counterexamples on every \(\C^{1,1}\) domain, $p<2$]
  \label{thm:unbounded}
  Let \(d\ge2\), and let \(\Omega\subset\R^d\) be a bounded \(\C^{1,1}\) domain. There is a family \(f_h\in
  \C_c^\infty(\overline\Omega\times\R^d)\), \(h\downarrow0\), such that
  \[
    \sup_h\mathcal E^\gamma_\Omega(f_h)<\infty,
  \]
  and, for every \(1\le p<\infty\),
  \begin{equation}\label{eq:unb-main-scale}
    \mathcal T^\gamma_{p,\Omega}(f_h)
    \asymp_{p,\Omega} h^{(p-2)/6}.
  \end{equation}
  Thus the weighted boundary functional diverges for \(p<2\). For \(p=2\), it satisfies
  \[
    0<c\le\mathcal T^\gamma_{2,\Omega}(f_h)\le C,
  \]
  and it tends to zero for every \(p>2\).
\end{theorem}

For the rest of the section, let $d\ge 2$ and \(m=d-1\), and retain the decomposition
\begin{equation}
  \label{eq:decomp_variables}
  x=(z,y)\in\R^m\times\R,
  \qquad
  v=(u,w)\in\R^m\times\R.
\end{equation}

\subsection{Velocity packets}

Fix a nonzero even \(\psi\in \C_c^\infty(-\rho,\rho)\), \(0<\rho<1\). For \(0<\delta\le1\), put \(\psi_\delta(w)=\psi(w/\delta)\). Changes of
variables and local comparability of Gaussian and Lebesgue measures give
\begin{align}
  \int\psi_\delta^2\dd\gamma_1     & \asymp\delta,
                                   &
  \int|\psi_\delta'|^2\dd\gamma_1  & \lesssim\delta^{-1},
                                   &
  \int w^2\psi_\delta^2\dd\gamma_1 & \lesssim\delta^3.
  \label{eq:unb-elementary}
\end{align}

The transport estimate uses the following cancellation.

\begin{lemma}[Odd cancellation]\label{lem:unb-odd}
  Let \(m\ge1\). For every \(0<\delta\le1\) and \(\theta\in \L^2_{\gamma_m}(\R^m)\),
  \begin{equation}\label{eq:unb-odd}
    \norm{\theta(u)w\psi_\delta(w)}_{\H^{-1}_\gamma(\R^{m+1})}
    \le C\norm{\theta}_{\L^2_{\gamma_m}}\delta^{5/2}.
  \end{equation}
\end{lemma}

\begin{proof}
  Set
  \[
    g_\delta(w):=w\psi_\delta(w),
    \qquad
    \varrho(w):=(2\pi)^{-1/2}e^{-w^2/2}.
  \]
  Since \(g_\delta\) is odd and \(\varrho\) is even, \(\int_\R g_\delta\dd\gamma_1=0\). Its Gaussian-weighted primitive
  \[
    G_\delta(w):=\varrho(w)^{-1}
    \int_{-\infty}^w g_\delta(s)\varrho(s)\dd s
  \]
  therefore belongs to \(\C_c^\infty(\R)\) and is supported in \([-\rho\delta,\rho\delta]\). On this interval,
  \[
    |G_\delta(w)|
    \le C\int_{-\rho\delta}^{\rho\delta}
    |s|\,|\psi(s/\delta)|\dd s
    \le C\delta^2,
  \]
  whence
  \[
    \norm{G_\delta}_{\L^2_{\gamma_1}}\le C\delta^{5/2}.
  \]
  Since \((G_\delta\varrho)'=g_\delta\varrho\), integration by parts and Fubini's theorem give, for \(\varphi\in\C_c^\infty(\R^{m+1})\),
  \begin{align*}
    \left|\pair{\theta(u)g_\delta(w)}{\varphi}\right|
     & =
    \left|\int_{\R^{m+1}}\theta(u)G_\delta(w)
    \partial_w\varphi(u,w)
    \dd\gamma_m(u)\dd\gamma_1(w)\right| \\
     & \le
    C\norm{\theta}_{\L^2_{\gamma_m}}\delta^{5/2}
    \norm{\varphi}_{\H^1_\gamma(\R^{m+1})}.
  \end{align*}
  The function \(\theta(u)g_\delta(w)\) belongs to \(\L^2_{\gamma_{m+1}}\) by \eqref{eq:unb-elementary}, and density now gives
  \eqref{eq:unb-odd}.
\end{proof}

Choose a nonzero radial function \(\vartheta\in\C_c^\infty(\R^m)\), normalised by
\[
  \norm{\vartheta}_{\L^2_{\gamma_m}}=1.
\]
For \(a\in\R^m\), define
\begin{equation}\label{eq:def-unitary-translation}
  (U_a\vartheta)(u)
  :=
  \exp\left(\frac{a\cdot u}{2}-\frac{|a|^2}{4}\right)
  \vartheta(u-a).
\end{equation}
The operator \(U_a\) is the unitary implementation on \(\L^2_{\gamma_m}\) of the Cameron--Martin translation by \(a\). The identity
\[
  e^{a\cdot u-|a|^2/2}\varrho_m(u)=\varrho_m(u-a),
  \qquad
  \varrho_m(u):=(2\pi)^{-m/2}e^{-|u|^2/2},
\]
and the change of variables \(q=u-a\) give the exact unitarity relation
\begin{equation}\label{eq:unb-unitarity}
  \norm{U_a\vartheta}_{\L^2_{\gamma_m}}
  =\norm{\vartheta}_{\L^2_{\gamma_m}}=1.
\end{equation}
Since the exponential factor in \eqref{eq:def-unitary-translation} never vanishes,
\begin{equation}\label{eq:unb-unitary-support}
  \supp(U_a\vartheta)=a+\supp\vartheta.
\end{equation}
Moreover, with \(U_a\) acting componentwise on vector-valued functions,
\[
  \nabla U_a\vartheta
  =U_a\left(\nabla\vartheta+\frac a2\vartheta\right).
\]
Radiality makes \(\vartheta(q)\nabla\vartheta(q)\) odd. Hence \(\int_{\R^m}\vartheta\nabla\vartheta\dd\gamma_m=0\), and expansion of the square
yields
\begin{equation}\label{eq:unb-unitary-gradient-exact}
  \norm{\nabla U_a\vartheta}_{\L^2_{\gamma_m}}^2
  =
  \norm{\nabla\vartheta}_{\L^2_{\gamma_m}}^2
  {}+\frac{|a|^2}{4}\norm{\vartheta}_{\L^2_{\gamma_m}}^2.
\end{equation}

\begin{lemma}[\(\L^2(\gamma)\)-normalised translated packet]
  \label{lem:unb-translated}
  For \(R\ge2\), \(0<\delta\le1\), and \(\tau\in\Sph^{m-1}\),
  \begin{align}
    \norm{(U_{R\tau}\vartheta)(u)\psi_\delta(w)}
    _{\H^1_\gamma(\R^d)}^2
     & \le C(\delta^{-1}+R^2\delta),
    \label{eq:unb-H1}                \\
    \norm{w(U_{R\tau}\vartheta)(u)\psi_\delta(w)}
    _{\H^{-1}_\gamma(\R^d)}^2
     & \le C\delta^5.
    \label{eq:unb-Hminus}
  \end{align}
\end{lemma}

\begin{proof}
  Product structure, \eqref{eq:unb-unitarity}, and the exact Gaussian translation cost \eqref{eq:unb-unitary-gradient-exact} give
  \begin{align*}
    \norm{(U_{R\tau}\vartheta)\psi_\delta}_{\H^1_\gamma}^2
     & =
    \norm{U_{R\tau}\vartheta}_{\L^2_{\gamma_m}}^2
    \left(
    \norm{\psi_\delta}_{\L^2_{\gamma_1}}^2
    {}+\norm{\psi_\delta'}_{\L^2_{\gamma_1}}^2
    \right)                                              \\
     & \quad+
    \norm{\nabla U_{R\tau}\vartheta}_{\L^2_{\gamma_m}}^2
    \norm{\psi_\delta}_{\L^2_{\gamma_1}}^2               \\
     & \le C\bigl(\delta+\delta^{-1}+(1+R^2)\delta\bigr) \\
     & \le C(\delta^{-1}+R^2\delta).
  \end{align*}
  This proves \eqref{eq:unb-H1}. Applying Lemma~\ref{lem:unb-odd} with \(\theta=U_{R\tau}\vartheta\), whose \(\L^2_{\gamma_m}\)-norm is one by
  \eqref{eq:unb-unitarity}, gives
  \[
    \norm{w(U_{R\tau}\vartheta)\psi_\delta}_{\H^{-1}_\gamma}
    \le C\delta^{5/2}.
  \]
  Squaring proves \eqref{eq:unb-Hminus}.
\end{proof}

\subsection{A supporting quadratic cap}

The next lemma extracts the two geometric inputs from an enclosing ball centred at a fixed interior point and touching the domain at a farthest
boundary point. At that contact point, the sphere separates quadratically from its tangent plane. Together with the \(\C^{1,1}\) upper bound, the
same separation produces a tangential direction in which the normal rotates coherently on a set of measure comparable to \(h^{(d-1)/2}\). No
pointwise lower curvature bound for \(\partial\Omega\) is used.

\begin{lemma}[Supporting-cap geometry]\label{lem:unb-cap}
  Let \(\Omega\subset\R^d\) be a bounded \(\C^{1,1}\) domain, and recall the decomposition \eqref{eq:decomp_variables}. After a rigid motion,
  there are \(\mathcal R,r_0,a_{\rm ch},c_0,C_0>0\) and a local chart \(\Phi\in \C^{1,1}(B_{r_0}^m)\) such that
  \begin{equation*}
    \overline\Omega\subset\overline B(\mathcal R e_d,\mathcal R),
  \end{equation*}
  and in the neighbourhood \(Q_{\rm ch}:=B_{r_0}^m\times(-a_{\rm ch},a_{\rm ch})\) we have
  \begin{equation}\label{eq:unb-supporting-ball}
    \begin{aligned}
      \Omega\cap Q_{\rm ch}
       & = \{(z,y)\in Q_{\rm ch}:y>\Phi(z)\}, \\
      \partial \Omega\cap Q_{\rm ch}
       & = \{(z,\Phi(z))\in Q_{\rm ch}\},
    \end{aligned}
  \end{equation}
  and
  \begin{equation}\label{eq:unb-cap-geometry}
    c_0|z|^2\le\Phi(z)\le C_0|z|^2,
    \qquad |\nabla\Phi(z)|\le C_0|z|.
  \end{equation}
  In these coordinates, there exist \(h_{\rm ch},c_{\rm rot},C_{\rm lay}>0\) and a unit vector \(\tau\in\Sph^{m-1}\) such that, for every
  \(0<h<h_{\rm ch}\),
  \begin{equation}\label{eq:unb-layer}
    \bigl|\Omega\cap\{0<y<2h\}\bigr|
    \le C_{\rm lay}h^{(d+1)/2},
  \end{equation}
  and there exists an open ball \(G_h\subset B_{r_0}^m\) satisfying
  \begin{equation}\label{eq:unb-direction-set}
    G_h\subset\{\Phi<h/2\},\qquad
    |G_h|\ge c_{\rm rot}h^{(d-1)/2},\qquad
    \tau\cdot\nabla\Phi\ge c_{\rm rot}\sqrt h
    \quad\text{on }G_h.
  \end{equation}
  More precisely, there is \(C_{\rm ch}>0\) such that for every \(0<h<h_{\rm ch}\) we have
  \begin{equation}\label{eq:unb-layer-chart}
    \overline\Omega\cap\{0\le y<2h\}
    \subset B_{C_{\rm ch}\sqrt h}(0)
    \Subset Q_{\rm ch}.
  \end{equation}
\end{lemma}

\begin{proof}
  \emph{Step 1: the supporting cap.} Fix \(x_\ast\in\Omega\), let
  \[
    \mathcal R:=\max_{x\in\overline\Omega}|x-x_\ast|,
  \]
  and choose a maximiser. Translate this farthest point to the origin and rotate so that \(x_\ast=\mathcal R e_d\). Then
  \(\overline\Omega\subset\overline B(\mathcal R e_d,\mathcal R)\). The contact point lies on \(\partial\Omega\), since an interior point
  cannot lie on the enclosing sphere. The restriction of \(x\mapsto|x-\mathcal R e_d|^2\) to \(\partial\Omega\) has a maximum there, so its
  differential vanishes on \(T_0\partial\Omega\). Thus \(T_0\partial\Omega=\{y=0\}\), \(\Phi(0)=0\), and \(\nabla\Phi(0)=0\). The local graph
  representation theorem for \(\C^{1,1}\) boundaries permits us, after decreasing \(r_0\) and \(a_{\rm ch}\), to ensure that \(Q_{\rm ch}\)
  meets no other boundary sheet. Every graph point \((z,\Phi(z))\), \(z\in B_{r_0}^m\), then lies in \(Q_{\rm ch}\), and \(\Omega\cap Q_{\rm
      ch}\) is exactly one side of this graph. It cannot be the lower side: that side contains \((0,-\varepsilon)\) for every sufficiently small
  \(\varepsilon>0\), whereas the enclosing ball is contained in \(\{y\ge0\}\). This proves \eqref{eq:unb-supporting-ball}.

  The supporting-ball construction is illustrated in Figure~\ref{fig:unb-supporting-ball}.

  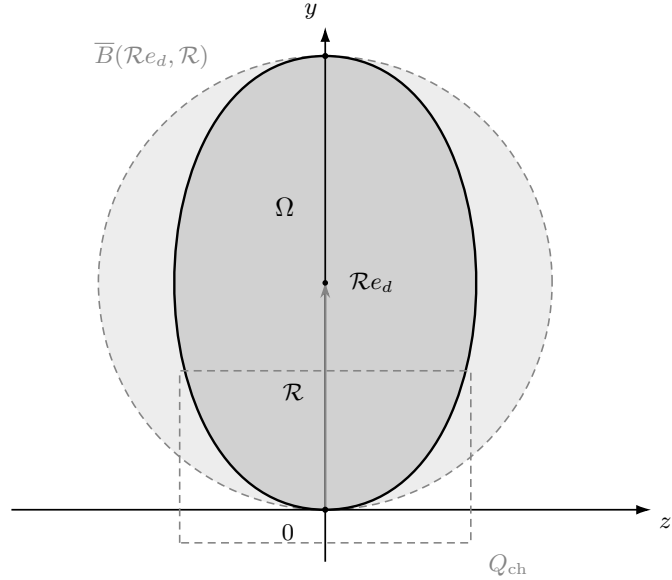
\begin{figure}[htbp]
    \centering
    \begin{tikzpicture}[
      font=\small,
      x=1.25cm,
      y=1.25cm,
      figaxis/.style={
      draw=black,
      line width=.65pt,
      -{Latex[length=2mm,width=1.3mm]}
      },
      figguide/.style={
          draw=black!48,
          densely dashed,
          line width=.6pt
        },
      fignote/.style={
          font=\footnotesize,
          text=black,
          align=center
        },
      figtitle/.style={
          font=\bfseries\small,
          text=black,
          align=center
        }
      ]
      \path[fill=black!7] (0,2.65) circle[radius=2.40];
      \draw[figguide] (0,2.65) circle[radius=2.40];
      \node[fignote,text=black!48,anchor=south east]
      at (-1.12,4.80)
      {\(\overline B(\mathcal R e_d,\mathcal R)\)};

      \path[fill=black!18,draw=black,line width=.9pt]
      (0,.25)
      .. controls (2.13,.25) and (2.13,5.05) .. (0,5.05)
      .. controls (-2.13,5.05) and (-2.13,.25) .. (0,.25)
      -- cycle;

      \draw[figaxis] (-3.32,.25) -- (3.45,.25)
      node[below right=-1pt,fignote] {\(z\)};
      \draw[figaxis] (0,-.30) -- (0,5.36)
      node[above left=-1pt,fignote] {\(y\)};

      \draw[draw=black!48,line width=.7pt,
      -{Stealth[length=2mm,width=1.3mm]}]
      (0,.25) -- (0,2.65);
      \fill[black] (0,2.65) circle[radius=1.05pt];
      \node[fignote,anchor=west] at (.15,2.66)
      {\(\mathcal R e_d\)};
      \node[fignote,anchor=east] at (-.12,1.50)
      {\(\mathcal R\)};

      \draw[figguide] (-1.54,-.10) rectangle (1.54,1.72);
      \node[fignote,anchor=north west,text=black!48]
      at (1.62,-.12) {\(Q_{\rm ch}\)};

      \fill[black] (0,.25) circle[radius=1.15pt];
      \fill[black] (0,5.05) circle[radius=1.15pt];
      \node[fignote,anchor=north east] at (-.22,.20) {\(0\)};
      \node[figtitle,anchor=east] at (-.22,3.45) {\(\Omega\)};
    \end{tikzpicture}
    \caption{An enclosing ball centred at a fixed interior point
      touches \(\overline\Omega\) at a farthest boundary point, here
      translated to the origin. Its centre lies on the normal axis, and the
      common tangent plane is \(y=0\). The dashed box indicates the single
      local chart used in the rest of Lemma~\ref{lem:unb-cap}.}
    \label{fig:unb-supporting-ball}
  \end{figure}

  Every graph point lies in the ball, so
  \[
    |z|^2+(\Phi(z)-\mathcal R)^2\le\mathcal R^2,
    \qquad
    \Phi(z)\ge\frac{|z|^2}{2\mathcal R}.
  \]
  The other two bounds in \eqref{eq:unb-cap-geometry} follow from the Lipschitz continuity of \(\nabla\Phi\).

  \emph{Step 2: the layer volume.} Ball containment also gives \(|z|^2\le2\mathcal R y-y^2\) throughout \(\Omega\). Integrating this
  cross-sectional bound over \(0<y<2h\) proves \eqref{eq:unb-layer}. Indeed,
  \[
    \bigl|\Omega\cap\{0<y<2h\}\bigr|
    \le |B_1^m|(2\mathcal R)^{m/2}
    \int_0^{2h}y^{m/2}\dd y
    =C_{\rm lay}h^{(d+1)/2},
  \]
  where
  \[
    C_{\rm lay}
    :=\frac{|B_1^m|(2\mathcal R)^{m/2}2^{m/2+1}}{m/2+1}.
  \]
  We now verify the one-chart assertion quantitatively. If \(x=(z,y)\in\overline\Omega\) and \(0\le y<2h\), then
  \[
    |x|^2=|z|^2+y^2\le2\mathcal R y\le4\mathcal R h.
  \]
  Set \(C_{\rm ch}=2\sqrt{\mathcal R}\). Since \(Q_{\rm ch}\) is an open neighbourhood of the origin, we may choose \(0<h_{\rm ch}\le1\) so
  that \(\overline B_{C_{\rm ch}\sqrt{h_{\rm ch}}}(0)\Subset Q_{\rm ch}\). The same inclusion then holds for every \(0<h<h_{\rm ch}\), which
  proves \eqref{eq:unb-layer-chart}.

  The resulting thin-layer geometry is shown in Figure~\ref{fig:unb-thin-layer}.

  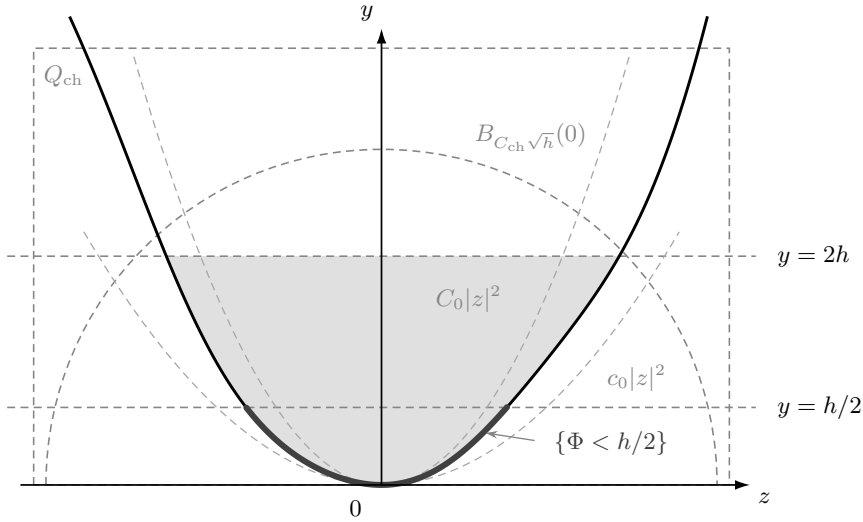
\begin{figure}[htbp]
  \centering
  \begin{tikzpicture}[
    font=\small,
    x=1.25cm,
    y=1.25cm,
    figaxis/.style={
      draw=black,
      line width=.65pt,
      -{Latex[length=2mm,width=1.3mm]}
    },
    figguide/.style={
      draw=black!48,
      densely dashed,
      line width=.6pt
    },
    fignote/.style={
      font=\footnotesize,
      text=black,
      align=center
    }
  ]
    % Chart box restricted to y >= 0.
    \draw[figguide] (-3.68,0) rectangle (3.68,4.62);
    \node[fignote,anchor=north west,text=black!48]
      at (-3.68,4.54) {\(Q_{\rm ch}\)};

    \begin{scope}
      \clip (-3.67,0) rectangle (3.67,2.42);
      \path[fill=black!12]
        plot[domain=-3.30:3.45,variable=\z,samples=181]
          ({\z},{.43*\z*\z+.05*\z*\z*sin(100*\z)})
        -- (3.45,5.20) -- (-3.30,5.20) -- cycle;
    \end{scope}

    \draw[figguide] (-3.55,0)
      arc[start angle=180,end angle=0,radius=3.55];
    \node[fignote,text=black!48,anchor=south west]
      at (.88,3.42) {\(B_{C_{\rm ch}\sqrt h}(0)\)};

    \draw[figguide] (-3.96,2.42) -- (3.96,2.42);
    \node[fignote,anchor=west] at (4.08,2.42) {\(y=2h\)};
    \draw[figguide] (-3.96,.82) -- (3.96,.82);
    \node[fignote,anchor=west] at (4.08,.82) {\(y=h/2\)};

    \draw[draw=black!34,line width=.5pt,densely dashed]
      plot[domain=-3.15:3.15,variable=\z,samples=101]
        ({\z},{.27*\z*\z});
    \draw[draw=black!34,line width=.5pt,densely dashed]
      plot[domain=-2.62:2.62,variable=\z,samples=101]
        ({\z},{.66*\z*\z});
    \node[fignote,text=black!48,anchor=south east]
      at (1.38,1.72) {\(C_0|z|^2\)};
    \node[fignote,text=black!48,anchor=north west]
      at (2.28,1.38) {\(c_0|z|^2\)};

    \draw[draw=black,line width=1.1pt]
      plot[domain=-3.30:3.45,variable=\z,samples=181]
        ({\z},{.43*\z*\z+.05*\z*\z*sin(100*\z)});
    \draw[draw=black!75,line width=2.25pt,line cap=round]
      plot[domain=-1.43:1.33,variable=\z,samples=101]
        ({\z},{.43*\z*\z+.05*\z*\z*sin(100*\z)});

    \draw[figaxis] (-3.82,0) -- (3.90,0)
      node[below right=-1pt,fignote] {\(z\)};
    \draw[figaxis] (0,0) -- (0,4.83)
      node[above left=-1pt,fignote] {\(y\)};
    \node[fignote,anchor=north east] at (-.10,-.06) {\(0\)};

    \node[fignote,text=black!77,anchor=west]
      at (1.72,.44) {\(\{\Phi<h/2\}\)};
    \draw[
      draw=black!48,
      line width=.5pt,
      -{Stealth[length=1.6mm,width=1mm]}
    ]
      (1.62,.47) -- (1.10,.55);
  \end{tikzpicture}
  \caption{The boundary graph stays
    between quadratic barriers. Consequently, the bottom layer has tangential
    scale \(\sqrt h\), lies in \(B_{C_{\rm ch}\sqrt h}(0)\), and is contained
    in one chart. The thicker boundary segment is the region
    \(\{\Phi<h/2\}\). Only the half-plane \(y\ge 0\) is shown.}
  \label{fig:unb-thin-layer}
\end{figure}

  \emph{Step 3: a large ball with coherently rotating normals.} Fix any \(\tau\in\Sph^{m-1}\), and let \(L>0\) be a Lipschitz constant for
  \(\nabla\Phi\). Choose \(0<\eta\le1\) so that \(L\eta\le c_0/2\), and then choose \(\beta>0\) so that
  \[
    C_0(1+\eta)^2\beta^2<\frac12.
  \]
  Reduce \(h_{\rm ch}\), if necessary, so that
  \[
    (1+\eta)\beta\sqrt{h_{\rm ch}}<r_0,
  \]
  and put \(r=\beta\sqrt h\). Then \((1+\eta)r<r_0\) for every \(0<h<h_{\rm ch}\). Since \(\Phi(0)=0\) and \(\Phi(r\tau)\ge c_0r^2\), the
  mean-value theorem applied to \(t\mapsto\Phi(t\tau)\) gives \(t_h\in(0,r)\) such that
  \[
    \tau\cdot\nabla\Phi(t_h\tau)
    =\frac{\Phi(r\tau)-\Phi(0)}r
    \ge c_0r.
  \]
  Define
  \[
    G_h:=B_{\eta r}^m(t_h\tau).
  \]
  For every \(z\in G_h\), Lipschitz continuity gives
  \[
    \tau\cdot\nabla\Phi(z)
    \ge c_0r-L|z-t_h\tau|
    \ge \frac{c_0}{2}r
    =\frac{c_0\beta}{2}\sqrt h.
  \]
  Moreover, \(|z|\le(1+\eta)r\), and hence
  \[
    \Phi(z)\le C_0|z|^2
    \le C_0(1+\eta)^2r^2<\frac h2.
  \]
  Finally,
  \[
    |G_h|=|B_1^m|(\eta\beta)^m h^{m/2}.
  \]
  Thus both bounds in \eqref{eq:unb-direction-set} hold with
  \[
    c_{\rm rot}
    :=\min\left\{\frac{c_0\beta}{2},
    |B_1^m|(\eta\beta)^m\right\}>0.
  \]
  This proves \eqref{eq:unb-direction-set}, including the case \(m=1\), with one direction \(\tau\) fixed for all \(h\).

\end{proof}

For a lower graph with \(\Omega\) above it,
\[
  n_\Omega
  (z,\Phi(z))=\frac{(\nabla\Phi(z),-1)}{\sqrt{1+|\nabla\Phi(z)|^2}},
  \qquad \dd S=\sqrt{1+|\nabla\Phi(z)|^2}\dd z,
\]
and therefore
\begin{equation}\label{eq:unb-graph-flux}
  |v\cdot n_\Omega
  (z,\Phi(z))|^p\dd S
  =
  \frac{|u\cdot\nabla\Phi(z)-w|^p}
  {(1+|\nabla\Phi(z)|^2)^{(p-1)/2}}\dd z.
\end{equation}

\subsection{Normalisation and bulk estimates}

Fix \(\chi\in \C_c^\infty(-1,2)\), with \(\chi=1\) on \([0,1/2]\). For small \(h,\delta>0\), large \(R\), and an amplitude \(A_h\), define
\begin{equation}\label{eq:unb-f}
  f_h(z,y,u,w):=
  A_h\chi(y/h)(U_{R\tau}\vartheta)(u)\psi_\delta(w),
\end{equation}
where the direction \(\tau\) is fixed by Lemma~\ref{lem:unb-cap}. Up to the common spatial factor, the normal-velocity, Gaussian translation,
and transport costs are
\[
  \delta^{-1}+R^2\delta+h^{-2}\delta^5
  =
  \delta^{-1}\left[
    1+(R\delta)^2+\left(\frac{\delta^3}{h}\right)^2
    \right].
\]
We balance them by imposing \(R\delta=1\) and \(\delta^3=h\), and choose \(A_h\) so that the resulting bulk energy is uniformly
bounded:
\begin{equation}\label{eq:unb-parameters}
  \delta=h^{1/3},
  \qquad R=h^{-1/3},
  \qquad A_h^2=h^{-(d+1)/2+1/3}.
\end{equation}

Choose \(0<h_0<\min\{h_{\rm ch},1/8\}\), to be decreased once more when the fixed constants in the trace estimate have been identified, and
henceforth restrict to \(0<h<h_0\). Then \(R\ge2\), \(0<\delta\le1\), and \(h<h_{\rm ch}\). These choices give the three balances used below:
\begin{equation}\label{eq:unb-scale-checkpoint}
  R^2\delta=\delta^{-1}=h^{-2}\delta^5=h^{-1/3},
  \qquad
  A_h^2h^{(d+1)/2}\delta^{-1}=1,
  \qquad
  R\sqrt h=h^{1/6}.
\end{equation}
The first identity balances the tangential, normal, and transport costs. The second makes the bulk energy of order at most one, while the third
is the normal-flux scale imposed by the supporting cap. For later use, set
\[
  \varepsilon_h:=R\sqrt h=h^{1/6}.
\]
The boundary mass factor produced by the same normalisation is
\begin{equation}\label{eq:unb-boundary-mass-scale}
  A_h^2\delta h^{(d-1)/2}
  =h^{-1/3}
  =\varepsilon_h^{-2}.
\end{equation}

\begin{lemma}[Bulk estimates for the high-tangential-velocity packet]
  \label{lem:unb-bulk}
  The functions \(f_h\) defined by \eqref{eq:unb-f}--\eqref{eq:unb-parameters} satisfy
  \begin{equation}
    \sup_h\mathcal E^\gamma_\Omega(f_h)<\infty,
  \end{equation}
  and
  \begin{align}
    \supp_xf_h
     & \subset\overline\Omega\cap\{0\le y<2h\},
    \label{eq:unb-spatial-support}              \\
    \supp_vf_h
     & \subset
    \{u=R\tau+q,\quad q\in\supp\vartheta,\quad |q|\le C,
    \quad |w|<\rho\delta\}.
    \label{eq:velocity-support}
  \end{align}
\end{lemma}

\begin{proof}
  The velocity support follows from \eqref{eq:unb-unitary-support} and the support of \(\psi_\delta\). The enclosing-ball containment gives
  \(y\ge0\) on \(\overline\Omega\); hence \(\supp\chi\subset(-1,2)\) implies \eqref{eq:unb-spatial-support}, and \eqref{eq:unb-layer-chart}
  places this support inside \(Q_{\rm ch}\).

  Put
  \[
    P_h(u,w):=(U_{R\tau}\vartheta)(u)\psi_\delta(w).
  \]
  Since \(f_h\) is independent of \(z\),
  \[
    v\cdot\nabla_xf_h
    =A_hh^{-1}w\chi'(y/h)P_h(u,w).
  \]
  Lemma~\ref{lem:unb-translated} and the layer bound \eqref{eq:unb-layer} therefore give the single energy estimate
  \begin{align*}
    \mathcal E^\gamma_\Omega(f_h)
     & \lesssim
    A_h^2h^{(d+1)/2}
    \left(\delta^{-1}+R^2\delta+h^{-2}\delta^5\right) \\
     & \lesssim
    A_h^2h^{(d+1)/2}h^{-1/3}
    =1,
  \end{align*}
  where the last equality follows from \eqref{eq:unb-parameters}.
\end{proof}

\subsection{Asymptotics and optimality for the weighted boundary
  functional}

\begin{proof}[Proof of Theorem~\ref{thm:unbounded}]
  Let \(f_h\) be the packets defined in \eqref{eq:unb-f}--\eqref{eq:unb-parameters}. Their energies are uniformly bounded by
  Lemma~\ref{lem:unb-bulk}; it remains to identify one boundary mass scale and one normal-flux scale.

  Put
  \[
    P_h(u,w):=(U_{R\tau}\vartheta)(u)\psi_\delta(w),
    \qquad
    Z_h:=\{z\in B_{r_0}^m:\ \Phi(z)/h\in\supp\chi\}.
  \]
  On \(Z_h\), \eqref{eq:unb-cap-geometry} and the support of \(\chi\) give \(0\le\Phi(z)<2h\). Hence \eqref{eq:unb-cap-geometry} gives the
  following
  bounds with the fixed constant
  \[
    C_\Phi:=C_0\sqrt{\frac{2}{c_0}}.
  \]
  \[
    |Z_h|\le Ch^{(d-1)/2},
    \qquad
    |\nabla\Phi(z)|\le C_\Phi\sqrt h
    \quad\text{for }z\in Z_h.
  \]
  If \(P_h(u,w)\ne0\), then
  \[
    u=R\tau+q,\qquad |q|\le B_\vartheta,
    \qquad |w|\le\rho\delta,
    \qquad
    B_\vartheta:=\sup_{q\in\supp\vartheta}|q|<\infty.
  \]
  Set
  \[
    C_{\rm gr}:=C_\Phi(1+B_\vartheta)+\rho.
  \]
  Since \(h<h_0<1\), throughout the boundary support we have
  \begin{equation}\label{eq:unb-master-flux-upper}
    |u\cdot\nabla\Phi(z)-w|
    \le C_\Phi(R+B_\vartheta)\sqrt h+\rho\delta
    \le C_{\rm gr}\varepsilon_h.
  \end{equation}
  Decrease \(h_0\), once and for all, so that
  \[
    C_{\rm gr}h_0^{1/6}<1,
    \qquad
    C_\Phi B_\vartheta h_0^{1/3}+\rho h_0^{1/6}
    \le \frac{c_{\rm rot}}2.
  \]
  The first condition gives \(C_{\rm gr}\varepsilon_h<1\). It follows that \(\omega_p(|v\cdot n_\Omega|)=|v\cdot n_\Omega|^p\) on the boundary
  support. Hence \eqref{eq:unb-graph-flux} gives the exact identity for the weighted boundary functional
  \begin{equation}\label{eq:unb-master-trace}
    \mathcal T^\gamma_{p,\Omega}(f_h)
    =
    A_h^2\int_{Z_h}
    \frac{\chi(\Phi(z)/h)^2}
    {(1+|\nabla\Phi(z)|^2)^{(p-1)/2}}
    \int_{\R^d}
    |u\cdot\nabla\Phi(z)-w|^pP_h(u,w)^2
    \dd\gamma_d(v)\dd z.
  \end{equation}
  Exact unitarity and \eqref{eq:unb-elementary} give
  \begin{equation}\label{eq:unb-packet-mass}
    \int_{\R^d}P_h(v)^2\dd\gamma_d(v)
    =
    \norm{U_{R\tau}\vartheta}_{\L^2_{\gamma_m}}^2
    \norm{\psi_\delta}_{\L^2_{\gamma_1}}^2
    \asymp\delta.
  \end{equation}
  Equations \eqref{eq:unb-master-flux-upper}--\eqref{eq:unb-packet-mass} therefore yield
  \[
    \mathcal T^\gamma_{p,\Omega}(f_h)
    \le
    C A_h^2\delta h^{(d-1)/2}\varepsilon_h^p.
  \]

  On the other hand, \eqref{eq:unb-direction-set} gives
  \[
    G_h\subset Z_h,\qquad
    |G_h|\ge c_{\rm rot}h^{(d-1)/2},\qquad
    \chi(\Phi/h)=1,\qquad
    \tau\cdot\nabla\Phi\ge c_{\rm rot}\sqrt h
    \quad\text{on }G_h.
  \]
  Uniformly on \(G_h\) and the velocity support,
  \begin{align*}
    u\cdot\nabla\Phi-w
     & \ge
    R\tau\cdot\nabla\Phi
    -|q|\,|\nabla\Phi|-|w|                                   \\
     & \ge
    c_{\rm rot}R\sqrt h-C_\Phi B_\vartheta\sqrt h-\rho\delta \\
     & =
    \varepsilon_h
    \bigl(c_{\rm rot}-C_\Phi B_\vartheta h^{1/3}
    -\rho h^{1/6}\bigr)
    \ge \frac{c_{\rm rot}}2\varepsilon_h.
  \end{align*}
  The last inequality is exactly the second condition imposed on \(h_0\). The denominator in \eqref{eq:unb-master-trace} is uniformly bounded
  above on \(G_h\). Thus \eqref{eq:unb-packet-mass} gives the matching lower estimate
  \[
    \mathcal T^\gamma_{p,\Omega}(f_h)
    \ge
    c A_h^2\delta h^{(d-1)/2}\varepsilon_h^p.
  \]
  We have proved
  \[
    \mathcal T^\gamma_{p,\Omega}(f_h)
    \asymp
    A_h^2\delta h^{(d-1)/2}\varepsilon_h^p
    =
    \varepsilon_h^{p-2}
    =
    h^{(p-2)/6},
  \]
  where \eqref{eq:unb-boundary-mass-scale} was used to obtain \(\varepsilon_h^{p-2}\). This is \eqref{eq:unb-main-scale} and proves all
  conclusions of the theorem. Together with Proposition~\ref{prop:quadratic-gaussian}, it also shows directly that \(p=2\) is the smallest
  admissible exponent in the scale \(\omega_p\).
\end{proof}

\section{\texorpdfstring{Bounded-velocity models: counterexamples from
    rapid normal variation and the sharp threshold} {Bounded-velocity models: counterexamples from rapid normal variation and the sharp
    threshold}}
\label{sec:bounded}

\subsection{The common spatial geometry and the counterexample from rapid normal variation}

Fix \(d\ge2\), set \(m=d-1\), fix \(0<\alpha<1\), and put \(q=1+\alpha\). Consider
\[
  \Omega_{\alpha,d}
  =\{(z,y)\in\R^m\times\R:\ |z|^q+|y-1|^q<1\}.
\]
The defining function \(F(z,y)=|z|^q+|y-1|^q-1\) is strictly convex and belongs to \(\C_{\mathrm{loc}}^{1,\alpha}(\R^d)\). Its gradient does not
vanish on \(F=0\), so the \(\C^{1,\alpha}\) implicit-function theorem gives local boundary charts; compactness of the boundary yields a finite
atlas. Thus \(\Omega_{\alpha,d}\) is strictly convex with \(\C^{1,\alpha}\) boundary. Near \((0,0)\), its lower boundary is
\begin{equation}\label{eq:bnd-Phi}
  y=\Phi_q(z):=1-(1-|z|^q)^{1/q}.
\end{equation}
For small \(|z|\),
\begin{equation}\label{eq:bnd-geometry}
  \Phi_q(z)\asymp_q|z|^q,
  \qquad
  |\nabla\Phi_q(z)|\asymp_q|z|^{q-1}.
\end{equation}
More precisely,
\[
  \nabla\Phi_q(z)=|z|^{q-2}z
  (1-|z|^q)^{-\alpha/q}.
\]
The standard \(\alpha\)-H\"older estimate for \(\xi\mapsto |\xi|^{q-2}\xi\), applied in dimensions \(d-1\) and \(1\) to the two blocks of
\(\nabla F\), shows that \(\nabla F\) is globally \(\alpha\)-H\"older. In particular, \(F\in\C_{\mathrm{loc}}^{1,\alpha}(\R^d)\), which directly
yields the global boundary regularity asserted above. This regularity is exact: for every \(\alpha<\beta\le1\),
\[
  \frac{|\nabla\Phi_q(te_1)-\nabla\Phi_q(0)|}{|t|^\beta}
  \asymp |t|^{\alpha-\beta}\longrightarrow\infty
  \qquad \text{as } t \to 0.
\]
If the hypersurface were \(\C^{1,\beta}\), then, since its tangent plane at the lower tip is horizontal, the \(\C^{1,\beta}\)
implicit-function theorem would represent it in these same coordinates by a \(\C^{1,\beta}\) graph. The preceding divergence rules this out.
Therefore
\begin{equation}\label{eq:bnd-exact-regularity}
  \partial\Omega_{\alpha,d}\in \C^{1,\alpha},
  \qquad
  \partial\Omega_{\alpha,d}\notin \C^{1,\beta}
  \quad\text{for every }\alpha<\beta\le1.
\end{equation}

\begin{theorem}[Counterexamples for the bounded-support Euclidean model]
  \label{thm:bounded-euclidean}
  Let \(d\ge2\), \(0<\alpha<1\), and \(L>0\). There are \(\ell_0>0\) and a family
  \[
    (f_\ell)_{0<\ell<\ell_0}
    \subset\C_c^\infty(
    \overline{\Omega_{\alpha,d}}\times\R^d)
  \]
  with a common compact velocity support in \(B_L\), such that, simultaneously for \(\mu\in\{\mathcal L^d,\gamma_d\}\) and every \(1\le
  p<\infty\),
  \begin{equation}\label{eq:bnd-invariant-scale}
    \sup_{0<\ell<\ell_0}
    \mathcal E^\mu_{\Omega_{\alpha,d}}(f_\ell)<\infty,
    \qquad
    \mathcal T^\mu_{p,\Omega_{\alpha,d}}(f_\ell)
    \asymp\ell^{\alpha(p+1)-1}
    \quad\text{as }\ell\downarrow0.
  \end{equation}
  The comparison constants may depend on \(p,\mu,L\), but not on \(\ell\); the family and \(\ell_0\) are independent of \(p\) and \(\mu\).
  Consequently, for each fixed \(p\), the corresponding weighted boundary functional diverges as \(\ell\downarrow0\) if and only if
  \(\alpha<\alpha_p=1/(p+1)\). For every \(\ell\), the incoming and outgoing weighted boundary functionals are equal.
\end{theorem}

\begin{theorem}[Counterexamples for the spherical velocity model]
  \label{thm:spherical}
  Let \(d\ge2\) and \(0<\alpha<1\). There are \(\ell_0>0\) and a family
  \[
    (f_\ell^\Sp)_{0<\ell<\ell_0}
    \subset
    \C^\infty(
    \overline{\Omega_{\alpha,d}}\times\Sph^{d-1})
  \]
  such that, simultaneously for every \(1\le p<\infty\),
  \begin{equation}\label{eq:spherical-rough-scale}
    \sup_{0<\ell<\ell_0}
    \mathcal E^\sigma_{\Omega_{\alpha,d}}(f_\ell^\Sp)<\infty,
    \qquad
    \mathcal T^\sigma_{p,\Omega_{\alpha,d}}(f_\ell^\Sp)
    \asymp\ell^{\alpha(p+1)-1}
    \quad\text{as }\ell\downarrow0.
  \end{equation}
  The comparison constants may depend on \(p\), but not on \(\ell\); the family and \(\ell_0\) are independent of \(p\). Consequently, for each
  fixed \(p\), the corresponding weighted boundary functional diverges as \(\ell\downarrow0\) if and only if \(\alpha<\alpha_p=1/(p+1)\). For
  every \(\ell\), the incoming and outgoing weighted boundary functionals are equal.
\end{theorem}

\subsection{The secant--tangent slope gap and the common scaling}

Retain \(m=d-1\) and \(q=1+\alpha\). For \(0<\ell\ll1\), set
\begin{equation}\label{eq:rough-common-scales}
  t:=\ell^{q-1}=\ell^\alpha,
  \qquad
  H:=\ell^q=\ell t,
  \qquad
  A_\ell:=\ell^{-\frac{m+1}{2}}.
\end{equation}
Fix \(r_0>0\), and put
\[
  Z^*:=r_0e_1,
  \qquad
  \omega^*:=\frac{r_0^{q-1}}q.
\]
The number \(\omega^*\) is the slope of the secant from the origin to \(Z^*\) for the limiting boundary graph \(Z\mapsto |Z|^q/q\), whereas the
tangent slope at \(Z^*\) is \(r_0^{q-1}\). Their gap
\begin{equation}\label{eq:rough-secant-tangent-gap}
  r_0^{q-1}-\omega^*
  =\frac{q-1}{q}r_0^{q-1}>0
\end{equation}
will be the limiting normal flux.

Choose \(\psi,\chi\in\C_c^\infty(\R)\), with \(\psi\) supported in a sufficiently small neighbourhood of the positive number \(\omega^*\) and
\(\chi\) supported near \(0\), so that
\[
  \psi(\omega^*)\chi(0)\ne0.
\]

\subsection{The spherical packet}

\begin{proof}[Proof of Theorem~\ref{thm:spherical}]
  Fix \(0<R<1\), and use the spherical chart
  \[
    \mathcal V(\xi,w)
    :=\bigl(u_1(\xi,w),\xi,w\bigr),
    \qquad
    u_1(\xi,w):=\sqrt{1-|\xi|^2-w^2},
  \]
  on
  \[
    D_R:=\{(\xi,w)\in\R^{d-2}\times\R:
    |\xi|^2+w^2<R^2\}.
  \]
  On every \(K\Subset D_R\), the chart density is
  \[
    \dd\sigma(\mathcal V(\xi,w))
    =u_1(\xi,w)^{-1}\dd\xi\,\dd w,
  \]
  and the coordinate metric is uniformly elliptic. Consequently, for functions supported in \(K\), the spherical \(\H^1\)-norm is bounded by a
  fixed multiple of the Euclidean coordinate \(\H^1\)-norm.

  Choose
  \[
    \eta_{\Sp}\in\C_c^\infty(\R^{d-2}),
    \qquad
    \eta_{\Sp}(0)\ne0,
  \]
  with support sufficiently close to the origin. When \(d=2\), the \(\xi\)-variable is absent, and we put \(\eta_{\Sp}\equiv1\). After
  decreasing \(\ell_0\), the support of \(\eta_{\Sp}(\xi)\psi(w/t)\) lies in one fixed \(K\Subset D_R\) for \(0<\ell<\ell_0\).

  For \(x=(z,y)\), with \(z=(z_1,z')\in\R^{d-1}\), define in the chart
  \begin{equation}\label{eq:rough-spherical-packet}
    \widetilde g_\ell(z,y,\xi,w)
    :=A_\ell\eta_{\Sp}(\xi)\psi(w/t)
    \chi\!\left(\frac{s(z,y,\xi,w)}H\right),
    \qquad
    s(z,y,\xi,w):=
    y-\frac{w}{u_1(\xi,w)}z_1.
  \end{equation}
  Let \(g_\ell^\Sp\) be the corresponding function on \(\Sph^{d-1}\), extended by zero outside the chart. This extension is smooth because its
  velocity support is compactly contained in the chart. Writing \(u(\xi,w):=(u_1(\xi,w),\xi)\), one has the exact cancellation
  \begin{equation}\label{eq:rough-spherical-invariance}
    \bigl(u(\xi,w)\cdot\nabla_z+w\partial_y\bigr)s
    =u_1\left(-\frac w{u_1}\right)+w=0,
    \qquad
    v\cdot\nabla_xg_\ell^\Sp=0.
  \end{equation}

  We next locate the packet. On its support,
  \[
    |w|\le Ct,
    \qquad
    |s|\le CH,
    \qquad
    u_1\ge c>0,
  \]
  and therefore
  \[
    0\le y\le C(t|z|+H).
  \]
  Since \(\Omega_{\alpha,d}\) is bounded, this shows that \(y\to0\) uniformly on the support. For small \(\ell\), the support is
  therefore contained in the lower graph neighbourhood, where \(y\ge\Phi_q(z)\ge c|z|^q\). Hence
  \[
    |z|^q\le C(t|z|+H).
  \]
  Writing \(|z|=\ell r\) and using \eqref{eq:rough-common-scales} gives \(r^q\le C(r+1)\), so
  \begin{equation}\label{eq:rough-spherical-support}
    |z|\le C\ell,
    \qquad
    0\le y\le CH.
  \end{equation}
  The spatial support thus has measure \(O(H\ell^m)\). The same argument on \(y=\Phi_q(z)\) confines the boundary support to \(|z|\le C\ell\),
  and only the lower graph meets it.

  Set \(\lambda(\xi,w):=w/u_1(\xi,w)\). On the fixed chart and packet support,
  \[
    |\lambda|\le Ct,
    \qquad
    |\nabla_{\xi,w}\lambda|\le C.
  \]
  Together with \eqref{eq:rough-spherical-support}, this implies
  \[
    |\nabla_{\xi,w}(s/H)|
    \le C\frac{\ell}{H}
    =Ct^{-1}.
  \]
  The cut-off \(\psi(w/t)\) has the same derivative cost, and hence
  \[
    |\widetilde g_\ell|
    +t|\nabla_{\xi,w}\widetilde g_\ell|
    \le CA_\ell.
  \]
  The coordinate velocity support has measure \(O(t)\). Uniform ellipticity of the chart and \eqref{eq:rough-spherical-support} now give
  \begin{align}
    \int_{\Omega_{\alpha,d}}
    \norm{g_\ell^\Sp(x,\cdot)}_{\H^1(\Sph^{d-1})}^2\dd x
     & \le
    C A_\ell^2(H\ell^m)t(1+t^{-2})\notag \\
     & \le C A_\ell^2H\ell^m t^{-1}
    =C. \label{eq:rough-spherical-energy}
  \end{align}
  Together with \eqref{eq:rough-spherical-invariance}, this proves the uniform graph-energy bound for the basic packet.

  It remains to estimate the weighted boundary functional. On the lower graph,
  \[
    |v\cdot n_{\Omega_{\alpha,d}}|^p\dd S
    =
    |u(\xi,w)\cdot\nabla\Phi_q(z)-w|^p
    (1+|\nabla\Phi_q(z)|^2)^{(1-p)/2}\dd z.
  \]
  Here \(\omega_p(|v\cdot n|)=|v\cdot n|^p\), since \(|v\cdot n|\le1\) on the sphere. On the boundary support,
  \[
    |\nabla\Phi_q(z)|+|w|\le Ct,
  \]
  so the flux is \(O(t)\). The \(z\)-support has measure \(O(\ell^m)\) and the coordinate velocity support has measure \(O(t)\); consequently,
  \begin{equation}\label{eq:rough-spherical-trace-upper}
    \mathcal T^\sigma_{p,\Omega_{\alpha,d}}(g_\ell^\Sp)
    \le C_p A_\ell^2\ell^m t^{p+1}.
  \end{equation}

  For the matching lower bound, put \(z=\ell Z\) and \(w=t\omega\), and set
  \[
    \overline u_1(\xi):=\sqrt{1-|\xi|^2},
    \qquad
    \overline u(\xi):=(\overline u_1(\xi),\xi).
  \]
  On the lower graph \(y=\Phi_q(z)\), using \(H=\ell t\), the rescaled argument of the characteristic cut-off is
  \begin{align*}
    \frac{s(\ell Z,\Phi_q(\ell Z),\xi,t\omega)}{H}
     & =
    \frac{\Phi_q(\ell Z)}{H}
    -\frac{t\omega}{u_1(\xi,t\omega)}
    \frac{\ell Z_1}{H} \\
     & =
    \frac{\Phi_q(\ell Z)}{H}
    -\omega\frac{Z_1}{u_1(\xi,t\omega)}
    =:P_\ell(Z,\xi,\omega).
  \end{align*}
  We likewise normalise the unnormalised boundary flux by setting
  \[
    G_\ell(Z,\xi,\omega)
    :=
    \frac{
      u(\xi,t\omega)\cdot\nabla\Phi_q(\ell Z)-t\omega
    }{t}.
  \]

  The explicit formulae for \(\Phi_q\) and \(\nabla\Phi_q\), together with \(H=\ell^q\) and \(t=\ell^{q-1}\), give
  \begin{align*}
    \frac{\Phi_q(\ell Z)}{H}
     & =
    \frac{
      1-(1-\ell^q|Z|^q)^{1/q}
    }{\ell^q}
    \longrightarrow \frac{|Z|^q}{q}, \\
    \frac{\nabla\Phi_q(\ell Z)}{t}
     & =
    |Z|^{q-2}Z
    (1-\ell^q|Z|^q)^{-\alpha/q}
    \longrightarrow |Z|^{q-2}Z.
  \end{align*}
  Moreover,
  \[
    u_1(\xi,t\omega)\longrightarrow\overline u_1(\xi),
    \qquad
    u(\xi,t\omega)\longrightarrow\overline u(\xi).
  \]
  These convergences are uniform on fixed compact sets, and therefore
  \begin{align*}
    P_\ell(Z,\xi,\omega)
     & \longrightarrow
    \frac{|Z|^q}{q}
    -\omega\frac{Z_1}{\overline u_1(\xi)}, \\
    G_\ell(Z,\xi,\omega)
     & \longrightarrow
    \overline u(\xi)\cdot|Z|^{q-2}Z-\omega.
  \end{align*}

  At \((Z,\xi,\omega)=(Z^*,0,\omega^*)\), the limiting phase is zero, since
  \[
    \frac{|Z^*|^q}{q}-\omega^*Z_1^*
    =
    \frac{r_0^q}{q}
    -\frac{r_0^{q-1}}{q}r_0
    =0,
  \]
  whereas the limiting normalised flux is
  \[
    r_0^{q-1}-\omega^*
    =\frac{q-1}{q}r_0^{q-1}>0.
  \]
  Since \(\eta_{\Sp}(0)\psi(\omega^*)\chi(0)\ne0\), continuity of the limiting expressions and uniform convergence allow us to choose a fixed
  product neighbourhood \(U\) of \((Z^*,0,\omega^*)\) such that, for every sufficiently small \(\ell\), the three cut-off factors are bounded
  away
  from zero in absolute value on \(U\), and
  \[
    |G_\ell|
    \ge \frac{q-1}{2q}r_0^{q-1}.
  \]

  Under the same changes of variables,
  \[
    \dd z=\ell^m\dd Z,
    \qquad
    \dd\sigma(\mathcal V(\xi,t\omega))
    =
    \frac{t}{u_1(\xi,t\omega)}
    \dd\xi\,\dd\omega,
  \]
  and, by the definition of \(G_\ell\),
  \[
    \left|
    u(\xi,t\omega)\cdot\nabla\Phi_q(\ell Z)-t\omega
    \right|^p
    =
    t^p|G_\ell(Z,\xi,\omega)|^p.
  \]
  Using these identities in the boundary integral and restricting it to \(U\) gives
  \begin{align*}
    \mathcal T^\sigma_{p,\Omega_{\alpha,d}}(g_\ell^\Sp)
    \ge{} &
    A_\ell^2\ell^m t^{p+1}
    \int_U
    |\eta_{\Sp}(\xi)|^2|\psi(\omega)|^2
    |\chi(P_\ell(Z,\xi,\omega))|^2 \\
          & \quad{}\times
    |G_\ell(Z,\xi,\omega)|^p
    \frac{
      (1+|\nabla\Phi_q(\ell Z)|^2)^{(1-p)/2}
    }{
      u_1(\xi,t\omega)
    }
    \dd Z\,\dd\xi\,\dd\omega.
  \end{align*}
  When \(d=2\), the \(\xi\)-variable and \(\dd\xi\) are absent. The remaining chart-density and surface factors are uniformly bounded above and
  away from zero on \(U\). Hence the integral is bounded below by a positive constant, and
  \begin{equation}\label{eq:rough-spherical-trace-lower}
    \mathcal T^\sigma_{p,\Omega_{\alpha,d}}(g_\ell^\Sp)
    \ge c_p A_\ell^2\ell^m t^{p+1}.
  \end{equation}
  Combining \eqref{eq:rough-spherical-trace-upper}, \eqref{eq:rough-spherical-trace-lower}, and \eqref{eq:rough-common-scales} proves
  \[
    \mathcal T^\sigma_{p,\Omega_{\alpha,d}}(g_\ell^\Sp)
    \asymp\ell^{\alpha(p+1)-1}.
  \]

  Finally, define
  \[
    (g_\ell^\Sp)^\#(x,v):=g_\ell^\Sp(x,-v),
    \qquad
    f_\ell^\Sp:=g_\ell^\Sp+(g_\ell^\Sp)^\#.
  \]
  The two summands have disjoint velocity supports. The antipodal map preserves surface measure and the spherical \(\H^1\)-norm, and
  \[
    v\cdot\nabla_x(g_\ell^\Sp)^\#(x,v)
    =-\bigl[(-v)\cdot\nabla_xg_\ell^\Sp\bigr](x,-v)=0.
  \]
  Thus
  \[
    \mathcal E^\sigma_{\Omega_{\alpha,d}}(f_\ell^\Sp)
    =2\mathcal E^\sigma_{\Omega_{\alpha,d}}(g_\ell^\Sp),
    \qquad
    \mathcal T^\sigma_{p,\Omega_{\alpha,d}}(f_\ell^\Sp)
    =2\mathcal T^\sigma_{p,\Omega_{\alpha,d}}(g_\ell^\Sp).
  \]
  Moreover, \(f_\ell^\Sp(x,-v)=f_\ell^\Sp(x,v)\), so the antipodal change of variables exchanges the incoming and outgoing phase boundaries.
  This proves Theorem~\ref{thm:spherical}.
\end{proof}

\subsection{Radial thickening to the bounded-support Euclidean model}

\begin{proof}[Proof of Theorem~\ref{thm:bounded-euclidean}]
  Fix the prescribed \(L>0\), choose
  \[
    0<a<b<\min\{1,L\},
    \qquad
    0\ne\vartheta_L\in\C_c^\infty((a,b)),
  \]
  and, for \(v=r\theta\), define
  \begin{equation}\label{eq:rough-radial-thickening}
    f_\ell(x,r\theta)
    :=\vartheta_L(r)f_\ell^\Sp(x,\theta).
  \end{equation}
  Set \(f_\ell(x,0):=0\). The right-hand side vanishes for \(r\) near zero, so it defines a function in
  \(\C_c^\infty(\overline{\Omega_{\alpha,d}}\times\R^d)\), with \(\supp_vf_\ell\Subset B_L\). The same family will be used for both Euclidean
  measures and for every \(p\).

  Write
  \[
    \rho_{\mathcal L^d}(r):=1,
    \qquad
    \rho_{\gamma_d}(r):=(2\pi)^{-d/2}e^{-r^2/2}.
  \]
  Since
  \[
    \nabla_v f_\ell(x,r\theta)
    =
    \vartheta_L'(r)f_\ell^\Sp(x,\theta)\theta
    +\frac{\vartheta_L(r)}r
    \nabla_{\Sph^{d-1}}f_\ell^\Sp(x,\theta),
  \]
  and the two terms are orthogonal, polar coordinates give
  \begin{align*}
    \norm{f_\ell(x,\cdot)}_{\H^1_\mu}^2
    ={} &
    a_{\mu,L}
    \norm{f_\ell^\Sp(x,\cdot)}_{\L^2(\Sph^{d-1})}^2 \\
        & +
    b_{\mu,L}
    \norm{\nabla_{\Sph^{d-1}}f_\ell^\Sp(x,\cdot)}
    {\L^2(\Sph^{d-1})}^2,
  \end{align*}
  where
  \begin{align*}
    a_{\mu,L}
     & :=
    \int_0^\infty
    \bigl(\vartheta_L(r)^2+\vartheta_L'(r)^2\bigr)
    \rho_\mu(r)r^{d-1}\dd r, \\
    b_{\mu,L}
     & :=
    \int_0^\infty
    \vartheta_L(r)^2\rho_\mu(r)r^{d-3}\dd r.
  \end{align*}
  Both constants are finite. When \(d=2\), the factor \(r^{d-3}=r^{-1}\) in \(b_{\mu,L}\) causes no singularity because \(\vartheta_L\) is
  supported away from \(r=0\). Moreover,
  \[
    v\cdot\nabla_xf_\ell(x,r\theta)
    =
    r\vartheta_L(r)
    \theta\cdot\nabla_xf_\ell^\Sp(x,\theta)=0.
  \]
  Consequently,
  \[
    \mathcal E^\mu_{\Omega_{\alpha,d}}(f_\ell)
    \le C_{\mu,L}
    \mathcal E^\sigma_{\Omega_{\alpha,d}}(f_\ell^\Sp),
  \]
  and \eqref{eq:spherical-rough-scale} gives the uniform Euclidean graph-energy bounds.

  Because \(r<b<1\) and \(|\theta\cdot n|\le1\),
  \[
    \omega_p(r|\theta\cdot n|)
    =r^p|\theta\cdot n|^p.
  \]
  Another polar-coordinate calculation therefore gives the exact identity
  \begin{equation}\label{eq:rough-radial-trace}
    \mathcal T^\mu_{p,\Omega_{\alpha,d}}(f_\ell)
    =
    c_{p,\mu,L}\,
    \mathcal T^\sigma_{p,\Omega_{\alpha,d}}(f_\ell^\Sp),
    \qquad
    c_{p,\mu,L}:=
    \int_0^\infty
    \vartheta_L(r)^2\rho_\mu(r)r^{d-1+p}\dd r>0.
  \end{equation}
  Since \(r>0\), the signs of \(v\cdot n\) and \(\theta\cdot n\) agree. Hence \eqref{eq:rough-radial-trace} holds separately on the incoming
  and outgoing phase boundaries. The spherical estimates and symmetry now prove Theorem~\ref{thm:bounded-euclidean}.
\end{proof}

\begin{remark}[Endpoint behaviour]
  \label{rem:bnd-endpoint}
  For each fixed \(p\), at \(\alpha=\alpha_p=1/(p+1)\), \eqref{eq:bnd-invariant-scale} and \eqref{eq:spherical-rough-scale} show that the
  weighted boundary functionals of both packet families remain bounded above and away from zero. For \(0<\alpha<\alpha_p\), these functionals
  diverge on the strictly convex bounded domain \(\Omega_{\alpha,d}\), whose boundary has exact regularity \(\C^{1,\alpha}\).
\end{remark}

\begin{remark}[Relation to \cite{Valentini2026}]\label{rem:valentini}
  When interpreted as a bounded restriction estimate on the entire kinetic energy space, the estimate in
  \cite[Equation~(2.14)]{Valentini2026} does not hold on all
  Lipschitz domains when \(d\ge2\): it fails on \(\Omega_{\alpha,d}\) for \(0<\alpha<1/2\). This is consistent with
  \cite[Definition~2.6 and Theorem~2.7]{Valentini2026}: the theorem there is formulated on \(\H^1_{\mathrm{hyp},\mathrm{tr}}\). In that space, a
  flux-integrable trace satisfying Green's identity is part of the definition, and the theorem controls either one-sided trace in terms of the
  bulk norm and the other trace.
  Our packets have equal incoming and outgoing weighted boundary functionals, so both may grow without violating that estimate. By contrast,
  Corollary~\ref{cor:spherical-positive}, together with Proposition~\ref{cor:intrinsic-p-trace}, proves the natural spherical trace estimate on
  the entire kinetic energy space for every bounded \(\C^{1,1/2}\) domain, showing that the regularity threshold \(1/2\) is sharp.
\end{remark}

\section{\texorpdfstring{Bounded-velocity models: \(\omega_p\)-trace
    estimates for \(\C^{1,\frac{1}{p+1}}\) domains} {Bounded-velocity models: omega-p-trace estimates for C1,1/(p+1) domains}}
\label{sec:bounded-positive}

Let $d \ge 2$. Fix \(1\le p<\infty\) and recall \(\alpha_p=1/(p+1)\). It suffices to prove the estimate at \(\alpha=\alpha_p\): on each bounded
chart, \(\C^{0,\alpha}\subset\C^{0,\alpha_p}\) whenever \(\alpha>\alpha_p\). For the positive implication below, we only use \(\omega_p(s)\le
s^p\).

The proof is organised around flattened-normal-velocity multipliers and Green identities, as in \cite[Proposition~4.3]{Silvestre2022} and
\cite[Proposition~3.7]{OuyangSilvestre2024}. Its local part has two genuinely different regularisations. A height-dependent mollified normal
velocity produces the critical bulk Hardy estimate, whereas height-independent dyadically mollified slopes recover the boundary weight \(\lvert
c\rvert^p\).

\begin{theorem}[\(\omega_p\)-trace estimate for the bounded-support
    Euclidean model]
  \label{thm:bounded-positive}
  Let \(d\ge2\), let \(\Omega\subset\R^d\) be a bounded \(\C^{1,\alpha}\) domain with \(\alpha_p\le\alpha\le1\), and fix \(L >0\). For either
  \(\mu\in\{\mathcal L^d,\gamma_d\}\), every \(f\in\C_c^\infty(\overline\Omega\times\R^d)\) with \(\supp_vf\subset \overline B_L\) satisfies
  \begin{align}
    \int_{\partial\Omega}\int_{\R^d}
    |f(x,v)|^2 & |v\cdot n_\Omega(x)|^p\dd\mu(v)\dd S(x)
    \notag                                               \\
               & \le C\left(
    \norm{f}_{\L^2(\Omega;\H^1_\mu)}^2+
    \norm{v\cdot\nabla_xf}_{\L^2(\Omega;\H^{-1}_\mu)}^2
    \right),
    \label{eq:bounded-positive}
  \end{align}
  and thus, in particular, the $\omega_p$-trace estimate \eqref{eq:trace_ineq}. Here \(C>0\) depends on \(p,d,L\) and on a fixed quantitative
  \(\C^{1,\alpha_p}\) atlas; see Remark~\ref{rem:quantitative-atlas}.
\end{theorem}

\begin{corollary}[Natural trace estimate at the critical regularity]
  \label{cor:bounded-positive-natural}
  Let \(d\ge2\), let \(\Omega\subset\R^d\) be a bounded \(\C^{1,\alpha}\) domain with \(1/2\le\alpha\le1\), and fix \(L >0\). For either
  \(\mu\in\{\mathcal L^d,\gamma_d\}\), every \(f\in\C_c^\infty(\overline\Omega\times\R^d)\) with \(\supp_vf\subset \overline B_L\) satisfies the
  $\omega_1$-trace estimate \eqref{eq:trace_ineq}. The constant in the natural trace estimate depends only on \(d,L\) and on a fixed quantitative
  \(\C^{1,1/2}\) atlas.
\end{corollary}

The proof has four steps.
\begin{enumerate}[leftmargin=*]
  \item We prove a local critical Hardy estimate using a height-dependent mollified normal velocity.
  \item We use that estimate to control a dyadic boundary multiplier whose boundary value is comparable with \(-\operatorname{sgn}(c)\lvert
        c\rvert^{p-1}\).
  \item We flatten the boundary, localise, and sum the local graph estimates.
  \item We obtain the spherical estimate by radial thickening.
\end{enumerate}
The argument includes \(p=1\) directly: all factors \(p-1\) vanish, and finite overlap keeps the dyadic multiplier bounded even though
\(a_j^{p-1}=1\).

\subsection{Local graph formulation}
Fix $1\le p < \infty$ and $L>0$, and let $\mu\in\{\mathcal L^d,\gamma_d\}$. Let \(x=(z,y)\in\R^{d-1}\times (0, \infty)\), and
\(v=(u,w)\in\R^{d-1}\times\R\). For a bounded field \(b\in\C^{0,\alpha_p}(\R^{d-1};\R^{d-1})\), define
\begin{equation}\label{eq:bounded-positive-X}
  c(z,v):=w-u\cdot b(z),
  \qquad
  X_b:=u\cdot\nabla_z+c(z,v)\partial_y.
\end{equation}
Here \(c\) is the flattened, unnormalised normal velocity, and \(X_b\) is the transport field obtained by flattening a graph \(x_d=\Phi(z)\),
with \(b=\nabla\Phi\). Unless a domain is specified, bulk integrals in this section are over \(\R^{d-1}\times(0,\infty)\times\R^d\) with measure
\(\dd z\dd y\dd\mu(v)\), and boundary integrals are over \(\R^{d-1}\times\R^d\) with measure \(\dd z\dd\mu(v)\).

\begin{proposition}[Local graph \(\omega_p\)-trace estimate]
  \label{prop:bounded-positive-local}
  If \(F\in\C_c^1(\R^{d-1}\times[0,\infty)\times\R^d)\) and \(\supp_vF\subset \overline B_L\), then
  \begin{align}
    \int_{\R^{d-1}}\int_{\R^d}
    |F(z,0,v)|^2|c(z,v)|^p\dd\mu(v)\dd z
    \le C\left(
    \norm{F}_{\L^2_{z,y}\H^1_\mu}^2+
    \norm{X_bF}_{\L^2_{z,y}\H^{-1}_\mu}^2
    \right),
    \label{eq:bounded-positive-local}
  \end{align}
  where \(C\) depends only on \(p,d,L,\norm{b}_{\L^\infty}\), and \([b]_{\C^{0,\alpha_p}}\).
\end{proposition}

\subsection{Multiplier framework and critical scales}
Write

\[
  B:=\|b\|_{\L^\infty},\quad
  H:=[b]_{\C^{0,\alpha_p}}.
\]
For the moment, let \(\Lambda\ge1\). Its lower bound will be fixed only after the coercive multiplier profile has been constructed. The dyadic
parameters are independent of this construction and will be introduced where they are first used.

Put
\begin{equation}\label{eq:bounded-positive-beta}
  \beta:=\frac1{p+2}.
\end{equation}
The critical identities are
\begin{equation}\label{eq:bounded-positive-critical-identities}
  \alpha_p=\frac1{p+1},\qquad
  (1-\beta)\alpha_p=\beta,\qquad
  p\beta-1=-2\beta,\qquad
  1-\beta=(p+1)\beta.
\end{equation}
Fix a nonnegative tangential mollifier
\[
  \eta\in\C_c^\infty(\R^{d-1}),\qquad
  \supp\eta\subset B_1^{d-1},\qquad
  \int_{\R^{d-1}}\eta(z)\,\dd z=1,
\]
and, for \(\ell>0\), set
\[
  \eta_\ell(z):=\ell^{-(d-1)}\eta(z/\ell).
\]
Define
\begin{equation}\label{eq:bounded-positive-scales}
  \rho(y):=\Lambda y^\beta,\qquad
  \ell(y):=\Lambda^{-1}y^{1-\beta},\qquad
  b_y:=\eta_{\ell(y)}*b,
\end{equation}
where the convolution is taken in the tangential variable \(z\). The continuous scales are
\[
  \begin{array}{c|c|c}
    \text{quantity} & \text{definition}                                 & \text{role} \\ \hline
    \rho(y)         & \Lambda y^{1/(p+2)}
                    & \text{width of the grazing-velocity region}                     \\[1mm]
    \ell(y)         & \Lambda^{-1}y^{(p+1)/(p+2)}
                    & \text{height-dependent slope-mollification scale}               \\[1mm]
  \end{array}
\]
These scales make the threshold visible:
\[
  \rho(y)\ell(y)=y,\qquad
  \ell(y)^{\alpha_p}
  =\Lambda^{-\alpha_p}y^\beta,\qquad
  \frac{\rho(y)^p}{y}
  =\Lambda^py^{-2\beta},\qquad
  \rho(y)^{-2}=\Lambda^{-2}y^{-2\beta}.
\]
Thus the normal-approximation error is of grazing size \(y^\beta\), while both transport coercivity and the squared velocity-gradient cost
carry the critical Hardy weight \(y^{-2\beta}\).

\begin{lemma}[Admissible-multiplier Green identity]
  \label{lem:bounded-positive-admissible-multiplier}
  \label{lem:bounded-positive-green}
  Let \(F\) be as in Proposition~\ref{prop:bounded-positive-local}, and let \(m\in\C^1(\R^{d-1}\times(0,\infty)\times\R^d)\). Suppose that, for
  some \(A>0\),
  \[
    |m|\le A\qquad\text{on }\supp F,
  \]
  and that
  \begin{equation}\label{eq:bounded-positive-admissibility-integral}
    \int\bigl(|\nabla_vm|^2+|X_bm|\bigr)|F|^2
    \dd z\dd y\dd\mu<\infty.
  \end{equation}
  For each fixed \(v\), the coefficient field \((u,c)\) of \(X_b\) is divergence-free in \((z,y)\):
  \[
    \nabla_z\cdot u+\partial_yc=0.
  \]
  Then
  \begin{equation}\label{eq:bounded-positive-admissibility-product}
    \norm{mF}_{\L^2_{z,y}\H^1_\mu}^2
    \le 2A^2\norm{F}_{\L^2_{z,y}\H^1_\mu}^2
    +2\int|\nabla_vm|^2|F|^2\dd z\dd y\dd\mu,
  \end{equation}
  and, for every \(\delta>0\),
  \begin{align}
    -\int_{\R^{d-1}}\int_{\R^d}
    c(z,v)m(z,\delta,v)|F(z,\delta,v)|^2\dd\mu(v)\dd z
     & =\int_{y>\delta}(X_bm)|F|^2\dd z\dd y\dd\mu
    \notag                                         \\
     & \quad+2\int_{y>\delta}
    \pair{X_bF}{mF}_{\H^{-1}_\mu,\H^1_\mu}\dd z\dd y.
    \label{eq:bounded-positive-green}
  \end{align}
  If, in addition,
  \begin{equation}\label{eq:bounded-positive-admissibility-boundary}
    B_\delta:=-\int c(z,v)m(z,\delta,v)|F(z,\delta,v)|^2
    \dd\mu(v)\dd z\longrightarrow B_0,
  \end{equation}
  then all terms below are absolutely convergent and
  \begin{align}
    B_0
    =\int_{y>0}(X_bm)|F|^2\dd z\dd y\dd\mu
    +2\int_{y>0}
    \pair{X_bF}{mF}_{\H^{-1}_\mu,\H^1_\mu}\dd z\dd y.
    \label{eq:bounded-positive-admissibility-limit}
  \end{align}
\end{lemma}

\begin{proof}
  The velocity product rule and \(|m|\le A\) give
  \[
    |mF|^2+|\nabla_v(mF)|^2
    \le2A^2\bigl(|F|^2+|\nabla_vF|^2\bigr)
    +2|\nabla_vm|^2|F|^2,
  \]
  which proves \eqref{eq:bounded-positive-admissibility-product}.

  Fix \(\delta>0\). For fixed \(v=(u,w)\), put \(h=mF^2\). Compact support gives
  \[
    \int_{y>\delta}u\cdot\nabla_zh\dd z\dd y=0.
  \]
  Moreover, \(c(z,v)=w-u\cdot b(z)\) is independent of \(y\), and hence
  \[
    \int_{y>\delta}c(z,v)\partial_yh\dd z\dd y
    =-\int_{\R^{d-1}}c(z,v)h(z,\delta,v)\dd z.
  \]
  Since \(u\) is independent of \(z\) and \(c=w-u\cdot b(z)\) is independent of \(y\), the coefficient field of \(X_b\)
  is divergence-free in \((z,y)\). Consequently, the Green-identity argument never differentiates the merely H\"older field \(b\). Expanding
  \(X_b(mF^2)=(X_bm)F^2+2mF\,X_bF\), integrating in \(v\), and using \(\H^{-1}_\mu,\H^1_\mu\) duality proves \eqref{eq:bounded-positive-green}.
  The bulk terms are absolutely integrable by \eqref{eq:bounded-positive-admissibility-integral} and the estimate
  \[
    \int\left|
    \pair{X_bF}{mF}_{\H^{-1}_\mu,\H^1_\mu}
    \right|\dd z\dd y
    \le
    \norm{X_bF}_{\L^2_{z,y}\H^{-1}_\mu}
    \norm{mF}_{\L^2_{z,y}\H^1_\mu}<\infty
  \]
  above. Letting \(\delta\downarrow0\) and using the assumed convergence of \(B_\delta\) proves
  \eqref{eq:bounded-positive-admissibility-limit}.
\end{proof}

\subsection{Height-dependent slope mollification}

With \(\rho,\ell,b_y\) as above, define
\begin{equation}\label{eq:bounded-positive-adapted-variables}
  \bar c:=w-u\cdot b_y(z),\qquad
  e:=c-\bar c=u\cdot(b_y-b),\qquad
  r:=\frac{\bar c}{\rho(y)}.
\end{equation}
When \(X_b\) differentiates \(\bar c\), the \(y\)-dependence of \(b_y\) produces the term \(-c\,u\cdot\partial_yb_y\); its size is controlled by
the critical choice of \(\ell(y)\). The core point is the approximate Riccati evolution
\[
  X_br\approx-\beta\frac{\rho}{y}r^2.
\]
In a multiplier \(\rho^{p-1}G(r)\), this quadratic drift produces the critical factor \(\rho^p/y=\Lambda^py^{-2\beta}\). The error \(E\) below
measures the cost of replacing the true slope \(b\) by its height-dependent mollification \(b_y\).

\begin{lemma}[Height-dependent mollified normal velocity]
  \label{lem:bounded-positive-adapted}
  On \(|v|\le L\),
  \begin{align}
    |b_y-b|         & \le CH\Lambda^{-\alpha_p}y^\beta,
    \label{eq:bounded-positive-b-error}                      \\
    |\nabla_zb_y|   & \le CH\Lambda^{1-\alpha_p}y^{-p\beta},
    \label{eq:bounded-positive-b-gradient}                   \\
    |\partial_yb_y| & \le
    CH\Lambda^{-\alpha_p}y^{-(p+1)\beta}.
    \label{eq:bounded-positive-b-yderivative}
  \end{align}
  Moreover,
  \begin{equation}\label{eq:bounded-positive-r-evolution}
    X_br=-\beta\frac{\rho}{y}r^2+E,\qquad
    |E|\le C_*\Lambda^{-\alpha_p}(1+|r|)
    y^{-(p+1)\beta},
  \end{equation}
  where \(C_*=C_*(p,d,L,B,H)\), and
  \begin{equation}\label{eq:bounded-positive-r-vgradient}
    \nabla_vr=\frac{(-b_y,1)}{\rho(y)}.
  \end{equation}
\end{lemma}

\begin{proof}
  The standard H\"older mollifier estimates give
  \[
    |b_y-b|\le CH\ell(y)^{\alpha_p},\qquad
    |\nabla_zb_y|+|\partial_\ell b_y|
    \le CH\ell(y)^{\alpha_p-1}.
  \]
  Since
  \[
    \ell'(y)=(1-\beta)\Lambda^{-1}y^{-\beta},
  \]
  the identities in \eqref{eq:bounded-positive-critical-identities} give
  \begin{align*}
    |b_y-b|
     & \le CH\Lambda^{-\alpha_p}y^{(1-\beta)\alpha_p}
    =CH\Lambda^{-\alpha_p}y^\beta,                    \\
    |\nabla_zb_y|
     & \le CH\Lambda^{1-\alpha_p}
    y^{(1-\beta)(\alpha_p-1)}
    =CH\Lambda^{1-\alpha_p}y^{-p\beta},               \\
    |\partial_yb_y|
     & \le CH\ell(y)^{\alpha_p-1}|\ell'(y)|
    \le CH\Lambda^{-\alpha_p}y^{-(p+1)\beta}.
  \end{align*}
  This proves \eqref{eq:bounded-positive-b-error}--\eqref{eq:bounded-positive-b-yderivative}. In particular, on \(|v|\le L\),
  \[
    |e|=|u\cdot(b_y-b)|
    \le CLH\Lambda^{-\alpha_p}y^\beta.
  \]

  We next compute the evolution of the height-dependent mollified normal velocity. Since \(u\) and \(w\) are independent of \((z,y)\),
  \[
    X_b\bar c
    =-u^{\mathsf T}(\nabla_zb_y)u
    -c\,u\cdot\partial_yb_y.
  \]
  Moreover, \(X_by=c\) and \(X_b\rho=\rho'(y)c=\beta\rho c/y\). Therefore
  \[
    X_br=\frac{X_b\bar c}{\rho}-\beta\frac{rc}{y}
    =-\beta\frac{\rho}{y}r^2+E,
  \]
  where, using \(c=\rho r+e\),
  \[
    E:=-\frac{u^{\mathsf T}(\nabla_zb_y)u}{\rho}
    -\frac{c\,u\cdot\partial_yb_y}{\rho}
    -\beta\frac{re}{y}.
  \]
  We estimate these three terms separately. First,
  \begin{align*}
    \frac{|u^{\mathsf T}(\nabla_zb_y)u|}{\rho}
     & \le CL^2H
    \frac{\Lambda^{1-\alpha_p}y^{-p\beta}}
    {\Lambda y^\beta}                           \le C\Lambda^{-\alpha_p}y^{-(p+1)\beta}.
  \end{align*}
  Next, \eqref{eq:bounded-positive-b-error} and \(\Lambda\ge1\) imply
  \[
    \frac{|c|}{\rho}
    \le |r|+\frac{|e|}{\rho}
    \le |r|+C\Lambda^{-1-\alpha_p}\le C(1+|r|).
  \]
  Hence
  \[
    \frac{|c\,u\cdot\partial_yb_y|}{\rho}
    \le C\Lambda^{-\alpha_p}(1+|r|)
    y^{-(p+1)\beta}.
  \]
  Finally, since \(\beta-1=-(p+1)\beta\),
  \[
    \frac{|re|}{y}
    \le C\Lambda^{-\alpha_p}|r|y^{\beta-1}
    =C\Lambda^{-\alpha_p}|r|y^{-(p+1)\beta}.
  \]
  Combining the three bounds proves \eqref{eq:bounded-positive-r-evolution}. Since \(\rho\) is independent of \(v\) and \(\bar c=w-u\cdot b_y\),
  direct differentiation gives \eqref{eq:bounded-positive-r-vgradient}.
\end{proof}

The following identity isolates the one-dimensional profile operator
\[
  \mathfrak L_pG(r):=(p-1)rG(r)-r^2G'(r),
\]
which governs the Hardy multiplier.

\begin{lemma}[Critical multiplier identity]
  \label{lem:bounded-positive-multiplier}
  For a smooth scalar function \(G\), on \(y>0\) we have
  \begin{align}
    X_b(\rho^{p-1}G(r))
     & =\beta\frac{\rho^p}{y}
    \bigl((p-1)rG(r)-r^2G'(r)\bigr)
    \notag                    \\
     & \quad+\rho^{p-1}G'(r)E
    +\beta(p-1)\frac{\rho^{p-1}e}{y}G(r).
    \label{eq:bounded-positive-multiplier}
  \end{align}
\end{lemma}

\begin{proof}
  The chain rule gives
  \[
    X_b(\rho^{p-1}G(r))
    =(p-1)\rho^{p-2}(X_b\rho)G(r)
    +\rho^{p-1}G'(r)X_br.
  \]
  Since \(X_b\rho=\beta\rho c/y\) and \(c=\rho r+e\), the first term is
  \[
    \beta(p-1)\frac{\rho^p}{y}rG(r)
    +\beta(p-1)\frac{\rho^{p-1}e}{y}G(r).
  \]
  Substitution of \eqref{eq:bounded-positive-r-evolution} in the second term gives
  \[
    -\beta\frac{\rho^p}{y}r^2G'(r)
    +\rho^{p-1}G'(r)E.
  \]
  Adding these two displays proves \eqref{eq:bounded-positive-multiplier}.
\end{proof}

\subsection{Critical Hardy estimate from a single multiplier profile}

\begin{lemma}[Single coercive multiplier profile]
  \label{lem:bounded-positive-single-profile}
  There is an odd \(G\in\C^\infty(\R)\) such that
  \begin{equation}\label{eq:bounded-positive-single-profile-bounds}
    G(r)=\frac1r\quad (|r|\ge2),
    \qquad
    |G(r)|+(1+|r|)|G'(r)|
    +|\mathfrak L_pG(r)|\le C.
  \end{equation}
  Moreover, there are \(c,C>0\), depending only on \(p,L,B\), such that for every \(|m|\le LB\), every \(0<\varrho\le1\), and every
  \(h\in\C_c^1(\R)\),
  \begin{equation}\label{eq:bounded-positive-single-profile-coercivity}
    \int_{\R}\mathfrak L_pG(r)|h(r)|^2
    \varpi_{m,\varrho}(r)\dd r
    \ge
    c\int_{\R}|h(r)|^2\varpi_{m,\varrho}(r)\dd r
    -C\int_{|r|<2}|h'(r)|^2
    \varpi_{m,\varrho}(r)\dd r,
  \end{equation}
  where
  \[
    \varpi_{m,\varrho}(r):=
    \begin{cases}
      1,
       & \mu=\mathcal L^d, \\[1mm]
      (2\pi)^{-1/2}
      \exp\!\left(-\dfrac{|m+\varrho r|^2}{2}\right),
       & \mu=\gamma_d.
    \end{cases}
  \]
  The same profile \(G\) works for both velocity measures.
\end{lemma}

\begin{proof}
  Choose an odd \(R\in\C^\infty(\R)\) such that
  \[
    R(r)=\frac1r\quad(|r|\ge2),
    \qquad
    |R(r)|+(1+|r|)|R'(r)|\le C.
  \]
  Then
  \[
    \mathfrak L_pR(r)=p\quad(|r|\ge2),
    \qquad
    \mathfrak L_pR(r)
    \ge p\one_{\{|r|\ge2\}}(r)
    -C_R\one_{\{|r|<2\}}(r)
  \]
  for some \(C_R<\infty\).

  Let \(Q_\varepsilon\) be supplied by Lemma~\ref{lem:compact-odd-hardy-profile}. Since \(p\ge1\) and \(rQ_\varepsilon(r)\ge0\),
  \begin{equation}\label{eq:bounded-positive-single-Q-coercivity}
    \mathfrak L_pQ_\varepsilon(r)
    \ge c_0\one_K(r)
    -C_0\varepsilon\one_{\{|r|<2\varepsilon\}}(r).
  \end{equation}

  We use a weighted form of the anchored inequality. On \(|r|<2\), the Gaussian weight \(\varpi_{m,\varrho}\) is bounded above and below
  uniformly for \(|m|\le LB\) and \(0<\varrho\le1\). Therefore \eqref{eq:flat-anchor} gives, for both velocity measures,
  \begin{equation}\label{eq:bounded-positive-weighted-anchor}
    \int_{|r|<2}|h|^2\varpi_{m,\varrho}\dd r
    \le C_{\rm an}\left(
    \int_K|h|^2\varpi_{m,\varrho}\dd r
    +\int_{|r|<2}|h'|^2\varpi_{m,\varrho}\dd r
    \right),
  \end{equation}
  where \(C_{\rm an}=C_{\rm an}(L,B)\).

  Choose
  \[
    0<\varepsilon<
    \min\left\{\frac18,\frac{c_0}{2C_{\rm an}C_0}\right\},
  \]
  and then choose \(A_*>0\) so large that
  \[
    \frac{A_*c_0}{2C_{\rm an}}-C_R\ge1.
  \]
  Set \(G:=R+A_*Q_\varepsilon\). The bounds in \eqref{eq:bounded-positive-single-profile-bounds} follow from the construction. To prove
  coercivity, write
  \begin{align*}
    I_{\rm in}
     & :=\int_{|r|<2}|h|^2\varpi_{m,\varrho}\dd r,
     & I_{\rm out}
     & :=\int_{|r|\ge2}|h|^2\varpi_{m,\varrho}\dd r, \\
    I_K
     & :=\int_K|h|^2\varpi_{m,\varrho}\dd r,
     & D
     & :=\int_{|r|<2}|h'|^2\varpi_{m,\varrho}\dd r.
  \end{align*}
  Inequality \eqref{eq:bounded-positive-weighted-anchor} implies \(I_K\ge C_{\rm an}^{-1}I_{\rm in}-D\). Consequently,
  \begin{align*}
    \int_{\R}\mathfrak L_pG\,|h|^2\varpi_{m,\varrho}\dd r
     & \ge pI_{\rm out}-C_RI_{\rm in}
    +A_*c_0I_K-A_*C_0\varepsilon I_{\rm in} \\
     & \ge pI_{\rm out}
    +\left(
    \frac{A_*c_0}{C_{\rm an}}
    -C_R-A_*C_0\varepsilon
    \right)I_{\rm in}
    -A_*c_0D                                \\
     & \ge pI_{\rm out}+I_{\rm in}-A_*c_0D.
  \end{align*}
  This proves \eqref{eq:bounded-positive-single-profile-coercivity}. The argument includes \(p=1\): then the term \((p-1)rQ_\varepsilon\)
  vanishes, while \(\mathfrak L_1R=1\) on \(\{|r|\ge2\}\).
\end{proof}

\paragraph{Parameter hierarchy.}
The critical-regularity proof \((\alpha=\alpha_p)\) fixes its choices in the following order:
\begin{enumerate}[leftmargin=*]
  \item choose the coercive profile \(G\) from Lemma~\ref{lem:bounded-positive-single-profile};
  \item choose \(\Lambda\) sufficiently large to absorb the mollified-slope errors in the Hardy estimate;
  \item choose \(h_0>0\) sufficiently small that \(\rho(2h_0)\le1\);
  \item independently choose the dyadic parameter \(\kappa>0\) sufficiently small to obtain \eqref{eq:bounded-positive-dyadic-error};
  \item after defining the dyadic scales, select \(J_0\) so that \(a_j\le1\) and \(2\tau_j\le h_0\) for every \(j\ge J_0\).
\end{enumerate}
Thus \(\kappa\) is independent of the choices of \(\Lambda\) and \(h_0\), while \(J_0\) is selected only after all the preceding parameters
have been fixed.

Choose \(\Lambda\ge1\) sufficiently large as specified in the proof below, and then choose \(h_0=h_0(\Lambda)>0\) so small that
\begin{equation}\label{eq:bounded-positive-h0}
  \rho(2h_0)\le1.
\end{equation}
Let \(\chi_0\in\C_c^\infty(-h_0,2h_0)\) equal one on \([0,h_0]\). For \(F\) as in Proposition~\ref{prop:bounded-positive-local}, put
\[
  M:=\norm{F}_{\L^2_{z,y}\H^1_\mu},
  \qquad
  N:=\norm{X_bF}_{\L^2_{z,y}\H^{-1}_\mu},
\]
and define
\begin{equation}\label{eq:bounded-positive-H}
  \mathcal H(F)
  :=\int_0^{h_0}\int_{\R^{d-1}}\int_{\R^d}
  y^{-2\beta}|F|^2\dd\mu(v)\dd z\dd y.
\end{equation}
This quantity is finite for the smooth test class because \(2\beta<1\).

\begin{proposition}[Critical Hardy estimate]
  \label{prop:bounded-positive-global-hardy}
  After choosing the profile above, one may choose \(\Lambda\) sufficiently large and then \(h_0\) satisfying \eqref{eq:bounded-positive-h0} so
  that
  \begin{equation}\label{eq:bounded-positive-global-hardy}
    \mathcal H(F)\le C(M^2+N^2).
  \end{equation}
\end{proposition}

We call the estimate critical because, at \(\alpha=\alpha_p\), the geometric and analytic scales coincide: \(\ell(y)^{\alpha_p}\asymp\rho(y)\)
and \(\rho(y)^p/y\asymp\rho(y)^{-2}\asymp y^{-2\beta}\).

\begin{proof}
  Let \(G\) be supplied by Lemma~\ref{lem:bounded-positive-single-profile}, and define
  \begin{equation}\label{eq:bounded-positive-single-multiplier}
    q(z,y,v):=\chi_0(y)\rho(y)^{p-1}G(r).
  \end{equation}

  We first establish the product bound. On \(0<y<2h_0\), condition \eqref{eq:bounded-positive-h0} gives \(\rho(y)\le1\), and
  \(\|b_y\|_\infty\le B\). Hence
  \[
    |q|\le C,\qquad
    \nabla_vq
    =\chi_0\rho^{p-2}G'(r)(-b_y,1).
  \]
  Since \(p\ge1\),
  \[
    |\nabla_vq|^2
    \le C\rho^{2p-4}
    \le C\rho^{-2}
    =C\Lambda^{-2}y^{-2\beta}.
  \]
  The velocity product rule therefore gives
  \begin{equation}\label{eq:bounded-positive-single-product}
    \norm{qF}_{\L^2_{z,y}\H^1_\mu}^2
    \le C_{\Lambda,h_0} M^2+C\Lambda^{-2}\mathcal H(F).
  \end{equation}
  Here the transition strip \(h_0<y<2h_0\) is harmless because \(y^{-2\beta}\le h_0^{-2\beta}\); its contribution can therefore be absorbed
  into \(C_{\Lambda,h_0}M^2\).

  On \(0<y<h_0\), where \(\chi_0=1\), Lemma~\ref{lem:bounded-positive-multiplier} gives
  \begin{equation}\label{eq:bounded-positive-single-transport}
    X_bq
    =\beta\frac{\rho^p}{y}\mathfrak L_pG(r)
    +\rho^{p-1}G'(r)E
    +\beta(p-1)\frac{\rho^{p-1}e}{y}G(r).
  \end{equation}
  Fix \((z,y,u)\). If the corresponding velocity slice of \(F\) is nonzero, then \(|u|\le L\). Put
  \[
    m:=u\cdot b_y(z),\qquad
    h(r):=F\bigl(z,y,u,m+\rho(y)r\bigr).
  \]
  In the Lebesgue case, apply \eqref{eq:bounded-positive-single-profile-coercivity} with weight one. In the Gaussian case, use
  \(\gamma_d=\gamma_{d-1}\otimes\gamma_1\) and the weight \(\varpi_{m,\rho(y)}\). Since
  \[
    h'(r)=\rho(y)\,
    \partial_wF\bigl(z,y,u,m+\rho(y)r\bigr),
  \]
  changing back to \(w\) gives, with \(\mu_1=\mathcal L^1\) in the Lebesgue case and \(\mu_1=\gamma_1\) in the Gaussian case,
  \begin{align}
    \int_{\R}\mathfrak L_pG(r)|F(z,y,u,w)|^2\dd\mu_1(w)
    \ge{} &
    c\int_{\R}|F(z,y,u,w)|^2\dd\mu_1(w)
    \notag              \\
          & -C\rho(y)^2
    \int_{\R}|\partial_wF(z,y,u,w)|^2\dd\mu_1(w).
    \label{eq:bounded-positive-single-sliced}
  \end{align}
  Here \(r=(w-m)/\rho(y)\) in the preceding display. Integrating in \((z,y,u)\) and using
  \[
    \frac{\rho^p}{y}=\Lambda^py^{-2\beta},
    \qquad
    \frac{\rho^{p+2}}{y}=\Lambda^{p+2},
  \]
  yields
  \begin{equation}\label{eq:bounded-positive-single-principal}
    \beta\int_{0<y<h_0}
    \frac{\rho^p}{y}\mathfrak L_pG(r)|F|^2
    \dd z\dd y\dd\mu
    \ge c\Lambda^p\mathcal H(F)-C\Lambda^{p+2}M^2.
  \end{equation}

  The two mollified-slope errors are controlled by \eqref{eq:bounded-positive-r-evolution}, \eqref{eq:bounded-positive-b-error}, and
  \eqref{eq:bounded-positive-single-profile-bounds}:
  \begin{align*}
    \rho^{p-1}|G'(r)E|
     & \le C\Lambda^{p-1-\alpha_p}y^{-2\beta}, \\
    \frac{\rho^{p-1}|e|}{y}|G(r)|
     & \le C\Lambda^{p-1-\alpha_p}y^{-2\beta}.
  \end{align*}
  On \(h_0<y<2h_0\), all coefficients in
  \[
    X_bq
    =\chi_0X_b\!\left(\rho^{p-1}G(r)\right)
    +c\chi_0'\rho^{p-1}G(r)
  \]
  are bounded on \(\supp F\). Combining these observations with \eqref{eq:bounded-positive-single-principal} gives
  \begin{equation}\label{eq:bounded-positive-single-lower}
    \int(X_bq)|F|^2
    \ge
    \left(c\Lambda^p-C\Lambda^{p-1-\alpha_p}\right)
    \mathcal H(F)-C_{\Lambda,h_0}M^2.
  \end{equation}
  The same estimates without signs show that
  \[
    \int|X_bq|\,|F|^2
    \le C_{\Lambda,h_0}\bigl(M^2+\mathcal H(F)\bigr).
  \]

  It remains to verify the boundary-flux limit. The multiplier \(q\) is uniformly bounded on \(\supp F\). For fixed \((z,v)\),
  \eqref{eq:bounded-positive-b-error} gives \(\bar c(z,\delta,v)\to c(z,v)\). If \(c(z,v)\ne0\), then \(|r|\to\infty\), so eventually
  \(G(r)=1/r\), and
  \[
    c(z,v)q(z,\delta,v)
    =\chi_0(\delta)
    \frac{c(z,v)}{\bar c(z,\delta,v)}\rho(\delta)^p
    \longrightarrow0.
  \]
  If \(c(z,v)=0\), the product vanishes for every \(\delta\). Bounded velocity support and compact support of \(F\) give an integrable
  dominator, so dominated convergence proves
  \[
    -\int c(z,v)q(z,\delta,v)|F(z,\delta,v)|^2
    \dd z\dd\mu(v)\longrightarrow0.
  \]
  Thus Lemma~\ref{lem:bounded-positive-admissible-multiplier} applies with \(B_0=0\).

  The lower bound on \(\Lambda\) announced above is fixed by requiring \(C\Lambda^{p-1-\alpha_p}\le(c/2)\Lambda^p\); after this choice, \(h_0\)
  is selected as in \eqref{eq:bounded-positive-h0}. The limiting Green identity and \eqref{eq:bounded-positive-single-product} give
  \[
    \frac c2\Lambda^p\mathcal H(F)
    \le C_{\Lambda,h_0}M^2
    +C_{\Lambda,h_0}NM
    +C\Lambda^{-1}N\mathcal H(F)^{1/2}.
  \]
  Young's inequality absorbs the last term and proves \eqref{eq:bounded-positive-global-hardy}. This includes \(p=1\), without a limiting
  argument in \(p\).
\end{proof}

\subsection{Boundary-weight recovery by dyadic regularisation}

Boundary-weight recovery instead uses a dyadic regularisation of the boundary slope. In contrast to the continuous family \(b_y\) in the Hardy
argument, each \(b_j\) below is independent of \(y\), and the cut-off \(\chi(y/\tau_j)\) assigns it to the corresponding height layer.

The formal multiplier \(-\sgn(c)|c|^{p-1}\) is discontinuous at \(c=0\) for \(p=1\) and has a singular velocity derivative there for \(1<p<2\);
moreover, it involves the merely H\"older slope \(b\). To treat every \(p\ge1\) uniformly, we regularise both the grazing singularity and the
slope dyadically. Set
\[
  a_0:=2\max\{1,L(1+B)\},\qquad a_j:=2^{-j}a_0.
\]
Choose \(0<\kappa\le1\) sufficiently small that \(CLH\kappa^{\alpha_p}\le1/16\), with \(C\) the constant in the H\"older mollifier estimate,
and define
\begin{equation}\label{eq:bounded-positive-dyadic-scales}
  \ell_j:=\kappa a_j^{p+1},\qquad
  \tau_j:=\kappa a_j^{p+2},\qquad
  b_j:=\eta_{\ell_j}*b,\qquad
  \bar c_j:=w-u\cdot b_j(z).
\end{equation}
At \(\alpha_p=1/(p+1)\), for \(|v|\le L\),
\[
  |\bar c_j-c|
  \le CL[b]_{\C^{0,\alpha_p}}\kappa^{\alpha_p}a_j.
\]
The choice of \(\kappa\) ensures that
\begin{equation}\label{eq:bounded-positive-dyadic-error}
  |\bar c_j-c|\le\frac{a_j}{16}
  \quad (j\ge0,\ |v|\le L).
\end{equation}
The exponents in \(\ell_j\) and \(\tau_j\) express the two balances used below:
\[
  \ell_j^{\alpha_p}
  =\kappa^{\alpha_p}a_j,
  \qquad
  \tau_j=\ell_ja_j=\kappa a_j^{p+2}.
\]
Thus \(b_j\) approximates \(b\) at the scale required to resolve flattened normal velocities of size \(a_j\), while the transport derivatives
cost \(a_j^{-2}\). On the support of the height cut-off, \(y\lesssim\tau_j\), and hence \(a_j^{-2}\lesssim y^{-2\beta}\), exactly the weight
controlled by the Hardy estimate for the sufficiently small scales whose layers lie below \(h_0\); the finitely many larger scales are absorbed
into \(M^2\).

Choose an odd \(\theta\in\C_c^\infty(\R)\) that is supported in \(\{1/4<|s|<4\}\), satisfies \(s\theta(s)\ge0\), and equals \(\sgn s\) on
\(\{3/4<|s|<9/4\}\). Let \(\chi\in\C_c^\infty(-1,2)\) equal one on \([0,1]\), and, for \(y>0\), define
\begin{equation}\label{eq:bounded-positive-dyadic-multiplier}
  m(z,y,v):=-\sum_{j\ge0}a_j^{p-1}
  \theta(\bar c_j/a_j)\chi(y/\tau_j).
\end{equation}
Write
\[
  m_j:=-a_j^{p-1}\theta(\bar c_j/a_j)\chi(y/\tau_j),
  \qquad m=\sum_{j\ge0}m_j.
\]

\begin{lemma}[Dyadic boundary multiplier]
  \label{lem:bounded-positive-dyadic-local-finiteness}
  \label{lem:bounded-positive-active-scales}
  \label{lem:bounded-positive-dyadic-derivatives}
  The series defining \(m\) is locally finite on \(\{y>0\}\). Consequently \(m\in\C^1(\{y>0\})\), its first derivatives are obtained term by
  term, and on every truncated half-space \(\{y\ge\delta\}\) the multiplier and its derivatives are finite sums.

  For every \((z,v)\) with \(|v|\le L\), at most five indices are active, meaning that \(\theta(\bar c_j/a_j)\ne0\) or \(\theta'(\bar
  c_j/a_j)\ne0\). Every active index satisfies
  \begin{equation}\label{eq:bounded-positive-active-comparison}
    \frac3{16}a_j\le|c|\le\frac{65}{16}a_j,
    \qquad
    \sgn\bar c_j=\sgn c.
  \end{equation}
  For \(|v|\le L\), the boundary value \(m(z,0,v)\) is therefore a finite sum and, for \(y\ge0\),
  \begin{equation}\label{eq:bounded-positive-boundary-coercivity}
    |m(z,y,v)|\le C,\qquad
    -c(z,v)m(z,0,v)\asymp|c(z,v)|^p.
  \end{equation}

  Finally, there is \(J_0\), depending only on the fixed data, such that for \(j\ge J_0\), pointwise on \(\{y>0,\ |v|\le L\}\cap\supp m_j\),
  \begin{equation}\label{eq:bounded-positive-dyadic-pointwise}
    |X_bm_j|\le Ca_j^{-2}\le Cy^{-2\beta},
    \qquad
    |\nabla_vm_j|^2\le Ca_j^{2p-4}\le Cy^{-2\beta}.
  \end{equation}
  The finitely many indices \(j<J_0\) contribute at most \(CM^2\) to the integrated product and derivative terms. Consequently,
  \begin{equation}\label{eq:bounded-positive-dyadic-product}
    \norm{mF}_{\L^2_{z,y}\H^1_\mu}^2
    +\int\bigl(|\nabla_vm|^2+|X_bm|\bigr)|F|^2
    \le C\bigl(M^2+\mathcal H(F)\bigr).
  \end{equation}
\end{lemma}

\begin{proof}
  Since \(\tau_j=\kappa a_0^{p+2}2^{-j(p+2)}\to0\), the height cut-off makes \(m_j\) vanish on every compact subset of \(\{y>0\}\) once \(j\) is
  sufficiently large. This proves local finiteness and the corresponding assertions on truncated half-spaces.

  If \(j\) is active, the support properties of \(\theta\) give
  \[
    \frac14a_j\le|\bar c_j|\le4a_j.
  \]
  Together with \eqref{eq:bounded-positive-dyadic-error}, this gives \eqref{eq:bounded-positive-active-comparison}; in particular, \(\bar c_j\)
  and \(c\) have the same sign. If \(c\ne0\), every active scale lies in
  \[
    \frac{16}{65}|c|\le a_j\le\frac{16}{3}|c|.
  \]
  Since \(a_j=2^{-j}a_0\) and \(65/3<2^5\), at most five indices are active. If \(c=0\), estimate \eqref{eq:bounded-positive-dyadic-error}
  shows that no index is active. Finite overlap, \(p\ge1\), and \(a_j\le a_0\) give \(|m|\le C\).

  If \(c\ne0\), choose \(j\ge0\) with \(a_j\le|c|<2a_j\), which is possible because \(|c|\le L(1+B)\le a_0/2\). Then
  \[
    \frac{15}{16}
    \le\frac{|\bar c_j|}{a_j}
    \le\frac{33}{16},
  \]
  and hence \(\theta(\bar c_j/a_j)=\sgn c\). This summand contributes a positive multiple of \(|c|^p\) to \(-cm(z,0,v)\). Every other active
  summand has the same sign and is at most \(C|c|^p\). Finite overlap proves \eqref{eq:bounded-positive-boundary-coercivity}; the case \(c=0\)
  is immediate. This also shows directly why \(p=1\) requires no limiting argument.

  It remains to estimate derivatives. The H\"older mollifier bound and \((p+1)\alpha_p=1\) give
  \[
    |\nabla_zb_j|\le CH\ell_j^{\alpha_p-1}\le Ca_j^{-p}.
  \]
  Since \(b_j\) is independent of \(y\),
  \[
    X_b\bar c_j=-u^{\mathsf T}(\nabla_zb_j)u,
    \qquad |X_b\bar c_j|\le Ca_j^{-p}
  \]
  on \(|v|\le L\). Direct differentiation yields
  \begin{align*}
    X_bm_j
    ={} & -a_j^{p-2}\theta'(\bar c_j/a_j)
    (X_b\bar c_j)\chi(y/\tau_j)           \\
        & -a_j^{p-1}\theta(\bar c_j/a_j)
    \chi'(y/\tau_j)\frac{c}{\tau_j},      \\
    \nabla_vm_j
    ={} & -a_j^{p-2}\theta'(\bar c_j/a_j)
    (-b_j,1)\chi(y/\tau_j).
  \end{align*}
  On an active scale, \(|c|\le Ca_j\), so
  \[
    |X_bm_j|\le Ca_j^{-2},
    \qquad
    |\nabla_vm_j|\le Ca_j^{p-2}.
  \]
  Choose \(J_0\) so that \(a_j\le1\) and \(2\tau_j\le h_0\) for \(j\ge J_0\). On the support of the corresponding height cut-off,
  \[
    y<2\kappa a_j^{p+2},
    \qquad
    a_j^{-2}\le Cy^{-2\beta},
    \qquad
    a_j^{2p-4}\le a_j^{-2}.
  \]
  This proves \eqref{eq:bounded-positive-dyadic-pointwise}.

  Uniform active-scale overlap gives
  \[
    \left|X_b\sum_{j\ge J_0}m_j\right|
    +\left|\nabla_v\sum_{j\ge J_0}m_j\right|^2
    \le Cy^{-2\beta}\one_{\{0<y<h_0\}}.
  \]
  For the finitely many indices \(j<J_0\), the preceding derivative formulae are uniformly bounded on \(\{|v|\le L\}\), so their integrated
  contribution is at most \(CM^2\). Thus
  \[
    \int\bigl(|\nabla_vm|^2+|X_bm|\bigr)|F|^2
    \le C\bigl(M^2+\mathcal H(F)\bigr).
  \]
  The product estimate \eqref{eq:bounded-positive-admissibility-product} now proves \eqref{eq:bounded-positive-dyadic-product}.
\end{proof}

\begin{proof}[Proof of Proposition~\ref{prop:bounded-positive-local}]
  The velocity restriction in the dyadic lemma is sufficient. Indeed, \(\supp_vF\cup\supp_v\nabla_vF\subset\overline B_L\), and
  \[
    \nabla_v(mF)=m\nabla_vF+(\nabla_vm)F.
  \]
  The Green identity contains the remaining multiplier factors only as \(mF\), \((X_bm)|F|^2\), and \(cm|F|^2\). Hence no estimate for \(m\) or
  its derivatives outside \(\overline B_L\) is required.

  Proposition~\ref{prop:bounded-positive-global-hardy} and Lemma~\ref{lem:bounded-positive-dyadic-derivatives} give
  \begin{equation}\label{eq:bounded-positive-dyadic-controlled}
    \norm{mF}_{\L^2_{z,y}\H^1_\mu}^2
    +\int\bigl(|\nabla_vm|^2+|X_bm|\bigr)|F|^2
    \le C(M^2+N^2).
  \end{equation}
  Lemma~\ref{lem:bounded-positive-dyadic-local-finiteness} shows that the multiplier and its first derivatives are finite sums on
  \(\{y\ge\delta\}\). Estimate \eqref{eq:bounded-positive-dyadic-controlled} verifies the integral hypotheses of
  Lemma~\ref{lem:bounded-positive-admissible-multiplier}. Hence the truncated Green identity applies for every \(\delta>0\).

  It remains to treat the boundary term. Since \(F(z,\delta,v)=0\) for \(|v|>L\), all boundary integrals may be restricted to \(\overline B_L\).
  For fixed \((z,v)\) in this ball, only finitely many indices are active. For each such index, \(\chi(\delta/\tau_j)\to\chi(0)=1\), and hence
  \[
    m(z,\delta,v)\longrightarrow m(z,0,v).
  \]
  Also \(F(z,\delta,v)\to F(z,0,v)\). The uniform bound for \(m\), the estimate \(|c|\le L(1+B)\), and compact support of \(F\) provide an
  integrable dominator of the form \(C\one_{K_z}(z)\one_{\overline B_L}(v)\). Dominated convergence gives
  \[
    -\int c\,m(z,\delta,v)|F(z,\delta,v)|^2\dd z\dd\mu
    \longrightarrow
    -\int c\,m(z,0,v)|F(z,0,v)|^2\dd z\dd\mu.
  \]

  Thus \eqref{eq:bounded-positive-admissibility-boundary} holds with
  \[
    B_0=-\int c\,m(z,0,v)|F(z,0,v)|^2\dd z\dd\mu.
  \]
  Lemma~\ref{lem:bounded-positive-admissible-multiplier} gives the limiting Green identity. Applying duality yields
  \[
    -\int c\,m(z,0,v)|F(z,0,v)|^2\dd z\dd\mu
    \le\int|X_bm|\,|F|^2
    +2N\norm{mF}_{\L^2\H^1_\mu}.
  \]
  Boundary coercivity from \eqref{eq:bounded-positive-boundary-coercivity} is used only where \(F(z,0,v)\ne0\), and hence only for \(|v|\le L\).
  Combining it with \eqref{eq:bounded-positive-dyadic-controlled} gives
  \[
    \int|c|^p|F(z,0,v)|^2\dd z\dd\mu
    \le C(M^2+N^2),
  \]
  which is \eqref{eq:bounded-positive-local}.
\end{proof}

\subsection{Globalisation}

We now collect the three ingredients needed to pass from the local graph estimate to a boundary atlas: rotation invariance, quantitative
extension of the boundary slope, and smooth flattening with localisation.

\begin{lemma}[Rotation invariance of the velocity norms]
  \label{lem:velocity-rotation}
  Let \(Q\in O(d)\). On \(\H^1_\mu\), define \(U_Qg(v)=g(Q^{\mathsf T}v)\); on \(\H^{-1}_\mu\), define \(U_Q\) by
  \[
    \langle U_Qq,\varphi\rangle
    :=\langle q,U_{Q^{\mathsf T}}\varphi\rangle.
  \]
  For \(\mu=\mathcal L^d\) or \(\gamma_d\),
  \[
    \norm{U_Qg}_{\H^1_\mu}=\norm{g}_{\H^1_\mu},
    \qquad
    \norm{U_Qq}_{\H^{-1}_\mu}=\norm{q}_{\H^{-1}_\mu}.
  \]
  The same identities hold on \(\Sph^{d-1}\), and \(U_Q\) preserves every velocity ball \(\overline B_L\).
\end{lemma}

\begin{proof}
  Lebesgue measure and the standard Gaussian measure are orthogonally invariant, and
  \[
    \nabla_v(U_Qg)=Q(\nabla_vg)\circ Q^{\mathsf T}.
  \]
  This proves the \(\H^1_\mu\) identity; the dual definition gives the \(\H^{-1}_\mu\) identity. Orthogonal maps are isometries of the sphere and
  preserve surface measure, tangential gradient, and Euclidean velocity balls.
\end{proof}

\begin{lemma}[Quantitative global H\"older extension]
  \label{lem:bounded-positive-holder-extension}
  Let \(E\subset\R^{d-1}\) be nonempty and let \(0<\alpha\le1\). Every bounded \(b\in\C^{0,\alpha}(E;\R^{d-1})\) has an extension \(\widetilde
  b\in\C^{0,\alpha}(\R^{d-1};\R^{d-1})\) such that
  \begin{equation}\label{eq:bounded-positive-holder-extension}
    \widetilde b=b\quad\text{on }E,\qquad
    \norm{\widetilde b}_{\L^\infty}
    \le\norm b_{\L^\infty(E)},\qquad
    [\widetilde b]_{\C^{0,\alpha}}
    \le\sqrt{d-1}\,[b]_{\C^{0,\alpha}(E)}.
  \end{equation}
  If \(E\) is bounded and \(\alpha_p\le\alpha\), then
  \begin{equation}\label{eq:bounded-positive-holder-reduction}
    [b]_{\C^{0,\alpha_p}(E)}
    \le\max\{1,\operatorname{diam}E\}^{\alpha-\alpha_p}
      [b]_{\C^{0,\alpha}(E)}.
  \end{equation}
\end{lemma}

\begin{proof}
  Set \(B=\norm b_{\L^\infty(E)}\) and \(H=[b]_{\C^{0,\alpha}(E)}\). The distance
  \[
    d_\alpha(z,\xi):=|z-\xi|^\alpha
  \]
  is a metric because \(0<\alpha\le1\). For \(1\le k\le d-1\), define
  \[
    e_k(z):=\inf_{\xi\in E}
    \bigl(b_k(\xi)+Hd_\alpha(z,\xi)\bigr).
  \]
  The scalar McShane extension theorem \cite[Theorem~1]{McShane1934}, applied to the metric \(d_\alpha\), gives \(e_k=b_k\) on \(E\) and
  \([e_k]_{\C^{0,\alpha}}\le H\). Thus \(e=(e_1,\ldots,e_{d-1})\) has H\"older seminorm at most \(\sqrt{d-1}H\).

  Let \(P_B\) be the metric projection of \(\R^{d-1}\) onto the closed ball \(\overline B_B\). It is one-Lipschitz and fixes \(b(E)\). Hence
  \(\widetilde b:=P_B\circ e\) agrees with \(b\) on \(E\), is bounded by \(B\), and has the stated seminorm. Finally, for \(z,\xi\in E\),
  \[
    |z-\xi|^\alpha
    \le\max\{1,\operatorname{diam}E\}^{\alpha-\alpha_p}
    |z-\xi|^{\alpha_p},
  \]
  which proves \eqref{eq:bounded-positive-holder-reduction}.
\end{proof}

\begin{lemma}[Smooth flattening and localisation]
  \label{lem:bounded-positive-flattening}
  Let \(\mu\in\{\mathcal L^d,\gamma_d\}\). Let \(P\subset\R^{d-1}\) be bounded and open, let \(\Phi\in\C^1(P)\), and set
  \[
    b:=\nabla\Phi,\qquad
    \Psi(z,y):=(z,y+\Phi(z)).
  \]
  Let \(D\subset P\times(0,\infty)\) be open, put \(O:=\Psi(D)\), and assume that
  \[
    b\in\C^{0,\alpha_p}(P;\R^{d-1})
    \cap\L^\infty(P;\R^{d-1}).
  \]
  Let \(\widetilde b\in \C^{0,\alpha_p}(\R^{d-1};\R^{d-1})\cap\L^\infty\) be an extension of \(b\).

  Suppose that \(f\in\C^1(\overline O\times\R^d)\) has velocity support in \(\overline B_L\), and let \(\zeta\in\C_c^1(\R^d)\) satisfy
  \[
    \operatorname{dist}\bigl(
    \operatorname{supp}(\zeta\circ\Psi),
    \partial D\cap\{y>0\}\bigr)>0.
  \]
  Define
  \[
    F_\zeta(z,y,v):=(\zeta f)(\Psi(z,y),v)
    \qquad ((z,y)\in D),
  \]
  and extend it by zero across the artificial faces \(\partial D\cap\{y>0\}\), retaining \(y=0\) as the physical boundary. Denote the extension
  by \(\widetilde F_\zeta\). Then
  \[
    \widetilde F_\zeta
    \in\C_c^1(\R^{d-1}\times[0,\infty)\times\R^d),
    \qquad
    \supp_v\widetilde F_\zeta\subset \overline B_L,
  \]
  and
  \begin{equation}\label{eq:bounded-positive-flattening-transport}
    X_bF_\zeta
    =(v\cdot\nabla_x(\zeta f))\circ\Psi
    \quad\text{on }D.
  \end{equation}
  After zero extension,
  \begin{equation}\label{eq:bounded-positive-flattening-zero-extension}
    X_{\widetilde b}\widetilde F_\zeta
    =\widetilde{(v\cdot\nabla_x(\zeta f))\circ\Psi}
    \quad\text{on }\R^{d-1}\times(0,\infty),
  \end{equation}
  where the tilde on the right-hand side denotes zero extension across the artificial faces. Moreover,
  \begin{align}
    \norm{\widetilde F_\zeta}_{\L^2_{z,y}\H^1_\mu}
     & =\norm{\zeta f}_{\L^2(O;\H^1_\mu)},\notag \\
    \norm{X_{\widetilde b}\widetilde F_\zeta}
    _{\L^2_{z,y}\H^{-1}_\mu}
     & =\norm{v\cdot\nabla_x(\zeta f)}
    _{\L^2(O;\H^{-1}_\mu)}.
    \label{eq:bounded-positive-flattening-norms}
  \end{align}
  Consequently,
  \begin{align}
     & \norm{\widetilde F_\zeta}_{\L^2_{z,y}\H^1_\mu}
    +\norm{X_{\widetilde b}\widetilde F_\zeta}
    _{\L^2_{z,y}\H^{-1}_\mu}\notag                    \\
     & \qquad\le C\left(
    \norm{\zeta}_{\L^\infty}
    +L\norm{\nabla_x\zeta}_{\L^\infty}\right)
    \left(
    \norm f_{\L^2(O;\H^1_\mu)}
    +\norm{v\cdot\nabla_xf}_{\L^2(O;\H^{-1}_\mu)}
    \right).
    \label{eq:bounded-positive-flattening-localisation}
  \end{align}

  On the lower graph \(x_d=\Phi(z)\), set \(n_\Phi=(b,-1)/(1+|b|^2)^{1/2}\). Then
  \begin{equation}\label{eq:bounded-positive-flattening-surface}
    |v\cdot n_\Phi|^p\dd S
    =|w-u\cdot b(z)|^p
    (1+|b(z)|^2)^{(1-p)/2}\dd z.
  \end{equation}
\end{lemma}

\begin{proof}
  The matrix \(D\Psi\) is triangular with determinant one. The ordinary chain rule gives
  \[
    X_bF_\zeta
    =u\cdot\nabla_zF_\zeta
    +(w-u\cdot b(z))\partial_yF_\zeta
    =\bigl(v\cdot\nabla_x(\zeta f)\bigr)\circ\Psi,
  \]
  proving \eqref{eq:bounded-positive-flattening-transport}. The positive-distance hypothesis makes \(F_\zeta\) identically zero in a collar of
  every artificial face. Its zero extension is therefore \(\C^1\), creates no artificial-face distribution, and proves
  \eqref{eq:bounded-positive-flattening-zero-extension}.

  To verify compact support, let \(\delta_*>0\) be the distance in the hypothesis and set \(K:=\pi_{\R^{d-1}}(\supp\zeta)\). Whenever
  \(F_\zeta(z,y,v)\ne0\) for some \(v\) and \(y>0\), we have
  \[
    z\in K,\qquad \operatorname{dist}(z,P^c)\ge\delta_*,
  \]
  since a horizontal path from \(z\) to \(P^c\) would otherwise cross an artificial face within distance \(\delta_*\). Thus \(z\) belongs to a
  compact subset of \(P\), on which \(\Phi\) is bounded. Because \(y+\Phi(z)\) lies in the bounded \(x_d\)-projection of \(\supp\zeta\), \(y\)
  is bounded as well. Together with the uniform velocity-support bound, this proves that \(\widetilde F_\zeta\) is compactly supported.

  Since \(\det D\Psi=1\), change of variables proves \eqref{eq:bounded-positive-flattening-norms}. The product rule gives
  \[
    v\cdot\nabla_x(\zeta f)
    =\zeta\,v\cdot\nabla_xf+(v\cdot\nabla_x\zeta)f,
  \]
  while the velocity-support restriction implies
  \[
    \norm{(v\cdot\nabla_x\zeta)f}_{\H^{-1}_\mu}
    \le L\norm{\nabla_x\zeta}_{\L^\infty}
    \norm f_{\L^2_\mu}.
  \]
  This proves \eqref{eq:bounded-positive-flattening-localisation}. Finally,
  \[
    \dd S=(1+|b|^2)^{1/2}\dd z,\qquad
    |v\cdot n_\Phi|
    =\frac{|w-u\cdot b|}{(1+|b|^2)^{1/2}},
  \]
  which proves \eqref{eq:bounded-positive-flattening-surface}.
\end{proof}

\begin{proof}[Proof of Theorem~\ref{thm:bounded-positive}]
  By the initial reduction, it is enough to work at \(\alpha=\alpha_p\). Quantitatively, on a bounded chart base \(P\),
  \eqref{eq:bounded-positive-holder-reduction} controls the \(\C^{0,\alpha_p}\) seminorm by the given \(\C^{0,\alpha}\) seminorm.

  Cover \(\partial\Omega\) by finitely many graph neighbourhoods \(U_j\), using rigid rotations applied simultaneously to \(x\) and \(v\).
  Lemma~\ref{lem:velocity-rotation} preserves the velocity norms and support under these rotations. Choose smaller neighbourhoods \(V_j\Subset
  U_j\) that still cover the boundary. Let \(P_j\) denote the tangential base of the \(j\)-th chart. In its coordinates,
  \[
    \begin{aligned}
      \Omega\cap U_j
                        & =\{(z,x_d)\in U_j:x_d>\Phi_j(z)\},                                    \\
      b_j^{\mathrm{ch}} & :=\nabla\Phi_j\in\C^{0,\alpha_p},
                        & c_j^{\mathrm{ch}}(z,v)             & :=w-u\cdot b_j^{\mathrm{ch}}(z).
    \end{aligned}
  \]
  Choose smooth cut-offs \(0\le\zeta_j\le1\), supported in \(V_j\) and at a positive distance from each artificial face, such that
  \(\sum_j\zeta_j^2=1\) near \(\partial\Omega\), and set \(f_j=\zeta_jf\). Lemma~\ref{lem:bounded-positive-holder-extension} gives a global field
  \(\widetilde b_j^{\mathrm{ch}}\) that agrees with \(b_j^{\mathrm{ch}}\) on \(P_j\) and satisfies
  \[
    \norm{\widetilde b_j^{\mathrm{ch}}}_{\L^\infty}
    \le\norm{b_j^{\mathrm{ch}}}_{\L^\infty(P_j)},\qquad
    [\widetilde b_j^{\mathrm{ch}}]_{\C^{0,\alpha_p}}
    \le\sqrt{d-1}\,
    [b_j^{\mathrm{ch}}]_{\C^{0,\alpha_p}(P_j)}.
  \]

  Flatten the physical boundary by
  \[
    y=x_d-\Phi_j(z),\qquad
    F_j(z,y,v)=f_j(z,y+\Phi_j(z),v).
  \]
  Let \(D_j\subset P_j\times(0,\infty)\) be the interior of the flattened chart. Because \(\supp\zeta_j\Subset V_j\Subset U_j\), the function
  \(F_j\) vanishes near every artificial lateral and upper face. Extend it by zero across only those faces, not across \(y=0\), and denote the
  extension by \(\widetilde F_j\). Apply Lemma~\ref{lem:bounded-positive-flattening} with \(f|_{\Omega\cap U_j}\), \(\zeta_j\), and \(\widetilde
  b_j^{\mathrm{ch}}\). It gives \eqref{eq:bounded-positive-flattening-zero-extension} and the norm bound
  \eqref{eq:bounded-positive-flattening-localisation}. In particular,
  \[
    X_{\widetilde b_j^{\mathrm{ch}}}\widetilde F_j
    =\widetilde{(v\cdot\nabla_xf_j)\circ\Psi_j}
    \quad\text{in }
    \mathcal D'(\R^{d-1}\times(0,\infty);\H^{-1}_\mu),
  \]
  where \(\Psi_j(z,y)=(z,y+\Phi_j(z))\). The lemma also places \(\widetilde F_j\) in precisely the test class of
  Proposition~\ref{prop:bounded-positive-local}. For the graph boundary, \eqref{eq:bounded-positive-flattening-surface} gives
  \begin{equation}\label{eq:bounded-positive-surface-factor}
    |v\cdot n_\Omega|^p\dd S
    =|c_j^{\mathrm{ch}}(z,v)|^p
    (1+|b_j^{\mathrm{ch}}(z)|^2)^{(1-p)/2}\dd z
    \le |c_j^{\mathrm{ch}}(z,v)|^p\dd z.
  \end{equation}
  Together with Lemma~\ref{lem:velocity-rotation}, this shows that Proposition~\ref{prop:bounded-positive-local} applies in every chart.

  On the boundary, \(\sum_j|f_j|^2=|f|^2\). To make the summation quantitative, set
  \[
    M_{\mathcal A}:=\sup_{x\in\Omega}
    \#\{j:x\in\operatorname{supp}\zeta_j\},
    \qquad
    Z_{\mathcal A}:=\max_j\|\nabla\zeta_j\|_\infty.
  \]
  Thus \(M_{\mathcal A}\) is the overlap multiplicity of the cut-off supports. The product rule and the fixed velocity support give
  \begin{align*}
    \sum_j\norm{f_j}_{\L^2_x\H^1_\mu}^2
     & \le M_{\mathcal A}\norm f_{\L^2_x\H^1_\mu}^2, \\
    \sum_j\norm{v\cdot\nabla_xf_j}_{\L^2_x\H^{-1}_\mu}^2
     & \le2M_{\mathcal A}
    \norm{v\cdot\nabla_xf}_{\L^2_x\H^{-1}_\mu}^2     \\
     & \quad+2L^2M_{\mathcal A}Z_{\mathcal A}^2
    \norm f_{\L^2_x\L^2_\mu}^2.
  \end{align*}
  Summing the local estimates and using these two bounds proves \eqref{eq:bounded-positive}.
\end{proof}

\begin{remark}[Quantitative geometric dependence]
  \label{rem:quantitative-atlas}
  Fix a finite graph atlas \(\mathcal A=\{(U_j,V_j,\Phi_j,\zeta_j)\}_{j=1}^{N_{\rm ch}}\) of the type used in the preceding proof. Define
  \[
    B_{\mathcal A}:=\max_j\|\nabla\Phi_j\|_\infty,
    \quad
    H_{\mathcal A}:=\max_j[\nabla\Phi_j]_{\C^{0,\alpha_p}},
    \quad
    Z_{\mathcal A}:=\max_j\|\nabla\zeta_j\|_\infty.
  \]
  Lemma~\ref{lem:bounded-positive-holder-extension} supplies the global slope extensions with bounds depending only on \(B_{\mathcal
      A},H_{\mathcal A}\), and \(d\); they introduce no further atlas parameter. If \(M_{\mathcal A}\) bounds the overlap multiplicity, the
  displayed
  proof gives
  \[
    C=C\!\left(
    p,d,L,B_{\mathcal A},H_{\mathcal A},
    Z_{\mathcal A},M_{\mathcal A}\right).
  \]
  The chart number, chart radii, and bi-Lipschitz constants enter only through the listed slope, cut-off, and overlap bounds; they need not be
  recorded as additional independent parameters.
\end{remark}

\begin{remark}[Dependence on the support radius]
  \label{rem:L-dependence}
  For a bounded \(\C^{1,1}\) domain \(\Omega\subset\R^d\), \(d\ge2\), and \(L>0\), let
  \[
    C^\gamma_{p,\Omega}(L):=
    \sup_{\substack{0\ne f\in
        \C_c^\infty(\overline\Omega\times\R^d)\\
        \supp_vf\subset\overline B_L}}
    \frac{\mathcal T^\gamma_{p,\Omega}(f)}
    {\mathcal E^\gamma_\Omega(f)}.
  \]
  For $1 \le p <2$, the dependence on \(L\) is unavoidable. The high-tangential-velocity packets in Theorem~\ref{thm:unbounded} have tangential
  scale
  \(R=h^{-1/3}\), support in \(B_{R+C_0}\) for a fixed \(C_0\), uniformly bounded energy, and a weighted boundary functional of size
  \(h^{(p-2)/6}=R^{(2-p)/2}\). Consequently, for all sufficiently large \(L\),
  \begin{equation}\label{eq:L-lower-growth}
    C^\gamma_{p,\Omega}(L)
    \gtrsim_{p,\Omega}L^{(2-p)/2},
    \qquad 1\le p<2.
  \end{equation}
  In particular, the constant in the natural trace estimate grows at least like \(L^{1/2}\). For every \(p\ge2\),
  Proposition~\ref{prop:quadratic-gaussian} and \(\omega_p\le\omega_2\) instead give an \(L\)-uniform \(\omega_p\)-trace estimate.

  The support restriction enters the proof at four stages. In the Hardy argument, the error \(e=u\cdot(b_y-b)\) and the term \(-u^{\mathsf
      T}(\nabla b_y)u\) in \(X_b\bar c\) force the absorption parameter \(\Lambda\) to depend on \(L\). In the dyadic boundary-weight recovery
  argument, the estimate
  \[
    |\bar c_j-c|\le CLH\kappa^{\alpha_p}a_j
  \]
  forces \(\kappa\) to depend on \(L\). Globalisation introduces \(L\) through \(v\cdot\nabla_x\zeta_j\). Finally, for Gaussian measure, the
  weighted anchored inequality in Lemma~\ref{lem:bounded-positive-single-profile} uses Gaussian--Lebesgue comparability on an \(L\)-dependent
  velocity region. For $1 \le p <2$, the proof yields the finite constant described in Remark~\ref{rem:quantitative-atlas}, but does not provide
  a
  matching upper growth rate. For comparison, the translated Lebesgue packets in Proposition~\ref{prop:interval-lebesgue-failure} give a linear
  lower bound in \(L\) for the corresponding \(\omega_p\)-constant on every bounded \(\C^1\) domain.
\end{remark}

\subsection{Spherical velocity model by radial thickening}

\begin{lemma}[Radial thickening]
  \label{lem:spherical-radial-lift}
  Fix a nonzero \(\vartheta\in \C_c^\infty(1/2,2)\). For \(g\in \H^1(\Sph^{d-1})\), define \(G(r\theta)=\vartheta(r)g(\theta)\). Then
  \[
    \norm{G}_{\H^1(\R^d)}
    \le C_\vartheta\norm{g}_{\H^1(\Sph^{d-1})}.
  \]
  Moreover, there is a bounded map
  \[
    \mathcal R_\vartheta\colon
    \H^{-1}(\Sph^{d-1})\longrightarrow \H^{-1}(\R^d)
  \]
  such that, for functions \(h\),
  \[
    (\mathcal R_\vartheta h)(r\theta)=r\vartheta(r)h(\theta),
    \qquad
    \norm{\mathcal R_\vartheta h}_{\H^{-1}(\R^d)}
    \le C_\vartheta\norm{h}_{\H^{-1}(\Sph^{d-1})}.
  \]
\end{lemma}

\begin{proof}
  Polar coordinates and orthogonality of radial and tangential derivatives give
  \[
    \nabla_vG(r\theta)
    =\vartheta'(r)g(\theta)\theta
    +\frac{\vartheta(r)}r
    \nabla_{\Sph^{d-1}}g(\theta).
  \]
  Because \(\vartheta\) is smooth and supported in the fixed annulus \(1/2<r<2\), integration in \(r\) yields
  \[
    \norm{G}_{\H^1(\R^d)}
    \le C_\vartheta\norm{g}_{\H^1(\Sph^{d-1})}.
  \]

  For \(\varphi\in\C_c^\infty(\R^d)\), define
  \[
    A\varphi(\theta)
    :=\int_0^\infty\vartheta(r)r^d\varphi(r\theta)\dd r.
  \]
  Cauchy--Schwarz in \(r\), together with
  \[
    \nabla_{\Sph^{d-1}}[\varphi(r\,\cdot)](\theta)
    =r(I-\theta\otimes\theta)\nabla_v\varphi(r\theta),
  \]
  gives
  \[
    \norm{A\varphi}_{\H^1(\Sph^{d-1})}
    \le C_\vartheta\norm{\varphi}_{\H^1(\R^d)}.
  \]
  Define \(\mathcal R_\vartheta:=A^*\). Then
  \[
    \norm{\mathcal R_\vartheta h}_{\H^{-1}(\R^d)}
    \le C_\vartheta
    \norm{h}_{\H^{-1}(\Sph^{d-1})}.
  \]
  For functions \(h\), polar coordinates identify
  \[
    (\mathcal R_\vartheta h)(r\theta)
    =r\vartheta(r)h(\theta).
  \]
\end{proof}

\begin{corollary}[Spherical \(\omega_p\)-trace estimate]
  \label{cor:spherical-positive}
  Let \(d\ge2\), \(1\le p<\infty\), and \(\alpha_p\le\alpha\le1\). If \(\Omega\subset\R^d\) is a bounded \(\C^{1,\alpha}\) domain, then every
  \(f\in\C^\infty(\overline\Omega\times\Sph^{d-1})\) satisfies
  \begin{align*}
     & \int_{\partial\Omega}\int_{\Sph^{d-1}}
    |f(x,v)|^2|v\cdot n_\Omega(x)|^p\dd\sigma(v)\dd S(x) \\
     & \qquad\le C\left(
    \norm{f}_{\L^2(\Omega;\H^1(\Sph^{d-1}))}^2+
    \norm{v\cdot\nabla_xf}_{\L^2(\Omega;\H^{-1}(\Sph^{d-1}))}^2
    \right).
  \end{align*}
  Because \(|v\cdot n_\Omega|\le1\) on the sphere, this is exactly the \(\omega_p\)-trace estimate.
\end{corollary}

\begin{proof}
  Choose a nonzero \(\vartheta\in\C_c^\infty(1/2,2)\), and, for \(v=r\theta\), define
  \[
    F(x,r\theta):=\vartheta(r)f(x,\theta).
  \]
  Lemma~\ref{lem:spherical-radial-lift} gives
  \[
    \norm{F}_{\L^2(\Omega;\H^1(\R^d))}
    \le C_\vartheta
    \norm{f}_{\L^2(\Omega;\H^1(\Sph^{d-1}))}.
  \]
  Moreover,
  \[
    v\cdot\nabla_xF(x,r\theta)
    =r\vartheta(r)\,
    \theta\cdot\nabla_xf(x,\theta)
    =\mathcal R_\vartheta
    \bigl(\theta\cdot\nabla_xf\bigr)(x,r\theta).
  \]
  The boundedness of \(\mathcal R_\vartheta\) therefore yields
  \[
    \norm{v\cdot\nabla_xF}
    _{\L^2(\Omega;\H^{-1}(\R^d))}
    \le C_\vartheta
    \norm{\theta\cdot\nabla_xf}
    _{\L^2(\Omega;\H^{-1}(\Sph^{d-1}))}.
  \]
  Polar coordinates give
  \[
    \int_{\R^d}|F(x,v)|^2|v\cdot n_\Omega(x)|^p\dd v
    =
    \left(
    \int_0^\infty\vartheta(r)^2r^{d-1+p}\dd r
    \right)
    \int_{\Sph^{d-1}}|f(x,\theta)|^2
    |\theta\cdot n_\Omega(x)|^p\dd\sigma(\theta).
  \]
  The radial constant is positive. Applying Theorem~\ref{thm:bounded-positive} with Lebesgue velocity measure and dividing by this constant
  proves the claim.
\end{proof}

\begin{remark}[Sharp boundary regularity threshold]
  \label{rem:bounded-positive-sharp}
  Together with Theorems~\ref{thm:bounded-euclidean} and~\ref{thm:spherical}, the positive results show that
  \[
    \alpha_p=\frac1{p+1}
  \]
  is the sharp boundary regularity threshold for bounded \(\C^{1,\alpha}\) domains in the bounded-support Euclidean model and in the spherical
  velocity model. Below the threshold, failure is existential rather than universal.
\end{remark}

\section{\texorpdfstring{Trace operators and Green's formula}
  {Trace operators and Green's formula}}
\label{sec:energy}

We now pass the \(\omega_p\)-trace estimates to their graph-space completions. On a half-space, this yields \(\omega_p\)-trace operators for the
unrestricted Lebesgue and unrestricted Gaussian models. On bounded domains in \(d\ge2\), it yields \(\omega_p\)-trace operators for the
unrestricted Gaussian model when \(p\ge2\), and at the stated thresholds for the two bounded-velocity models. Green's formula below is proved for
the natural trace operators in the unrestricted half-space models and, on bounded domains, in the two bounded-velocity models. Density is used
only for approximation. We collect all density statements and their provenance first.

\subsection{Density}
\label{sec:functional-density}

For convenience, we recall the complete graph spaces introduced above:
\begin{align}
  \mathcal X_\mu(\Omega)
   & =\{f\in \L^2(\Omega;\H^1_\mu):
  v\cdot\nabla_xf\in \L^2(\Omega;\H^{-1}_\mu)\}, \\
  \mathcal X_{\mu,L}(\Omega)
   & =\left\{f\in\mathcal X_\mu(\Omega):
  \operatorname*{ess\,supp}_v f\subset\overline B_L\right\},
  \qquad \qquad \mu\in\{\mathcal L^d,\gamma_d\}, \quad L>0
  \\
  \mathcal X_\sigma(\Omega)
   & =\left\{f\in \L^2(\Omega;\H^1(\Sph^{d-1})):
  v\cdot\nabla_xf\in \L^2(\Omega;\H^{-1}(\Sph^{d-1}))\right\}.
\end{align}

For \(L>0\), let
\[
  \mathcal D_L(\overline\Omega)
  :=\left\{F|_\Omega:\ F\in \C_c^\infty(\R_x^d\times\R_v^d),\
  \supp_vF\Subset B_L\right\}.
\]
Set
\[
  \mathcal D_c(\overline\Omega)
  :=\bigcup_{L>0}\mathcal D_L(\overline\Omega).
\]

\begin{proposition}[Density of smooth functions on the half-space]
  \label{prop:flat-intrinsic-density}
  For either the Lebesgue or the standard Gaussian velocity measure, \(\C_c^\infty(\overline{\mathbb H^d}\times\R^d)\) is dense in \(\mathcal
  X_\mu(\mathbb H^d)\) in the graph norm.
\end{proposition}

The proof is given in Appendix~\ref{app:flat-density}.

\begin{proposition}[Density in the unrestricted Gaussian model]
  \label{prop:bounded-gaussian-density}
  Let \(\Omega\subset\R^d\) be a bounded \(\C^1\) domain. Then \(\mathcal D_c(\overline\Omega)\) is dense in \(\mathcal X_\gamma(\Omega)\).
\end{proposition}

\begin{proof}
  This is \cite[Proposition~2.2]{AAMN2024}. Their space \(\H^1_{\mathrm{hyp}}(\Omega)\) and its graph norm agree with \(\mathcal
  X_\gamma(\Omega)\) and the norm used here. Since \(\Omega\) is bounded, their class \(\C_c^\infty(\overline\Omega\times\R^d)\) agrees with
  \(\mathcal D_c(\overline\Omega)\).
\end{proof}

\begin{proposition}[Density in the bounded-support Euclidean model]
  \label{prop:bounded-support-density}
  Let \(\Omega\subset\R^d\) be a bounded \(\C^1\) domain, let \(\mu\in\{\mathcal L^d,\gamma_d\}\), and let \(L>0\). Then \(\mathcal
  D_L(\overline\Omega)\) is dense in \(\mathcal X_{\mu,L}(\Omega)\).
\end{proposition}

The proof is given in Appendix~\ref{app:density}.

\begin{proposition}[Density in the unrestricted Lebesgue model]
  \label{prop:bounded-lebesgue-density}
  Let \(\Omega\subset\R^d\) be a bounded \(\C^1\) domain. Then \(\mathcal D_c(\overline\Omega)\) is dense in \(\mathcal X_{\mathcal
      L^d}(\Omega)\).
\end{proposition}

\begin{proof}
  Let \(f\in\mathcal X_{\mathcal L^d}(\Omega)\) and use the velocity cut-offs \(\chi_R\) from Lemma~\ref{lem:density-velocity-cutoff}. Since
  \(\chi_R\) is independent of \(x\),
  \[
    v\cdot\nabla_x(\chi_Rf)
    =\chi_R(v\cdot\nabla_xf)
  \]
  in distributions. The lemma and dominated convergence for the corresponding Bochner norms give
  \[
    \chi_Rf\longrightarrow f
    \qquad\text{in }\mathcal X_{\mathcal L^d}(\Omega).
  \]
  Each \(\chi_Rf\) has velocity support in a fixed ball, so Proposition~\ref{prop:bounded-support-density} applies. A diagonal sequence as
  \(R\to\infty\) gives approximants in \(\mathcal D_c(\overline\Omega)\).
\end{proof}

\begin{proposition}[Density in the spherical velocity model]
  \label{prop:spherical-density}
  Let \(d\ge2\) and let \(\Omega\subset\R^d\) be a bounded \(\C^1\) domain. Then \(\C^\infty(\overline\Omega\times\Sph^{d-1})\) is dense in
  \(\mathcal X_\sigma(\Omega)\).
\end{proposition}

\begin{proof}
  Brunken and Smetana prove the corresponding space--time density statement \cite[Proposition~3.1]{BrunkenSmetana2022}; a bounded \(\C^1\) domain
  satisfies their globally Lipschitz, piecewise \(\C^1\) spatial-domain hypothesis. Although their proposition is stated under the standing
  restriction \(d\in\{2,3\}\), its proof is dimension-independent: the spatial regularisation is Bochner-valued, while the velocity
  regularisation uses only a smooth \(\H^1(\Sph^{d-1})\)-basis and the boundedness of multiplication by the smooth coordinate functions on
  \(\H^1(\Sph^{d-1})\) \cite[Supplement, Section~SM1]{BrunkenSmetana2022}. Inspection of its proof suggests that the same argument extends to
  every \(d\ge2\).

  To pass from their space--time statement to the stationary one, put \(I=(0,1)\) and lift \(f\in\mathcal X_\sigma(\Omega)\) to \(\widetilde
  f(t,x,v)=f(x,v)\). Choose
  \[
    \Phi_n\in
    \C^\infty(\overline{I\times\Omega}\times\Sph^{d-1})
  \]
  such that
  \begin{align*}
    \Phi_n
     & \to\widetilde f
     &                                            & \text{in }\L^2(I\times\Omega;\H^1(\Sph^{d-1})),    \\
    (\partial_t+v\cdot\nabla_x)\Phi_n
     & \to(\partial_t+v\cdot\nabla_x)\widetilde f
     &                                            & \text{in }\L^2(I\times\Omega;\H^{-1}(\Sph^{d-1})).
  \end{align*}
  Choose \(\eta\in\C_c^\infty(I)\) with \(\int_I\eta=1\), and set
  \[
    f_n(x,v):=\int_I\eta(t)\Phi_n(t,x,v)\dd t.
  \]
  Then \(f_n\in \C^\infty(\overline\Omega\times\Sph^{d-1})\). Writing \(G_n=\Phi_n-\widetilde f\), integration by parts in \(t\) gives
  \[
    v\cdot\nabla_x(f_n-f)
    =\int_I\eta\,(\partial_t+v\cdot\nabla_x)G_n\dd t
    +\int_I\eta'G_n\dd t.
  \]
  Since \(\H^1(\Sph^{d-1})\hookrightarrow\H^{-1}(\Sph^{d-1})\), Cauchy--Schwarz in \(t\) yields
  \[
    \norm{f_n-f}_{\mathcal X_\sigma(\Omega)}
    \le C_\eta\left(
    \norm{G_n}_{\L^2_{t,x}\H^1_v}
    +\norm{(\partial_t+v\cdot\nabla_x)G_n}
    _{\L^2_{t,x}\H^{-1}_v}\right)
    \longrightarrow0.
  \]
\end{proof}

\subsection{Trace operators by density}

\begin{proposition}[Extension of boundary restriction]
  \label{prop:trace-completion}
  Let \(X\) be one of the graph spaces above corresponding to the velocity model $(V,\mu)$. Let \(\mathcal D\subset X\) be its stated dense
  smooth core, and let
  \[
    Y=\L^2\!\left(\partial\Omega\times V,
    W(x,v)\dd S(x)\dd\mu(v)\right)
  \]
  for one of the boundary weights for which an estimate was proved above. Suppose that
  \begin{equation}\label{eq:abstract-trace-completion-assumption}
    \norm{f|_{\partial\Omega}}_Y\le C\norm f_X,
    \qquad f\in\mathcal D.
  \end{equation}
  Then restriction extends uniquely to a bounded linear operator \(\operatorname{Tr}\colon X\to Y\), with \(\norm{\operatorname{Tr}}\le C\). If
  \(f_n\in\mathcal D\) and \(f_n\to f\) in \(X\), then
  \[
    f_n|_{\partial\Omega}\longrightarrow\operatorname{Tr}f
    \qquad\text{in }Y,
  \]
  and this limit is independent of the approximating sequence.
\end{proposition}

\begin{proof}
  Applying \eqref{eq:abstract-trace-completion-assumption} to \(f_n-f_m\) shows that the restrictions are Cauchy in \(Y\). Applying the same
  estimate to the difference of two smooth approximating sequences shows that their limits agree. Passing to the limit gives the norm bound, and
  density gives uniqueness.
\end{proof}

\begin{proposition}[\(\omega_p\)-trace operators on the half-space]
  \label{cor:flat-p}
  Let \(d\ge1\), \(\mu\in\{\mathcal L^d,\gamma_d\}\), and \(1\le p<\infty\). Then boundary restriction extends uniquely to the bounded
  \(\omega_p\)-trace operator
  \[
    \operatorname{Tr}^{(p)}_{\mu,\mathbb H}\colon
    \mathcal X_\mu(\mathbb H^d)\longrightarrow
    \L^2\!\left(\R^{d-1}\times\R^d,
    \omega_p(|w|)\dd\mu(v)\dd z\right).
  \]
  For \(p=1\), we write
  \[
    \operatorname{Tr}_{\mu,\mathbb H}
    :=\operatorname{Tr}^{(1)}_{\mu,\mathbb H}.
  \]
  This is the natural trace operator for the unrestricted Lebesgue and unrestricted Gaussian models.
\end{proposition}

\begin{proof}
  Combine the smooth estimate of Theorem~\ref{thm:flat}, the density statement in Proposition~\ref{prop:flat-intrinsic-density}, and
  Proposition~\ref{prop:trace-completion}.
\end{proof}

\begin{proposition}[\(\omega_p\)-trace operators in the unrestricted
    Gaussian model]
  \label{cor:unrestricted-gaussian-p-trace}
  Let \(d\ge2\), let \(\Omega\subset\R^d\) be a bounded \(\C^{1,1}\) domain, and let \(2\le p<\infty\). Then boundary restriction extends
  uniquely to the bounded \(\omega_p\)-trace operator
  \[
    \operatorname{Tr}^{(p)}_\gamma\colon
    \mathcal X_\gamma(\Omega)\longrightarrow
    \L^2\!\left(\partial\Omega\times\R^d,
    \omega_p(|v\cdot n_\Omega|)\dd S\dd\gamma_d\right).
  \]
\end{proposition}

\begin{proof}
  Proposition~\ref{prop:quadratic-gaussian} gives the estimate with weight \(\omega_2\) on \(\mathcal D_c(\overline\Omega)\), with a constant
  independent of the compact velocity support. For \(p\ge2\), \(\omega_p(s)\le\omega_2(s)\) for every \(s\ge0\), so the same estimate holds with
  \(\omega_p\). Proposition~\ref{prop:bounded-gaussian-density} and Proposition~\ref{prop:trace-completion} now give the unique extension.
\end{proof}

\begin{proposition}[Bounded-velocity
    \(\omega_p\)-trace operators]
  \label{cor:intrinsic-p-trace}
  Let \(d\ge2\), \(1\le p<\infty\), and \(\alpha_p\le\alpha\le1\). Let \(\Omega\subset\R^d\) be a bounded \(\C^{1,\alpha}\) domain. In the
  Euclidean statement below, let \(\mu\) be either the Lebesgue or the standard Gaussian measure.
  \begin{enumerate}
    \item For the bounded-support Euclidean model and every \(L>0\), boundary restriction extends uniquely to the bounded \(\omega_p\)-trace
          operator
          \[
            \operatorname{Tr}^{(p)}_{\mu,L}\colon
            \mathcal X_{\mu,L}(\Omega)\longrightarrow
            \L^2\!\left(\partial\Omega\times\R^d,
            \omega_p(|v\cdot n_\Omega|)\dd S\dd\mu\right),
          \]
          with \(\|\operatorname{Tr}^{(p)}_{\mu,L}\|\le C_{\mathcal A}(p,d,L)\). If \(0<L_1<L_2\), then \(\operatorname{Tr}^{(p)}_{\mu,L_2}f
          =\operatorname{Tr}^{(p)}_{\mu,L_1}f\) for every \(f\in\mathcal X_{\mu,L_1}(\Omega)\).
    \item For the spherical velocity model, boundary restriction extends uniquely to the bounded \(\omega_p\)-trace operator
          \[
            \operatorname{Tr}^{(p)}_\sigma\colon
            \mathcal X_\sigma(\Omega)\longrightarrow
            \L^2\!\left(\partial\Omega\times\Sph^{d-1},
            \omega_p(|v\cdot n_\Omega|)\dd S\dd\sigma\right),
          \]
          with \(\|\operatorname{Tr}_\sigma^{(p)}\|\le C_{\mathcal A}(p,d)\).
  \end{enumerate}
  The same conclusions hold with \(|v\cdot n_\Omega|^p\), rather than
  \(\omega_p(|v\cdot n_\Omega|)\), as the boundary weight. Here \(\mathcal A\) is a fixed quantitative boundary atlas,
  with the dependence specified in Remark~\ref{rem:quantitative-atlas}.

  In particular, in either bounded-velocity model, the natural trace operator exists on bounded \(\C^{1,1/2}\) domains.
\end{proposition}

\begin{proof}
  Apply the smooth trace estimates from Theorem~\ref{thm:bounded-positive} and Corollary~\ref{cor:spherical-positive}, and
  Proposition~\ref{prop:trace-completion} to the dense cores from Propositions~\ref{prop:bounded-support-density}
  and~\ref{prop:spherical-density}. For \(L_1<L_2\), both Euclidean \(\omega_p\)-trace operators are continuous extensions of boundary
  restriction from \(\mathcal D_{L_1}(\overline\Omega)\); uniqueness in Proposition~\ref{prop:trace-completion} proves their compatibility.
\end{proof}

For \(p=1\), we write
\[
  \operatorname{Tr}_{\mu,L}
  :=\operatorname{Tr}^{(1)}_{\mu,L},
  \qquad
  \operatorname{Tr}_\sigma
  :=\operatorname{Tr}^{(1)}_\sigma.
\]
These are the natural trace operators used in Green's formula in Proposition~\ref{prop:intrinsic-green}.

\begin{remark}[\(\omega_p\)-trace operators and the natural trace
    operator] For \(p>1\), the target spaces of the \(\omega_p\)-trace operators in the preceding two propositions have weights that are
  strictly weaker than the natural weight near grazing. Elements of these target spaces therefore need not belong to \(\L^2(|v\cdot
  n_\Omega|\dd S\dd\mu)\), so these trace operators cannot by themselves justify the signed boundary product in Green's formula. In particular,
  Proposition~\ref{cor:unrestricted-gaussian-p-trace} does not provide the natural-flux integrability required for Green's formula.
\end{remark}

\subsection{Green's formulae}

The passage from smooth functions to graph-space functions in Green's identities, using density and trace continuity, is standard and has also
been used in
\cite[Lemma~4.5]{Silvestre2022}, \cite[Lemma~3.8]{OuyangSilvestre2024}, and \cite[Lemma~5.3]{AvelinHou2026}.

We first record the half-space identity.

\begin{proposition}[Green's formula on the half-space]
  \label{prop:flat-intrinsic-green}
  Let \(d\ge1\) and \(\mu\in\{\mathcal L^d,\gamma_d\}\). For every \(f,g\in\mathcal X_\mu(\mathbb H^d)\),
  \begin{align}
     & -\int_{\R^{d-1}}\int_{\R^d}
    (\operatorname{Tr}_{\mu,\mathbb H}f)
    (\operatorname{Tr}_{\mu,\mathbb H}g)
    \,w\dd\mu(v)\dd z
    \notag                         \\
     & \quad=
    \int_{\mathbb H^d}\left(
    \pair{v\cdot\nabla_xf}{g}_{\H^{-1}_\mu,\H^1_\mu}
    +\pair{v\cdot\nabla_xg}{f}_{\H^{-1}_\mu,\H^1_\mu}
    \right)\dd x.
    \label{eq:flat-intrinsic-green}
  \end{align}
  Since \(n_{\mathbb H^d}=-e_d\), the left-hand side is the usual signed boundary integral with factor \(v\cdot n_{\mathbb H^d}\).
\end{proposition}

\begin{proof}
  By Proposition~\ref{prop:flat-intrinsic-density}, choose
  \[
    f_n,g_n\in\C_c^\infty(\overline{\mathbb H^d}\times\R^d),
    \qquad
    f_n\to f,\quad g_n\to g
    \quad\text{in the graph norm}.
  \]
  Ordinary integration by parts gives
  \begin{align*}
     & -\int_{\R^{d-1}}\int_{\R^d}
    f_n(z,0,v)g_n(z,0,v)w\dd\mu(v)\dd z \\
     & \quad=
    \int_{\mathbb H^d}\left(
    \pair{v\cdot\nabla_xf_n}{g_n}_{\H^{-1}_\mu,\H^1_\mu}
    +\pair{v\cdot\nabla_xg_n}{f_n}_{\H^{-1}_\mu,\H^1_\mu}
    \right)\dd x.
  \end{align*}
  The bulk pairings converge by \(\H^{-1}_\mu,\H^1_\mu\) duality and graph-norm convergence. On the boundary, define
  \[
    \mathcal B_{\mathbb H}(a,b)
    :=-\int_{\R^{d-1}}\int_{\R^d}
    a(z,v)b(z,v)w\dd\mu(v)\dd z.
  \]
  Cauchy--Schwarz with the natural absolute weight gives
  \[
    |\mathcal B_{\mathbb H}(a,b)|
    \le
    \norm{a}_{\L^2(|w|\dd\mu\dd z)}
    \norm{b}_{\L^2(|w|\dd\mu\dd z)}.
  \]
  Boundedness of \(\operatorname{Tr}_{\mu,\mathbb H}\) and the preceding estimate therefore give
  \[
    \mathcal B_{\mathbb H}
    (f_n(\cdot,0,\cdot),g_n(\cdot,0,\cdot))
    \longrightarrow
    \mathcal B_{\mathbb H}
    (\operatorname{Tr}_{\mu,\mathbb H}f,
    \operatorname{Tr}_{\mu,\mathbb H}g).
  \]
  Passing to the limit proves \eqref{eq:flat-intrinsic-green}.
\end{proof}

We next turn to the bounded-velocity models on bounded domains.

\begin{proposition}[Green's formula for bounded-velocity models on
    bounded domains]
  \label{prop:intrinsic-green}
  Let \(d\ge2\), let \(\Omega\subset\R^d\) be a bounded \(\C^{1,1/2}\) domain, and let \(L>0\). For the bounded-support Euclidean model, let
  \(\mu\) be either the Lebesgue or the standard Gaussian measure and let \(f,g\in\mathcal X_{\mu,L}(\Omega)\). Then
  \begin{align}
     & \int_{\partial\Omega}\int_{\R^d}
    (\operatorname{Tr}_{\mu,L}f)(\operatorname{Tr}_{\mu,L}g)
    \,v\cdot n_\Omega\dd\mu\dd S
    \notag                              \\
     & \quad=
    \int_\Omega\left(
    \pair{v\cdot\nabla_xf}{g}_{\H^{-1}_\mu,\H^1_\mu}
    +\pair{v\cdot\nabla_xg}{f}_{\H^{-1}_\mu,\H^1_\mu}
    \right)\dd x.
    \label{eq:intrinsic-green}
  \end{align}
  For the spherical velocity model, the same identity holds for \(f,g\in\mathcal X_\sigma(\Omega)\), with \(\R^d,\mu\) replaced by
  \(\Sph^{d-1},\sigma\).
\end{proposition}

\begin{proof}
  For smooth functions defined on the ambient space, the identity follows from ordinary integration by parts. The passage to graph-space
  functions is the same as in
  the proof of Proposition~\ref{prop:flat-intrinsic-green}, using Propositions~\ref{prop:bounded-support-density}
  and~\ref{prop:spherical-density} for the smooth-core approximations and Proposition~\ref{cor:intrinsic-p-trace} with \(p=1\) for continuity of
  the natural trace operators. The bulk terms pass to the limit by \(\H^{-1}\)-\(\H^1\) duality, and the boundary term by the same
  Cauchy--Schwarz argument, now with weight \(|v\cdot n_\Omega|\). This proves both the Euclidean and spherical cases.
\end{proof}

\section{The one-dimensional case}
\label{sec:dimension-one}

In one spatial dimension, the positive Euclidean statements follow from the half-space theorem by localisation at the two endpoints. The
bounded-support Lebesgue estimate can alternatively be obtained from Baouendi and Grisvard after choosing a bounded velocity interval containing
the velocity support \cite[Theorem~1 and Corollary~1]{BaouendiGrisvard1968}.

Let \(I=(0,\ell)\).

\begin{lemma}[Velocity multiplication and spatial cut-offs]
  \label{lem:interval-multiplication}
  For the standard Gaussian measure on \(\R\),
  \begin{equation}
    \norm{vh}_{\H^{-1}_\gamma}\le C\norm{h}_{\H^1_\gamma}.
    \label{eq:interval-gaussian-multiplication}
  \end{equation}
  For either the Lebesgue or the Gaussian measure, if \(\supp h\subset[-L,L]\), then
  \begin{equation}
    \norm{vh}_{\H^{-1}_\mu}
    \le\norm{vh}_{\L^2_\mu}
    \le L\norm{h}_{\L^2_\mu}.
    \label{eq:interval-bounded-multiplication}
  \end{equation}
  Consequently, for every \(\chi\in \W^{1,\infty}(I)\),
  \[
    \norm{\chi f}_{\mathcal X_\gamma(I)}
    \le C_\chi\norm{f}_{\mathcal X_\gamma(I)}.
  \]
  If \(\operatorname*{ess\,supp}_vf\subset[-L,L]\), then also
  \[
    \norm{\chi f}_{\mathcal X_{\mu,L}(I)}
    \le C_{\chi,L}\norm{f}_{\mathcal X_{\mu,L}(I)}.
  \]
\end{lemma}

\begin{proof}
  Gaussian integration by parts gives, initially for compactly supported smooth functions,
  \[
    \int_\R vh\phi\dd\gamma
    =\int_\R(h'\phi+h\phi')\dd\gamma.
  \]

  Taking the supremum over \(\norm\phi_{\H^1_\gamma}\le1\) proves \eqref{eq:interval-gaussian-multiplication}; density gives the general case.
  Estimate \eqref{eq:interval-bounded-multiplication} is immediate. Finally,

  \[
    v\partial_x(\chi f)
    =\chi\,v\partial_xf+\chi'vf,
  \]
  and the preceding estimates control the second term in \(\L^2(I;\H^{-1}_\mu)\).

\end{proof}

\begin{corollary}[Interval traces and unique continuous extensions]
  \label{cor:interval-intrinsic}
  Let \(I=(0,\ell)\). For the standard Gaussian velocity measure, there is \(C_\ell<\infty\) such that every \(f\in\C_c^\infty(\overline
  I\times\R)\) satisfies
  \begin{equation}
    \int_\R |v|
    \bigl(|f(0,v)|^2+|f(\ell,v)|^2\bigr)\dd\gamma(v)
    \le C_\ell\norm{f}_{\mathcal X_\gamma(I)}^2.
    \label{eq:interval-gaussian-trace}
  \end{equation}
  The constant is independent of the velocity support.

  For the Lebesgue velocity measure and every \(L>0\), there is \(C_{\ell,L}<\infty\) such that every
  \(f\in\C_c^\infty(\overline I\times\R)\) with
  \(\supp_vf\subset[-L,L]\) satisfies
  \begin{equation}
    \int_\R |v|
    \bigl(|f(0,v)|^2+|f(\ell,v)|^2\bigr)\dd v
    \le C_{\ell,L}
    \norm{f}_{\mathcal X_{\mathcal L^1}(I)}^2.
    \label{eq:interval-bounded-trace}
  \end{equation}

  With \(\operatorname{Tr}f=(f(0,\cdot),f(\ell,\cdot))\) on smooth functions, restriction extends uniquely to bounded maps
  \begin{align*}
    \operatorname{Tr}_\gamma\colon
    \mathcal X_\gamma(I)
     & \longrightarrow
    \L^2(\R,|v|\dd\gamma)
    \oplus\L^2(\R,|v|\dd\gamma), \\
    \operatorname{Tr}_{\mathcal L^1,L}\colon
    \mathcal X_{\mathcal L^1,L}(I)
     & \longrightarrow
    \L^2(\R,|v|\dd v)
    \oplus\L^2(\R,|v|\dd v).
  \end{align*}
  For every \(1\le p<\infty\), both estimates remain valid with \(|v|\) replaced throughout by \(\omega_p(|v|)\), and the corresponding boundary
  restrictions extend uniquely to bounded \(\omega_p\)-trace operators.
\end{corollary}

\begin{proof}
  In the Gaussian case, set \(\mu=\gamma\); in the bounded-support case, set \(\mu=\mathcal L^1\). Choose
  \(\chi\in\C^\infty([0,\ell])\) such that
  \(\chi=1\) near \(0\) and \(\supp\chi\subset[0,2\ell/3)\). Extend
  \[
    F_0(y,v):=\chi(y)f(y,v),\qquad 0<y<\ell,
  \]
  by zero to \(y\ge\ell\). Then \(F_0\in\C_c^\infty(\overline{\mathbb H^1}\times\R)\), and Theorem~\ref{thm:flat} gives
  \[
    \int_\R |v||f(0,v)|^2\dd\mu(v)
    \lesssim
    \norm{F_0}_{\L^2_y\H^1_\mu}^2
    +\norm{v\partial_yF_0}_{\L^2_y\H^{-1}_\mu}^2.
  \]
  Since
  \[
    v\partial_yF_0
    =\chi\,v\partial_xf+\chi'vf,
  \]
  Lemma~\ref{lem:interval-multiplication} bounds the right-hand side by \(C_\ell\norm{f}_{\mathcal X_\gamma(I)}^2\) in the Gaussian case and by
  \(C_{\ell,L}\norm{f}_{\mathcal X_{\mathcal L^1}(I)}^2\) in the bounded-support case.

  For the right endpoint, extend \(F_\ell(y,v):=\chi(y)f(\ell-y,v)\) by zero to \(y\ge\ell\). Here
  \[
    v\partial_yF_\ell
    =\chi'vf(\ell-y,v)
    -\chi\,(v\partial_xf)(\ell-y,v),
  \]
  so the same argument applies. Adding the endpoint estimates proves \eqref{eq:interval-gaussian-trace} and \eqref{eq:interval-bounded-trace}.
  Since \(\omega_p(s)\le s\), the corresponding smooth \(\omega_p\)-estimates follow for every \(1\le p<\infty\).

  Propositions~\ref{prop:bounded-gaussian-density} and~\ref{prop:bounded-support-density}, followed by Proposition~\ref{prop:trace-completion},
  give the stated unique continuous extensions.
\end{proof}

For the unrestricted Lebesgue model, Proposition~\ref{prop:interval-lebesgue-failure}, applied with \(d=1\) and \(\Omega=I\), shows that the
\(\omega_p\)-trace estimate---and hence any bounded \(\omega_p\)-trace operator agreeing with smooth boundary restriction---fails for every
\(1\le p<\infty\).

For \(\Sph^0=\{-1,+1\}\), one has \(\mathcal X_\sigma(I)=\H^1(I)\oplus\H^1(I)\) and \(\omega_p(1)=1\), so the ordinary one-dimensional Sobolev
trace theorem gives unique continuous endpoint trace operators for every \(1\le p<\infty\).

\appendix

\section{Density on the half-space}
\label{app:flat-density}

We use the following two elementary approximation facts. The first is also used in Proposition~\ref{prop:bounded-lebesgue-density}.

\begin{lemma}[Velocity cut-offs]\label{lem:density-velocity-cutoff}
  Let \(\mu\in\{\mathcal L^d,\gamma_d\}\), choose \(\chi\in\C_c^\infty(B_2)\) with \(\chi=1\) on \(B_1\), and set \(\chi_R(v)=\chi(v/R)\) for
  \(R\ge1\). As \(R\to\infty\), multiplication by \(\chi_R\) is uniformly bounded and converges strongly to the identity on \(\H^1_\mu\). On
  \(\H^{-1}_\mu\), define it by
  \[
    \pair{\chi_Rq}{\phi}:=\pair q{\chi_R\phi}.
  \]
  These dual operators have the same boundedness and convergence properties.
\end{lemma}

\begin{proof}
  The \(\H^1_\mu\) assertion follows from the product rule and dominated convergence. The dual operators are uniformly bounded on
  \(\H^{-1}_\mu\); they converge first on the dense subspace \(\L^2_\mu\subset\H^{-1}_\mu\), again by dominated convergence, and hence on all of
  \(\H^{-1}_\mu\).
\end{proof}

\begin{lemma}[Inward and velocity mollification]
  \label{lem:density-regularisation}
  Fix \(\mu\in\{\mathcal L^d,\gamma_d\}\) and \(M>0\). Let \(\rho,j\in\C_c^\infty(B_1)\) be nonnegative, with \(\int\rho=\int j=1\), and use the
  usual rescalings \(\rho_\delta\) and \(j_\eta\).

  \begin{enumerate}[label=\textup{(\roman*)},leftmargin=*]
    \item Let
          \[
            \mathcal O_\varphi
            =\{(x',x_d):x_d>\varphi(x')\},
            \qquad \operatorname{Lip}\varphi=K,
          \]
          and suppose that \(g\in\L^2_x\H^1_\mu\) and \(h\in\L^2_x\H^{-1}_\mu\) have compact spatial support, \(\supp_vg,\supp_vh\subset\overline
          B_M\), and \(v\cdot\nabla_xg=h\) in \(\mathcal O_\varphi\). Extend \(g,h\) by zero and, for \(\varepsilon>0\) and
          \(0<\kappa<(1+K)^{-1}\), set
          \[
            S_\varepsilon g(x)
            =\int\rho_{\kappa\varepsilon}(y)
            \widetilde g(x+\varepsilon e_d-y)\dd y,
            \qquad
            S_\varepsilon h(x)
            =\int\rho_{\kappa\varepsilon}(y)
            \widetilde h(x+\varepsilon e_d-y)\dd y.
          \]
          Then, in \(\mathcal O_\varphi\),
          \[
            v\cdot\nabla_x(S_\varepsilon g)=S_\varepsilon h,
          \]
          and \(S_\varepsilon g\to g\) in \(\L^2_x\H^1_\mu\) while \(S_\varepsilon h\to h\) in \(\L^2_x\H^{-1}_\mu\) as $\varepsilon \downarrow
            0$. The regularisations are smooth in \(x\), remain compactly supported in \(x\), and preserve velocity support.

    \item Let \(G\in\C_c^\infty(\R_x^d;\H^1_\mu)\) have velocity support in \(\overline B_M\), and let \(J_\eta G=j_\eta*_vG\), \(0<\eta\le1\).
          Then \(J_\eta G\to G\) in the whole-space graph norm as \(\eta\downarrow0\), and \(J_\eta G\) is jointly smooth and compactly
          supported,
          with \(\supp_vJ_\eta G\subset\overline B_{M+\eta}\).
  \end{enumerate}
\end{lemma}

\begin{proof}
  For (i), if \(x\in\overline{\mathcal O_\varphi}\) and \(y\in\supp\rho_{\kappa\varepsilon}\), then
  \[
    (x_d+\varepsilon-y_d)-\varphi(x'-y')
    \ge x_d-\varphi(x')
    +\bigl(1-\kappa(1+K)\bigr)\varepsilon>0.
  \]
  Thus the convolution samples only the interior of the epigraph. The same inequality keeps the translated mollification of every compactly
  supported test function inside \(\mathcal O_\varphi\). Since multiplication by \(v_i\) is bounded from the support-restricted \(\H^1_\mu\)
  space into \(\H^{-1}_\mu\), testing the equation for \(g\) gives the transport identity, with no boundary term from the zero extension. The
  convergence statements follow from the standard strong continuity of translations and the convergence of approximate identities in Bochner
  \(\L^2\).

  For (ii), the support-restricted subspace of \(\H^1_\mu\) is closed, so all \(x\)-derivatives of \(G\) retain the same velocity support. Since
  \(\norm{v_i a}_{\H^1_\mu}\le C_M\norm a_{\H^1_\mu}\) on \(\overline B_M\), it follows that \(v\cdot\nabla_xG\in\L^2_x\H^1_\mu\). On the fixed
  input and output balls \(B_M\) and \(B_{M+1}\), the Lebesgue and Gaussian Sobolev norms are equivalent. Moreover,
  \[
    v\cdot\nabla_x(J_\eta G)
    =J_\eta(v\cdot\nabla_xG)
    +\sum_{i=1}^d[v_i,J_\eta]\partial_{x_i}G,
    \qquad
    \norm{[v_i,J_\eta]a}_{\L^2_\mu}
    \le C_{M,j}\eta\norm a_{\L^2_\mu}.
  \]
  The estimate follows directly from the kernel \((v_i-w_i)j_\eta(v-w)\). Standard mollifier convergence now proves graph convergence. Finally,
  \[
    \partial_x^\alpha\partial_v^\beta(J_\eta G)
    =(\partial_v^\beta j_\eta)*_v(\partial_x^\alpha G),
  \]
  which gives joint smoothness.
\end{proof}

\begin{proof}[Proof of Proposition~\ref{prop:flat-intrinsic-density}]
  Let \(f\in\mathcal X_\mu(\mathbb H^d)\), and put \(q=v\cdot\nabla_xf\).

  \emph{Localisation.} With the cut-offs from Lemma~\ref{lem:density-velocity-cutoff},
  \[
    \chi_Rf\longrightarrow f\quad\text{in }\L^2_x\H^1_\mu,
    \qquad
    \chi_Rq\longrightarrow q\quad\text{in }\L^2_x\H^{-1}_\mu,
  \]
  and \(v\cdot\nabla_x(\chi_Rf)=\chi_Rq\). Thus we may first reduce to a common compact velocity support. If \(\zeta_K(x)=\zeta(x/K)\), with
  \(\zeta\in\C_c^\infty(\R^d)\) equal to one near the origin, then
  \[
    v\cdot\nabla_x(\zeta_Kf)
    =\zeta_Kq+(v\cdot\nabla_x\zeta_K)f.
  \]
  On a fixed velocity ball \(B_M\),
  \[
    \norm{(v\cdot\nabla_x\zeta_K)f}_{\L^2_x\H^{-1}_\mu}
    \le \frac{C_{\zeta}M}{K}\norm f_{\L^2_x\L^2_\mu}.
  \]
  Standard cut-off convergence also gives \(\zeta_Kf\to f\) in \(\L^2_x\H^1_\mu\) and \(\zeta_Kq\to q\) in \(\L^2_x\H^{-1}_\mu\). Hence a second
  approximation reduces the proof to \(f\) and \(q\) compactly supported in \(x\) and with velocity support in \(B_M\).

  \emph{Regularisation.} Apply Lemma~\ref{lem:density-regularisation}(i) with \(\varphi=0\) and \(\kappa=1/2\). It gives \(x\)-smooth, compactly
  supported pairs
  \[
    f_\varepsilon=S_\varepsilon f,\qquad
    q_\varepsilon=S_\varepsilon q,\qquad
    v\cdot\nabla_xf_\varepsilon=q_\varepsilon,
  \]
  converging to \(f\) and \(q\) in their respective Bochner norms and retaining the fixed velocity support.
  Lemma~\ref{lem:density-regularisation}(ii)
  then gives
  \[
    J_\eta f_\varepsilon\longrightarrow f_\varepsilon
    \quad\text{in }\mathcal X_\mu(\mathbb H^d),
  \]
  with \(J_\eta f_\varepsilon \in\C_c^\infty(\overline{\mathbb H^d}\times\R^d)\).

  For each \(n\), choose \(R\), then \(K\), then \(\varepsilon\), and finally \(\eta\), so that the four successive graph-norm errors are each
  less than \(1/(4n)\). The resulting diagonal sequence converges to \(f\), as required.
\end{proof}

\section{\texorpdfstring{Density in the bounded-support Euclidean
    model} {Density in the bounded-support Euclidean model}}
\label{app:density}

\begin{proof}[Proof of Proposition~\ref{prop:bounded-support-density}]
  Fix \(\mu\in\{\mathcal L^d,\gamma_d\}\), \(L>0\), and \(f\in\mathcal X_{\mu,L}(\Omega)\). Set \(q=v\cdot\nabla_xf\). Then
  \(\supp_vq\subset\overline B_L\): testing \(q\) against velocity functions supported outside \(\overline B_L\) gives zero because transport
  differentiates only in \(x\).

  \emph{Spatial regularisation.} Choose a finite smooth partition of unity \(\sum_{j=0}^N\zeta_j=1\) near \(\overline\Omega\), with the interior
  patch \(\zeta_0\) compactly supported in \(\Omega\) and every remaining patch supported compactly in a \(\C^1\) graph chart. Set
  \[
    f_j=\zeta_jf,\qquad
    q_j=\zeta_jq+(v\cdot\nabla_x\zeta_j)f.
  \]
  Then \(v\cdot\nabla_xf_j=q_j\), both \(f_j\) and \(q_j\) retain velocity support in \(\overline B_L\), and
  \[
    \norm{(v\cdot\nabla_x\zeta_j)f(x,\cdot)}_{\H^{-1}_\mu}
    \le L\norm{\nabla_x\zeta_j}_{\infty}
    \norm{f(x,\cdot)}_{\L^2_\mu}.
  \]
  Moreover, \(\sum_jq_j=q\), because \(\sum_j\nabla\zeta_j=0\).

  The interior patch is regularised by ordinary spatial convolution.
  When \(d=1\), the required rotation invariance follows directly from
  \(O(1)=\{-1,+1\}\); simultaneous reflection of \(x\) and \(v\)
  preserves the graph norm, the transport relation, and
  \(\overline B_L\).
  For a boundary patch, apply a rigid motion in \(x\) and the corresponding
  orthogonal map in \(v\). Lemma~\ref{lem:velocity-rotation} preserves the velocity norms and \(\overline B_L\), while the simultaneous change of
  variables preserves the transport relation and the graph norm. In the new coordinates write
  \[
    \Omega\cap Q=\{(x',x_d)\in Q:x_d>\varphi(x')\},
    \qquad \operatorname{Lip}\varphi=K,
  \]
  with \(\supp\zeta_j\Subset Q'\Subset Q\). Extend \(\varphi\), without changing it on the chart base, to a global \(K\)-Lipschitz function, and
  extend \(f_j,q_j\) by zero in the resulting epigraph. Because the patch is supported in \(Q'\), this creates no distributional term on an
  artificial face of the chart, so the zero extensions still satisfy \(v\cdot\nabla_xf_j=q_j\) throughout the epigraph. For
  \(0<\kappa<(1+K)^{-1}\),
  Lemma~\ref{lem:density-regularisation}(i) applies to these extensions and gives an inward, \(x\)-smooth approximation of the patch in the graph
  norm, still supported in \(\overline B_L\).

  Summing the finitely many regularised patches and rotating back yields
  \[
    F_n\in\C_c^\infty(\R_x^d;\H^1_\mu),\qquad
    \supp_vF_n\subset\overline B_L,\qquad
    F_n|_\Omega\longrightarrow f
    \quad\text{in }\mathcal X_\mu(\Omega).
  \]
  All spatial derivatives of \(F_n\) lie in \(\L^2_x\H^1_\mu\). In particular, \(v\cdot\nabla_xF_n\in\L^2_x\H^1_\mu\), since multiplication by
  \(v\) defines a bounded operator on \(\H^1_\mu\) for functions supported in the fixed ball.

  \emph{Velocity regularisation with exact support.} For \(0<\lambda<1\), put \(D_\lambda F(x,v)=F(x,v/\lambda)\). Then
  \[
    \supp_vD_\lambda F\subset\overline B_{\lambda L},
    \qquad
    v\cdot\nabla_x(D_\lambda F)
    =\lambda D_\lambda(v\cdot\nabla_xF).
  \]
  For each fixed \(n\), both \(F_n\) and \(v\cdot\nabla_xF_n\) belong to \(\L^2_x\H^1_\mu\) and have fixed compact velocity support. Strong
  continuity of dilations therefore holds in \(\H^1_\mu\); in the Gaussian case this follows from the Lebesgue statement by fixed-ball norm
  equivalence. The embedding \(\H^1_\mu\hookrightarrow\H^{-1}_\mu\) then gives the required graph-norm convergence as \(\lambda\uparrow1\).

  Choose \(\lambda_n\in(1-1/n,1)\) so that
  \[
    \norm{D_{\lambda_n}F_n-F_n}_{\L^2_x\H^1_\mu}
    +\norm{\lambda_nD_{\lambda_n}(v\cdot\nabla_xF_n)
      -v\cdot\nabla_xF_n}
    _{\L^2_x\H^{-1}_\mu}<\frac1n,
  \]
  and set \(G_n=D_{\lambda_n}F_n\). By Lemma~\ref{lem:density-regularisation}(ii), we may then choose
  \[
    0<\eta_n<\frac12(1-\lambda_n)L
  \]
  so that
  \[
    \norm{J_{\eta_n}G_n-G_n}_{\mathcal X_\mu(\R^d)}<\frac1n.
  \]
  The functions \(u_n=J_{\eta_n}G_n\) are jointly smooth and compactly supported, and
  \[
    \supp_vu_n
    \subset\overline B_{\lambda_nL+\eta_n}\Subset B_L.
  \]
  Thus \(u_n|_\Omega\in\mathcal D_L(\overline\Omega)\), while
  \[
    \norm{u_n|_\Omega-f}_{\mathcal X_\mu(\Omega)}
    \le \frac2n+
    \norm{F_n|_\Omega-f}_{\mathcal X_\mu(\Omega)}
    \longrightarrow0.
  \]
\end{proof}

\vfill
\end{document}